\documentclass[reqno]{amsart}

\usepackage{enumitem}
\usepackage{comment}
\usepackage{relsize}
\usepackage{enumitem}

\usepackage{xcolor}

\usepackage{amsmath,amssymb}
\usepackage{mathrsfs}
\usepackage{cases}
\usepackage[toc,page]{appendix}
\allowdisplaybreaks
\theoremstyle{plain}
\newtheorem{lemma}{Lemma}[section]
\newtheorem{theorem}[lemma]{Theorem}
\newtheorem{corollary}[lemma]{Corollary}

\newtheorem{proposition}[lemma]{Proposition}
\newtheorem{definition}[lemma]{Definition}

\theoremstyle{remark}
\newtheorem{remark}{Remark}

\makeatletter
\newcommand*{\rom}[1]{\expandafter\@slowromancap\romannumeral #1@}
\makeatother

\numberwithin{equation}{section}

\usepackage{mathtools}

\begin{document}

\title[The Effect of Fast Rotation and vertical viscosity on lifespan of the Primitive Equations]
{On the effect of fast rotation and vertical viscosity on the lifespan of the $3D$ primitive equations}

\date{\today}

\author[Q. Lin]{Quyuan Lin*}\thanks{*Corresponding author. Department of Mathematics, University of California,	Santa Barbara, CA 93106, USA. E-mail address: abellyn@hotmail.com}
\address[Q. Lin]
{	Department of Mathematics \\
     University of California  \\
	Santa Barbara, CA 93106, USA.} \email{abellyn@hotmail.com}

\author[X. Liu]{Xin Liu}
\address[X. Liu]
{Weierstrass-Institut
f\"ur Angewandte Analysis und Stochastik\\
    Leibniz-Institut im Forschungsverbund Berlin\\
    Mohrenstr. 39, 10117 Berlin, Germany.
Department of Applied Mathematics and Theoretical Physics\\
    University of Cambridge\\
    Cambridge CB3 0WA, UK.
}
\email{stleonliu@gmail.com} \email{stleonliu@live.com}

\author[E.S. Titi]{Edriss S. Titi}
\address[E.S. Titi]
{Department of Mathematics  \\
	Texas A\&M University  \\
	College Station  \\
	Texas, TX 77840, USA.
Department of Applied Mathematics and Theoretical Physics\\
    University of Cambridge\\
    Cambridge CB3 0WA, UK.
Department of Computer Science and Applied Mathematics \\
    Weizmann Institute of Science  \\
    Rehovot 76100, Israel.}
\email{titi@math.tamu.edu} \email{Edriss.Titi@maths.cam.ac.uk}
\email{edriss.titi@weizmann.ac.il}

\begin{abstract}
We study the effect of the fast rotation and vertical viscosity on the lifespan of solutions to the three-dimensional primitive equations (also known as the hydrostatic Navier-Stokes equations) with impermeable and stress-free boundary conditions. Firstly, for a short time interval, independent of the rate of rotation $|\Omega|$, we establish the local well-posedness of solutions with initial data that is analytic in the horizontal variables and only $L^2$ in the vertical variable. Moreover, it is shown that the solutions immediately become analytic in all the variables with increasing-in-time (at least linearly)  radius of analyticity in the vertical variable for as long as the solutions exist. On the other hand, the radius of analyticity in the horizontal variables might decrease with time, but as long as it remains positive the solution exists.  Secondly, with fast rotation, i.e., large $|\Omega|$, we show that the existence time of the solution can be prolonged, with ``well-prepared'' initial data. Finally, in the case of two spatial dimensions with $\Omega=0$, we  establish the global well-posedness provided that the initial data is small enough. The smallness condition on the initial data depends on the vertical viscosity and the initial radius of analyticity in the horizontal variables.
\end{abstract}

\maketitle



\section{Introduction}

We consider the following $3D$ viscous primitive equations (PEs) with only vertical viscosity for the large-scale oceanic and atmospheric dynamics:
\begin{subequations}\label{PE-system}
\begin{gather}
    \partial_t \mathcal{V} + \mathcal{V}\cdot \nabla  \mathcal{V} + w\partial_z \mathcal{V}  - \nu \partial_{zz} \mathcal{V} +\Omega \mathcal{V}^\perp + \nabla  p = 0 , \label{PE-1}
    \\
    \partial_z p = 0, \label{PE-2}
    \\
    \nabla  \cdot \mathcal{V} + \partial_z w =0, \label{PE-3}
\end{gather}
\end{subequations}
in the horizontal channel $\mathcal{D}:=\big\{(\boldsymbol{x},z)^\top = (x,y,z)^\top:  \boldsymbol{x}^\top \in \mathbb{T}^2,z \in (0,1) \big\}$, subject to the following initial and boundary conditions:
\begin{gather}
\mathcal{V}|_{t=0} =\mathcal{V}_0,  \label{PE-IC}\\
(\partial_z \mathcal{V}, w)|_{z=0,1}=0, \;\text{and} \; (\mathcal{V}, w)\;\text{are periodic in} \; \boldsymbol{x} \; \text{with period} \; 1. \label{PE-BC}
\end{gather}
Here the horizontal velocity field $\mathcal{V}=(u,v)^\top$, the vertical velocity $w$, and the pressure $p$ are the unknowns of the initial-boundary value problem. The $2D$ horizontal gradient is denoted by $\nabla  = (\partial_x, \partial_y)^\top$. The positive constant $\nu$ is the vertical viscosity coefficient. $ \Omega\mathcal{V}^\perp = \Omega (-v, u)^\top $ represents the Coriolis force with magnitude $ |\Omega| \in \mathbb R^+ $. As one will see later, the Coriolis force induces linear rotation waves with rotating rate $ |\Omega| $. The $3D$ viscous PEs can be derived as the asymptotic limit of the small aspect ratio between the vertical and horizontal length scales from the Boussinesq system, which is justified rigorously first in \cite{AG01} in a weak sense, then later in \cite{LT18} in a strong sense with error estimates (see also a recent paper \cite{LTY21} for the PEs with anisotropic horizontal viscosity). Notice that we have omitted the coupling with temperature in \eqref{PE-system} for the sake of simple and clear presentation. System \eqref{PE-system} is also referred to as the anisotropic vertically viscous hydrostatic Navier-Stokes equations.

The global well-posedness of strong solutions to the $3D$ PEs with full viscosity was first established in \cite{CT07}, and later in \cite{K06}. See also \cite{KZ07,KZ072} for different boundary conditions, and \cite{Hieber-Kashiwabara} for solutions with less regular initial data. In \cite{CLT16,CLT17,CLT17b}, the authors consider global well-posedness of strong solutions to the $3D$ PEs with only horizontal viscosity.

In the inviscid case without rotation ($\Omega = 0$), the linear ill-posedness of solutions in Sobolev spaces has been established in \cite{RE09}. Later on, the nonlinear ill-posedness of the inviscid PEs without rotation was established in \cite{HN16}. Moreover, without rotation, it was proved that smooth solutions to the inviscid PEs can develop singularity in finite time \cite{CINT15,W12}. It is shown later in \cite{ILT20} that these results can be extended to the case with rotation, i.e., $\Omega \neq 0$. Recently, the stability of the blowup is studied in \cite{CIL21}. Under some structural (local Rayleigh condition) or analyticity assumption of the initial data, the well-posedness theory was studied in \cite{BR99,BR03,GILT20,GR99,KMVW14,KTVZ11,MW12}. In particular, it has been shown that the lifespan of solutions to the $3D$ inviscid PEs can be prolonged provided that the rate of rotation is fast enough and the initial data is ``well-prepared'' in \cite{GILT20}. Similar results have been studied in the case of the $3D$ fast rotating Euler, Navier-Stokes, and Boussinesq equations in \cite{BMN97, BMN99a, BMN99b, BMN00, CDGG06,D05,EM96,IY,KLT14} (see also \cite{BIT11,GST15,KTZ18,LT04} for some explicit examples demonstrating the mechanism).

For the PEs with only vertical viscosity, it has been shown in \cite{RE09} that system \eqref{PE-system} is ill-posed in any Sobolev space. This ill-posedness can be overcome by considering additional linear (Rayleigh-like friction) damping, see \cite{CLT19} for the reduced $3D$ case.
On the other hand, with Gevrey regularity and some convex conditions on the initial data, the local well-posedness is established in \cite{GMV20}. When the initial data is analytic in the horizontal variables $\boldsymbol{x}$ and is sufficiently small, the global well-posedness is proved in \cite{PZZ18} in $2D$, with $\Omega = 0$ and Dirichlet boundary condition. In this paper, we consider \eqref{PE-system} in $3D$, with arbitrary $\Omega\in \mathbb{R}$ and subject to impermeable and stress-free boundary conditions.

The main results of this paper are roughly summarized as follows:

\begin{enumerate}[label = R\arabic*, ref = R\arabic*]
    \item Local well-posedness (see Theorem \ref{theorem-local}): Assume that $\mathcal{V}_0$ is analytic in the horizontal variables $\boldsymbol{x}$  and only $L^2$ in the vertical variable $z$. Let $\Omega \in \mathbb{R}$ be arbitrary but fixed. Then there exists a positive time $\mathcal T >0$, independent of $\Omega$, such that there exists a unique Leray-Hopf type weak solution $\mathcal V$ to system \eqref{PE-system} (see Definition \ref{definition-weak-solution}, below). Moreover the weak solution $\mathcal V$ depends continuously on the initial data and in particular it is unique.
    \item Instantaneous analyticity in the vertical variable (see Theorem \ref{theorem-radius}): With the same assumptions as in R1 above, the unique Leray-Hopf type weak solution $\mathcal V$ immediately becomes analytic in $z$ for $t>0$. Moreover, thanks to the viscous effect the radius of analyticity in $z$ increases in time, at least linearly, for as long as the solution exists. On the other hand, the radius of analyticity in the horizontal variables might decrease with time, but as long as it remains positive the solution exists.
    \item \label{R3} Long-time existence (see Theorem \ref{theorem-main}): Let $|\Omega|\geq |\Omega_0|$ with $|\Omega_0|$ large enough, in particular $|\Omega_0|>1$. Assume that the analytic-Sobolev norm (see \eqref{analytic-Sobolev-norm}, below) of both the barotropic mode $\overline{\mathcal{V}}_0$ and baroclinic mode $\widetilde{\mathcal{V}}_0$ are $ \mathcal O(1) $, and that some Sobolev norm of $\widetilde{\mathcal{V}}_0$ is $ \mathcal O (\frac{1}{|\Omega_0|})$, as $|\Omega_0| \rightarrow \infty$. Then a lower bound, $\mathcal T$, of the existence time of the Leray-Hopf type weak solution to system \eqref{PE-system} with $|\Omega|\geq |\Omega_0|$ satisfies
\begin{equation}\label{intro-time-1}
     \mathcal T = \mathcal O( \log[\log[\log(\log (|\Omega_0|) )]] ) \rightarrow \infty \text{ as } |\Omega_0|\rightarrow \infty.
\end{equation}
Moreover, as a corollary of R2, the solution is analytic in all variables (see Remark \ref{remark-analytic-in-z-main}, below).
\item \label{R4} Long-time existence with small barotropic mode (see Theorem \ref{theorem-small-barotropic}):
Let $|\Omega|\geq |\Omega_0| >1$ and $|\Omega_0|$ be large enough.
\begin{enumerate}
    \item Under the assumption that the solution $\overline{V}$ to the $2D$ Euler equations with initial data $\overline{\mathcal{V}}_0$ is uniformly-in-time bounded in the analytic space norm, \eqref{intro-time-1} can be improved to $\mathcal T = \mathcal O( \log(\log (|\Omega_0|) ) )$. Let us note that this result is parallel to a similar one in the inviscid case \cite{GILT20}.

    \item Moreover, under the assumption that $\overline{V}$ is uniformly-in-time small enough (the smallness condition is independent of ${|\Omega_0|}$) in the analytic space norm, the smallness requirement on the Sobolev norm of $\widetilde{V}_0$ can be relaxed and is independent of $\Omega_0$, and \eqref{intro-time-1} can be improved to $\mathcal T = \mathcal O( \log (|\Omega_0| ) )$, as $|\Omega_0| \to \infty$.
        In view of work reported in \cite{Kiselev-Sverak} about the growth of solutions of 2D Euler equations, we observe that the above assumptions about the smallness of $\overline{V}$ might not be valid for all initial data.

    \item \label{rz4-3} If the analytic norm of $\overline{\mathcal{V}}_0$ is of order $ \mathcal O ( \frac{1}{|\Omega_0|} ) $, as $|\Omega_0| \to \infty$, then the smallness requirement on the Sobolev norm of $\widetilde{V}_0$ can be relaxed and independent of $\Omega_0$; moreover,  \eqref{intro-time-1} can be improved to $\mathcal T = \mathcal O( |\Omega_0|^{\frac{1}{2}} )$.
\end{enumerate}

\item \label{rz2d}  Global well-posedness in $2D$ with $\Omega=0$ (see Theorem \ref{theorem-global}):
In the $2D$ case with $\Omega =0$, suppose that the initial data $\mathcal{V}_0$ is analytic only in the horizontal variable with small analytic-Sobolev norm (the smallness condition depends on $\nu$ and the initial radius of analyticity $\tau_0$). Then the unique Leray-Hopf type weak solution exists globally in time. Furthermore, R2 implies that the solution is analytic in all variables.
\end{enumerate}

Compared to the inviscid case \cite{GILT20}, this paper investigates the combined effect of the fast rotation and the vertical viscosity. The main differences are the following:
\begin{itemize}
    \item With analytic initial data in all variables, aside from the fast rotation, we do not observe the effect of the vertical viscosity in prolonging the lifespan in comparison to the inviscid case in \cite{GILT20}.

    \item However, with a larger class of the initial data, namely with initial data analytic in the horizontal $ x y $-variables and only Sobolev in the $z$-variable, the vertical viscosity allows us to establish the local well-posedness, which is not possible for the inviscid case (see \cite{RE09}). Moreover, the existence time is proportional to $ \nu $ and shrinks to zero as $\nu \rightarrow 0 $ (see \eqref{time-T}), which is consistent with the ill-posedness result in the inviscid case.

    \item Such a regularizing effect of the vertical viscosity can also be seen in the proof of Theorem \ref{theorem-main} in \eqref{est:301} and the absorbing argument \eqref{constraint-2}.

\end{itemize}

Compared to the work \cite{PZZ18}, which studies the $2D$ model subject to Dirichlet boundary condition without rotation, we investigate here both the $2D$ and $3D$ models subject to the impermeable and stress-free boundary conditions. While recognizing the subtle difference between the imposed boundary conditions and their mathematical and physical implications,  the result reported in \cite{PZZ18} is, roughly speaking, along the lines of the statement in R5, above, focusing on the $2D$ case. Meanwhile, our main objective in this contribution is to study the combined effect of the fast rotation and viscosity in the $3D$ case, as it has been summarized in R1 -- R4 above.

The paper is organized as follows. In section \ref{section-preliminaries}, we introduce the notations and some preliminary results which will be used throughout this paper. In section \ref{section-local}, we establish the local well-posedness of system \eqref{PE-system} and instantaneous analytic regularity in the vertical variable by proving Theorem \ref{theorem-local} (i.e., R1) and Theorem \ref{theorem-radius} (i.e., R2). In section \ref{section-limit-system}, we derive the formal limit resonant system of \eqref{PE-system} when $|\Omega|\rightarrow \infty$ and establish some properties about the limit system. Section \ref{section-rotation} is the centerpiece of this paper and is devoted to studying the effect of rotation, where we prove Theorem \ref{theorem-main} (i.e., R3) and Theorem \ref{theorem-small-barotropic} (i.e., R4). In section \ref{section-global}, we prove the global well-posedness in the $2D$ case with $\Omega=0$, i.e., Theorem \ref{theorem-global} (i.e., R5).

\section{Preliminaries}\label{section-preliminaries}
In this section, we introduce the notations and collect some preliminary results that will be used in this paper. The generic constant $C$ appearing in this paper may change from line to line. We use subscript, e.g., $C_r$, to emphasize the dependence of the constant on $r$.

\subsection{Functional settings}
We use the notation $(\boldsymbol{x},z) = (x,y,z) \in \mathcal{D}=\mathbb{T}^2\times [0,1]$, where $\boldsymbol{x}$ and $z$ represent the horizontal and vertical variables, respectively. $\mathbb{T}^2$ is the two-dimensional torus with unit length. Denote by $L^2(\mathcal{D})$, the Lebesgue space of complex/real valued functions $f(\boldsymbol{x},z)$ satisfying $\int_{\mathcal{D}} |f(\boldsymbol{x},z)|^2 d\boldsymbol{x} dz < \infty$, endowed with the norm
\begin{equation*}
    \|f \|:=\|f\|_{L^2(\mathcal{D})} = (\int_{\mathcal{D}} |f(\boldsymbol{x},z)|^2 d\boldsymbol{x} dz)^{\frac{1}{2}},
\end{equation*}
and the inner product
\begin{equation}\label{L2-inner-product}
    \langle f,g\rangle := \int_{\mathcal{D}} f(\boldsymbol{x},z)g^*(\boldsymbol{x},z) \;d\boldsymbol{x} dz
\end{equation}
for $f,g \in L^2(\mathcal{D})$. Here $ g^* $ represents the complex conjugate of $ g $. Given any time $\mathcal{T}>0$, $L^p(0,\mathcal{T};X)$ represents the space of functions $f: [0,T]\rightarrow X$ satisfying $\int_0^\mathcal{T} \|f(t)\|_X^p dt < \infty$, where $X$ is a Banach space with norm $\|\cdot\|_X$. For a function $f \in L^2(\mathcal{D})$, we use $\hat{f}_{\boldsymbol{k}}(z),
\boldsymbol{k} \in 2\pi \mathbb{Z}^2 $, to denote its Fourier coefficients in the $\boldsymbol{x}$-variables, i.e.,
\begin{eqnarray}\label{Fourier_coefficient}
&&\hskip-.1in
\hat{f}_{\boldsymbol{k}}(z) := \int_{\mathbb{T}^2} e^{- i\boldsymbol{k}\cdot \boldsymbol{x}} f(\boldsymbol{x},z) d\boldsymbol{x}, \qquad \text{and hence} \qquad
f(\boldsymbol{x},z) = \sum\limits_{\boldsymbol{k}\in 2\pi \mathbb{Z}^2} \hat{f}_{\boldsymbol{k}}(z) e^{ i\boldsymbol{k}\cdot \boldsymbol{x}}.
\end{eqnarray}

Let $A := \sqrt{-\Delta_h}$, where $\Delta_h = \partial_{xx} + \partial_{yy}$ is the horizontal Laplacian, defined by, in terms of the Fourier coefficients,
\begin{equation*}
    \widehat{Af}_{\boldsymbol{k}} (z) := |\boldsymbol{k}| \hat{f}_{\boldsymbol{k}}(z), \qquad \boldsymbol k \in 2\pi \mathbb Z^2.
\end{equation*}
For $r \geq 0$, we define
\begin{equation*}
    H^r(\mathcal{D}) : = \{ f\in L^2(\mathcal{D}):  \|f\|_{H^r} <\infty \},
\end{equation*}
with
\begin{equation*}
   \|f\|_{H^r} : = \sum\limits_{0 \leq m\leq r, m \in \mathbb Z} \big(\|A^{r-m}  \partial_z^m f\|^2 + \|\partial_z^m f\|^2\big)^{\frac{1}{2}}.
\end{equation*}
Notice that, with \eqref{Fourier_coefficient}, we have
\begin{equation*}
    \|\partial_z^m f\|^2 =\int_0^1 \Big( \sum\limits_{\boldsymbol{k}\in 2\pi  \mathbb{Z}^2}   |\partial^m_{z}\hat{f}_{\boldsymbol{k}}(z)|^2 \Big) dz \quad \text{and} \quad \|A^{r-m} \partial_z^m f\|^2 =\int_0^1 \Big( \sum\limits_{\boldsymbol{k}\in 2\pi  \mathbb{Z}^2}  |\boldsymbol{k}|^{2(r-m)}  |\partial^m_{z}\hat{f}_{\boldsymbol{k}}(z)|^2 \Big) dz.
\end{equation*}

In addition, given any $r\geq 0$ and $s \geq 0 $ with $ s \in \mathbb Z$, we define the anisotropic Sobolev space
\begin{equation*}
    H^r_{\boldsymbol{x}}H^s_z(\mathcal{D}) : = \{ f\in L^2(\mathcal{D}):  \|f\|_{H^r_{\boldsymbol{x}}H^s_z} <\infty \},
\end{equation*}
where the anisotropic Sobolev norm is given by
\begin{equation*}
    \|f\|_{H^r_{\boldsymbol{x}}H^s_z} := \sum\limits_{m\leq s}\big(\|A^r  \partial_z^m f\|^2 + \|\partial_z^mf\|^2\big)^{\frac{1}{2}}.
\end{equation*}

On the other hand,
given any $r \geq 0$, $s\geq 0$, and $\tau \geq 0$, with $ s \in \mathbb Z $, we define the analytic-Sobolev space
\begin{equation*}
    \mathcal{S}_{r,s,\tau} := \{ f\in L^2(\mathcal{D}): \|  f\|_{r,s,\tau} <\infty \},
\end{equation*}
where the norm is given by
\begin{equation}\label{analytic-Sobolev-norm}
    \| f\|_{r,s,\tau} := \sum\limits_{m\leq s}(\|A^r e^{\tau A} \partial_z^m f\|^2 + \|\partial_z^mf\|^2)^{\frac{1}{2}},
\end{equation}
with, recalling \eqref{Fourier_coefficient},
\begin{equation*}
    \|A^r e^{\tau A} \partial_z^m f\|^2 :=\int_0^1 \Big( \sum\limits_{\boldsymbol{k}\in 2\pi  \mathbb{Z}^2}  |\boldsymbol{k}|^{2r} e^{2\tau |\boldsymbol{k}|} |\partial^m_{z}\hat{f}_{\boldsymbol{k}}(z)|^2 \Big) dz.
\end{equation*}

Roughly speaking, $\mathcal{S}_{r,s,\tau}$ is the space of functions that are analytic with radius $\tau$ in the $\boldsymbol{x}$-variables, and $H^s$ in the $z$-variable. The space of analytic functions is a special case of Gevrey class. For more details about Gevrey class, we refer readers to \cite{FT98, FT89, GILT20, LO97}. Notice that when $\tau = 0$, one has $\mathcal{S}_{r,s,0} = H^r_{\boldsymbol{x}} H^s_z (\mathcal{D})$.

\begin{remark}
With abuse of notation, we also write $f\in \mathcal{S}_{r,0,\tau}$ for $ f = f(\boldsymbol{x}) $ depending only on the horizontal variables.
\end{remark}

The following lemma summarizes the algebraic property of functions with analyticity in the horizontal variables:
\begin{lemma} \label{lemma-banach-algebra}
For $\tau \geq 0$ and $r>1$, we have
\begin{eqnarray*}
\| A^r e^{\tau A} (fg)(z)\|_{L^2_{\boldsymbol{x}}} \leq C_{r}  \Big(|\hat{f}_0(z)|+\| A^r e^{\tau A} f(z)\|_{L^2_{\boldsymbol{x}}} \Big) \Big(|\hat{g}_0(z)|+ \|  A^r e^{\tau A} g(z)\|_{L^2_{\boldsymbol{x}}} \Big),
\end{eqnarray*}
provided that the right hand side is bounded, where, according to \eqref{Fourier_coefficient},
\begin{equation*}
    \hat{f}_0(z) = \int_{\mathbb{T}^2} f(\boldsymbol{x},z) d \boldsymbol{x}.
\end{equation*}
\end{lemma}

The proof of Lemma \ref{lemma-banach-algebra} is standard. We refer to \cite{FT98,GILT20,OL96} for details.


With $ \boldsymbol{k} = (k_1,k_2,k_3) \in 2\pi \bigl(\mathbb Z^2 \times ( \mathbb Z_+ \cup \lbrace 0 \rbrace ) \bigr) $
, we define
\begin{equation}\label{phik}
\phi_{\boldsymbol{k}} =  \phi_{k_1,k_2,k_3} :=
\begin{cases}
\sqrt{2}e^{ i\left( k_1 x_1 + k_2 x_2 \right)}\cos(\frac{1}{2} k_3 z) & \text{if} \;  k_3\neq 0,\\
e^{i\left( k_1 x_1 + k_2 x_2 \right)} & \text{if} \;  k_3=0,
\end{cases}
\end{equation}
and
\begin{eqnarray}\label{def:func-space}
	&&\hskip-.28in
	\mathscr{V}:=  \{ \phi \in C^\infty(\mathcal{D}) \; \Big| \; \phi= \sum\limits_{\boldsymbol{k}\in 2\pi \bigl(\mathbb Z^2 \times ( \mathbb Z_+ \cup \lbrace 0 \rbrace ) \bigr)} a_{\boldsymbol{k}} \phi_{\boldsymbol{k}}, \; a_{-k_1, -k_2,k_3}=a_{k_1,k_2,k_3}^{*}, \; \int_0^1 \nabla\cdot \phi=0 \}.
\end{eqnarray}
Here $a^*$ denotes the complex conjugate of $a$. Let
\begin{center}
    $H :=$ the closure of $\mathscr{V}$ in $L^2(\mathcal{D})$ and  $V:=$ the closure of $\mathscr{V}$ in $H^1(\mathcal{D})$,
\end{center}
with norms given by
\begin{center}
    $\|\cdot\|_H := \|\cdot\|_{L^2(\mathcal{D})}$  \text{ and }  $\|\cdot\|_V := \|\cdot\|_{H^1(\mathcal{D})}$, respectively.
\end{center}
Then one has
\begin{equation*}
    V  \subset H \equiv H'  \subset V', \quad V \hookrightarrow\hookrightarrow H \hookrightarrow\hookrightarrow V'.
\end{equation*}

\subsection{Projections and reformulation of the problem}
In this paper, we assume that $ \int_{\mathcal{D}} \mathcal{V}_0(\boldsymbol{x},z) d\boldsymbol{x} dz = 0$. This assumption is made to simplify the mathematical presentation. 
In fact,
integrating \eqref{PE-1} in $\mathcal{D}$ leads to, after applying integration by parts, \eqref{PE-3}, and \eqref{PE-BC},
\begin{eqnarray}\label{mean-zero}
\partial_t \int_{\mathcal{D}} \mathcal{V} d\boldsymbol{x}dz  + \Omega \int_{\mathcal{D}} \mathcal{V}^\perp d\boldsymbol{x}dz =0.
\end{eqnarray}
Therefore, under our assumption, one has
\begin{equation}\label{mean-zero-2}
    \int_{\mathcal{D}} \mathcal{V}(t) d\boldsymbol{x}dz = \int_{\mathcal{D}} \mathcal{V}_0(\boldsymbol{x},z) d\boldsymbol{x} dz = 0.
\end{equation}
With slight modifications, our result applies to the case when $\int_{\mathcal{D}} \mathcal{V}_0(\boldsymbol{x},z) d\boldsymbol{x} dz \neq 0$.

Let
\begin{equation*}
     \dot{L}^2 : = \Big\{\varphi\in L^2(\mathcal{D},\mathbb{R}^2)  : \int_{\mathcal{D}} \varphi(\boldsymbol{x},z)d\boldsymbol{x}dz = 0  \Big\}.
\end{equation*}
Denote the barotropic mode and the baroclinic mode of $ \mathcal {V} $ by
\begin{equation}\label{barotropic-and-baroclinic}
    \overline{\mathcal{V}}(\boldsymbol{x}):= \int_0^1 \mathcal{V}(\boldsymbol{x},z)dz \quad \text{and} \quad \widetilde{\mathcal{V}}(\boldsymbol{x},z):= \mathcal{V}-\overline{\mathcal{V}}, \qquad \text{respectively}.
    \end{equation}
From \eqref{PE-BC} and \eqref{PE-3}, we have
\begin{equation}\label{incompressible-2d}
    \nabla \cdot \overline{\mathcal{V}} = \int_0^1 \nabla \cdot \mathcal{V}(\boldsymbol{x},z)dz = -\int_0^1 \partial_z w(\boldsymbol{x},z)dz =0,
\end{equation}
and
\begin{equation}\label{PE-w}
    w(\boldsymbol{x},z) = -\int_0^z \nabla \cdot \widetilde{\mathcal{V}}(\boldsymbol{x},s)ds.
\end{equation}
\begin{remark}
In the remaining of this paper, we will substitute $ w $ by its representation \eqref{PE-w} without explicitly pointing it out.
\end{remark}
Since $\nabla \cdot \overline{\mathcal{V}} =0$, and $\overline{\mathcal{V}}$ has zero mean over $\mathbb{T}^2$ thanks to \eqref{mean-zero-2}, there exists a stream function $\psi(\boldsymbol{x})$ such that $\overline{\mathcal{V}} = \nabla ^{\perp}\psi = (-\partial_{y} \psi, \partial_{x}\psi)^\top.$
Therefore, the space of solutions to \eqref{PE-system} is given by
\begin{equation}\label{def:pe-fnt-space}
\begin{gathered}
    \mathcal{S}:= \dot{L}^2 \cap H = \Big\{\varphi\in \dot{L}^2: \nabla \cdot \overline{\varphi} = 0  \Big\} = \Big\{\varphi\in \dot{L}^2: \varphi = \nabla ^\perp \psi(\boldsymbol{x}) + \widetilde{\varphi}(\boldsymbol{x},z), \\
    \text{for some} ~ \psi, ~ \int_{\mathbb T^2 }\psi(\boldsymbol{x}) \,d \boldsymbol{x} = 0  \Big\}.
\end{gathered}
\end{equation}
Indeed, $ \mathcal S $ is the analogy of ``incompressible function space'' for the PEs. Here $ \overline\varphi $ and $ \widetilde\varphi $ are the barotropic and baroclinic modes of $ \varphi $, respectively, as in \eqref{barotropic-and-baroclinic}.

For $\varphi\in\dot{L}^2$, let the rotating operator be
$
    \mathcal{J}\varphi := \varphi^\perp = (-\varphi_2 , \varphi_1)^\top.
$
Denote the Leray projection in $ \mathbb T^2 $ by
\begin{equation}\label{leray-projection}
    \mathfrak{P}_h \overline{\varphi} := \overline{\varphi} - \nabla  \Delta_h^{-1} \nabla \cdot \overline{\varphi}.
\end{equation}
Here, $ \Delta_h^{-1} $ represents the inverse of Laplacian operator in $ \mathbb T^2 $ with zero mean value.
We define the analogy of the Leray projection for the PEs
$\mathfrak P_p : \dot{L}^2\rightarrow\mathcal{S}$ as
\begin{equation*}
    \mathfrak P_p \varphi := \widetilde{\varphi} + \mathfrak{P}_h \overline{\varphi}.
\end{equation*}
Moreover, let
$ \mathfrak R: \mathcal{S}\rightarrow\mathcal{S}$ be defined as
\begin{equation*}
    \mathfrak R \varphi:= \mathfrak P_p (\mathcal{J}\varphi) .
\end{equation*}

With notations as above, a direct computation shows that 
\begin{equation*}
    \mathfrak R \varphi = \widetilde{\varphi}^\perp \qquad \text{for} \quad \varphi \in \mathcal S.
\end{equation*}
Indeed, owing to \eqref{def:pe-fnt-space}, $ \varphi = \nabla^\perp \psi(\boldsymbol{x}) + \widetilde{\varphi} \in \mathcal S $ for some $ \psi(\boldsymbol{x}) $. Then
\begin{equation*}
    \begin{aligned}
    \mathfrak R \varphi = & \mathfrak P_p(\mathcal J \widetilde\varphi ) + \mathfrak P_p(\mathcal J \nabla^\perp \psi(\boldsymbol{x})) \\
    = & \widetilde{\varphi}^\perp - \underbrace{\mathfrak P_h \nabla \psi(\boldsymbol{x})}_{\equiv 0} = \widetilde{\varphi}^\perp .
    \end{aligned}
\end{equation*}
Therefore, the kernel of $\mathfrak R$ is given by
\begin{equation}\label{def:ker-R}
    \ker \mathfrak R =  \Big\{\varphi\in \mathcal{S}:  \widetilde{\varphi}^\perp = 0  \Big\} =  \Big\{\varphi\in \mathcal{S}:  \varphi = \overline{\varphi} \Big\}.
\end{equation}
One can define the projection $\mathfrak P_0: \mathcal{S} \rightarrow \ker \mathfrak R$ by
\begin{equation}\label{P0}
    \mathfrak P_0 \varphi := \overline{\varphi} = \int_0^1 \varphi(\boldsymbol{x},z) dz.
\end{equation}
Notice that $ \mathfrak P_0 $ can be interpreted as projection to the barotropic mode.
The fact that $ \ker \mathfrak R $ coincides with the space of functions with only the barotropic mode plays an important role in our analysis.

Furthermore, let
\begin{equation}\label{Pplus+Pminus}
    \mathfrak P_+ \varphi :=
\frac{1}{2}(\widetilde{\varphi} + i\widetilde{\varphi}^\perp),
\qquad \text{and} \qquad
    \mathfrak P_- \varphi :=
\frac{1}{2}(\widetilde{\varphi} - i\widetilde{\varphi}^\perp).
\end{equation}
Then it is easy to verify that
\begin{equation*}
\mathfrak R \mathfrak P_\pm \varphi = \mp i \mathfrak P_\pm \varphi,
\end{equation*}
i.e., $ \mathfrak P_\pm $ are the projection operators to eigenspaces of $ \mathfrak R $ with eigenvalues $ \mp i $, respectively.

Similarly to \cite{GILT20,D05,KLT14}, Lemma \ref{lemma-orthogonal-decomposition}--\ref{lemma-projection}, below, summarize projection properties of $ \mathfrak P_0, \mathfrak P_\pm $.
For the proof, we refer readers to \cite{GILT20} for details.
\begin{lemma}\label{lemma-orthogonal-decomposition}
For any $\varphi \in L^2(\mathcal{D})$, we have the following decomposition:
\begin{equation}\label{orthogonal-decomposition}
    \varphi = \mathfrak P_0 \varphi + \mathfrak P_+ \varphi + \mathfrak P_- \varphi.
\end{equation}
Moreover, we have the following properties:
\begin{equation*}
    \mathfrak P_\pm \mathfrak P_\pm \varphi = \mathfrak P_\pm \varphi, \qquad \mathfrak P_0 \mathfrak P_0 \varphi = \mathfrak P_0 \varphi,\qquad \text{and} \qquad 0 \equiv \mathfrak P_\pm \mathfrak P_\mp \varphi = \mathfrak P_0 \mathfrak P_\pm \varphi = \mathfrak P_\pm \mathfrak P_0 \varphi.
\end{equation*}
\end{lemma}

\begin{lemma} \label{lemma-projection}
	For $f,g \in L^2(\mathcal{D})$, we have
	\begin{eqnarray*}
	\langle \mathfrak P_0 f, g \rangle = \langle f, \mathfrak P_0 g\rangle = \langle \mathfrak P_0 f, \mathfrak P_0 g\rangle
	\qquad \text{and} \qquad
	\langle \mathfrak P_\pm f, g\rangle = \langle f, \mathfrak P_\pm g\rangle.
	\end{eqnarray*}
	Here the $L^2$ inner product is defined as \eqref{L2-inner-product}.
    Moreover, if $f\in \mathcal{S}_{r,s,\tau}$ with $r, s, \tau \geq 0$, $ s \in \mathbb Z $, we have
\begin{equation*}
    A^{r}e^{\tau A} \partial_z^s \mathfrak P_0 f = \mathfrak P_0 A^{r}e^{\tau A} \partial_z^s f \;\;\; \text{and}\;\;\; A^{r}e^{\tau A} \partial_z^s \mathfrak P_\pm f = \mathfrak P_\pm A^{r}e^{\tau A} \partial_z^s f.
\end{equation*}
\end{lemma}
Let $ \mathfrak I $ be the identity operator.
A direct corollary of Lemma \ref{lemma-projection} is the following:
\begin{corollary}\label{lemma-decomposition}
Consider $ r\geq 0, \tau \geq 0 $, and $ s \in \mathbb Z_+ $.
Since $\mathcal{V} = \mathfrak P_0 \mathcal{V} + (\mathfrak I-\mathfrak P_0)\mathcal{V} = \overline{\mathcal{V}} + \widetilde{\mathcal{V}}$, we have
\begin{equation*}
    \| \mathcal{V}\|^2 = \|\overline{\mathcal{V}}\|^2 + \| \widetilde{\mathcal{V}}\|^2,  \qquad \| \partial_z^s \mathcal{V}\|^2 = \|\partial_z^s \widetilde{\mathcal{V}}\|^2,
\end{equation*}
and
\begin{equation*}
  \| A^r e^{\tau A}\mathcal{V}\|^2 = \|A^r e^{\tau A}\overline{\mathcal{V}}\|^2 + \|A^r e^{\tau A} \widetilde{\mathcal{V}}\|^2,  \qquad \|A^r e^{\tau A} \partial_z^s \mathcal{V}\|^2 = \|A^r e^{\tau A}\partial_z^s \widetilde{\mathcal{V}}\|^2.
\end{equation*}
\end{corollary}
Moreover, after applying $\mathfrak P_0$ and $\mathfrak I-\mathfrak P_0$ to equation \eqref{PE-1}, thanks to \eqref{PE-BC}, \eqref{incompressible-2d}, and \eqref{PE-w}, one can derive the evolutionary equations for $\overline{\mathcal{V}}$ and $\widetilde{\mathcal{V}}$ as follows:
\begin{subequations}\label{barotropic-baroclinic}
\begin{gather}
    \partial_t \overline{\mathcal{V}} +  \overline{\mathcal{V}}\cdot \nabla  \overline{\mathcal{V}} +  \mathfrak P_0 \Big((\nabla \cdot \widetilde{\mathcal{V}}) \widetilde{\mathcal{V}} + \widetilde{\mathcal{V}}\cdot \nabla  \widetilde{\mathcal{V}} \Big) + \nabla  p = 0, \label{barotropic-evolution-1}
    \\
    \label{baroclinic-evolution-1}
    \begin{gathered}
    \partial_t \widetilde{\mathcal{V}} + \widetilde{\mathcal{V}} \cdot \nabla  \widetilde{\mathcal{V}} + \widetilde{\mathcal{V}} \cdot \nabla  \overline{\mathcal{V}} + \overline{\mathcal{V}} \cdot \nabla  \widetilde{\mathcal{V}}
  - \mathfrak P_0\Big(\widetilde{\mathcal{V}} \cdot \nabla  \widetilde{\mathcal{V}} + (\nabla  \cdot \widetilde{\mathcal{V}}) \widetilde{\mathcal{V}} \Big)
   \\
    \qquad- \Big(\int_0^z \nabla \cdot \widetilde{\mathcal{V}}(\boldsymbol{x},s)ds \Big) \partial_z \widetilde{\mathcal{V}}  + \Omega \widetilde{\mathcal{V}}^{\perp} - \nu \partial_{zz} \widetilde{\mathcal{V}} = 0.
    \end{gathered}
\end{gather}
\end{subequations}
Here, we have abused the notation by denoting $ p - \Omega \psi $ with $ \nabla^\perp \psi(\boldsymbol{x},t) = \overline{\mathcal V}(\boldsymbol{x},t) $ as $ p $, where $ \psi $ is the stream function of $ \mathcal V $ (see \eqref{def:pe-fnt-space}).

\begin{remark}
According to \eqref{def:ker-R}, \eqref{barotropic-baroclinic} can be viewed as the orthogonal decomposition of \eqref{PE-system} into $ \ker \mathfrak R $ and $ (\ker \mathfrak R)^\perp $. As $ |\Omega| \rightarrow \infty $, formal asymptotic analysis of \eqref{baroclinic-evolution-1} assures that, for well-prepared data (i.e., data ensuring that \eqref{baroclinic-evolution-1} makes sense), $ \widetilde{\mathcal V} \rightarrow 0 $ in some functional space. Therefore, in the limiting equations, \eqref{barotropic-baroclinic} converge to the $2D$ Euler equations at leading order. In particular, in \cite{GILT20}, it has been shown that the lifespan of the solutions can be prolonged with well-prepared initial data in the inviscid case.
\end{remark}

According to \eqref{Pplus+Pminus}, one has $ \widetilde {\mathcal V}^\perp = -i \mathfrak P_+ \mathcal V + i \mathfrak P_- \mathcal V $.
Therefore, after applying $\mathfrak P_\pm$ to \eqref{baroclinic-evolution-1}, we arrive at
\begin{equation}\label{baroclinic-evolution-2}
\begin{split}
    \partial_t \mathfrak P_\pm \mathcal{V} + \mathfrak P_\pm\Big(&\widetilde{\mathcal{V}} \cdot \nabla \widetilde{\mathcal{V}} + \widetilde{\mathcal{V}} \cdot \nabla \overline{\mathcal{V}} + \overline{\mathcal{V}} \cdot \nabla \widetilde{\mathcal{V}} - \mathfrak P_0(\widetilde{\mathcal{V}} \cdot \nabla \widetilde{\mathcal{V}} + (\nabla \cdot \widetilde{\mathcal{V}}) \widetilde{\mathcal{V}} )\\
    &- (\int_0^z \nabla\cdot \widetilde{\mathcal{V}}(\boldsymbol{x},s)ds ) \partial_z \widetilde{\mathcal{V}} \Big) \mp i\Omega \mathfrak P_\pm \mathcal{V} - \nu \partial_{zz} \mathfrak P_\pm \mathcal{V} = 0 .
\end{split}
\end{equation}

Let \begin{equation}\label{def:V+-}
\mathcal{V}_+ := e^{-i\Omega t}\mathfrak P_+ \mathcal{V} \qquad \text{and} \qquad \mathcal{V}_- := e^{i\Omega t} \mathfrak P_- \mathcal{V} .
\end{equation}
Then, for $ r \geq 0, \tau \geq 0, s \geq 0 $, and $ s \in \mathbb Z $, it is straightforward to check that,
\begin{equation}\label{vtilde-upm-analytic}
   \|A^r e^{\tau A} \partial_z^s \mathcal{V}_+\|^2 = \|A^r e^{\tau A} \partial_z^s \mathcal{V}_-\|^2 = \frac{1}{2}\|A^r e^{\tau A} \partial_z^s \widetilde{\mathcal{V}}\|^2.
\end{equation}

One can derive from \eqref{baroclinic-evolution-2} that
\begin{equation}\label{baroclinic-evolution-3}
\begin{split}
    \partial_t \mathcal{V}_\pm + e^{\mp i\Omega t}\mathfrak P_\pm\Big(&\widetilde{\mathcal{V}} \cdot \nabla \widetilde{\mathcal{V}} + \widetilde{\mathcal{V}} \cdot \nabla \overline{\mathcal{V}} + \overline{\mathcal{V}} \cdot \nabla \widetilde{\mathcal{V}} - \mathfrak P_0(\widetilde{\mathcal{V}} \cdot \nabla \widetilde{\mathcal{V}} + (\nabla \cdot \widetilde{\mathcal{V}}) \widetilde{\mathcal{V}} )\\
    &- (\int_0^z \nabla\cdot \widetilde{\mathcal{V}}(\boldsymbol{x},s)ds ) \partial_z \widetilde{\mathcal{V}} \Big) - \nu \partial_{zz} \mathcal{V}_\pm = 0.
\end{split}
\end{equation}
Thanks to Lemma \ref{lemma-orthogonal-decomposition} and \eqref{Pplus+Pminus}, we have
\begin{gather*}
\begin{aligned}
    \mathfrak P_+(\widetilde{\mathcal{V}} \cdot \nabla \widetilde{\mathcal{V}})
    &=\frac{1}{2} (\widetilde{\mathcal{V}} \cdot \nabla \widetilde{\mathcal{V}} + i \widetilde{\mathcal{V}} \cdot \nabla \widetilde{\mathcal{V}}^\perp) - \frac{1}{2} \mathfrak P_0\Big( \widetilde{\mathcal{V}} \cdot \nabla \widetilde{\mathcal{V}} + i \widetilde{\mathcal{V}} \cdot \nabla \widetilde{\mathcal{V}}^\perp \Big)\\
    &= \frac{1}{2} \widetilde{\mathcal{V}} \cdot \nabla (\widetilde{\mathcal{V}} + i\widetilde{\mathcal{V}}^\perp) - \frac{1}{2}\mathfrak P_0\Big( \widetilde{\mathcal{V}} \cdot \nabla (\widetilde{\mathcal{V}} + i\widetilde{\mathcal{V}}^\perp) \Big) = e^{i\Omega t}\Big(\widetilde{\mathcal{V}} \cdot \nabla \mathcal{V}_+  - \mathfrak P_0(\widetilde{\mathcal{V}} \cdot \nabla \mathcal{V}_+) \Big)  ,
\end{aligned}\\
    \mathfrak P_+(\widetilde{\mathcal{V}} \cdot \nabla \overline{\mathcal{V}}) = \frac{1}{2} (\widetilde{\mathcal{V}} \cdot \nabla \overline{\mathcal{V}} + i \widetilde{\mathcal{V}} \cdot \nabla \overline{\mathcal{V}}^\perp) = \frac{1}{2} \widetilde{\mathcal{V}} \cdot \nabla (\overline{\mathcal{V}} + i\overline{\mathcal{V}}^\perp) ,\\
    \mathfrak P_+(\overline{\mathcal{V}} \cdot \nabla \widetilde{\mathcal{V}}) = \frac{1}{2} (\overline{\mathcal{V}} \cdot \nabla \widetilde{\mathcal{V}} + i \overline{\mathcal{V}} \cdot \nabla \widetilde{\mathcal{V}}^\perp)  = e^{i\Omega t}(\overline{\mathcal{V}} \cdot \nabla \mathcal{V}_+ ),\\
    \mathfrak P_+\mathfrak P_0\Big(\widetilde{\mathcal{V}} \cdot \nabla \widetilde{\mathcal{V}} +(\nabla \cdot \widetilde{\mathcal{V}}) \widetilde{\mathcal{V}}\Big) = 0.
\end{gather*}
After applying integration by parts, one has
\begin{equation*}
\begin{split}
    \mathfrak P_+\Big((\int_0^z \nabla\cdot \widetilde{\mathcal{V}}(\boldsymbol{x},s)ds) \partial_z \widetilde{\mathcal{V}}\Big)  = &\frac{1}{2} \Big((\int_0^z \nabla\cdot \widetilde{\mathcal{V}}(\boldsymbol{x},s)ds ) \partial_z \widetilde{\mathcal{V}} + i (\int_0^z \nabla\cdot \widetilde{\mathcal{V}}(\boldsymbol{x},s)ds ) \partial_z \widetilde{\mathcal{V}}^\perp \Big) \\
    &- \frac{1}{2} \mathfrak P_0\Big((\int_0^z \nabla\cdot \widetilde{\mathcal{V}}(\boldsymbol{x},s)ds ) \partial_z \widetilde{\mathcal{V}} + i (\int_0^z \nabla\cdot \widetilde{\mathcal{V}}(\boldsymbol{x},s)ds ) \partial_z \widetilde{\mathcal{V}}^\perp \Big)\\
     = &e^{i\Omega t} (\int_0^z \nabla\cdot \widetilde{\mathcal{V}}(\boldsymbol{x},s)ds )  \partial_z \mathcal{V}_+ + e^{i\Omega t} \mathfrak P_0 \Big( (\nabla\cdot\widetilde{\mathcal{V}}) \mathcal{V}_+  \Big).
\end{split}
\end{equation*}
Moreover, thanks to \eqref{Pplus+Pminus} and \eqref{def:V+-}, $ \widetilde{\mathcal V} = \mathcal V_+ e^{i\Omega t} + \mathcal V_-e^{-i\Omega t} $.
Therefore, the $\mathcal{V}_+$ part of \eqref{baroclinic-evolution-3} can be written as
\begin{equation}\label{up3}
\begin{aligned}
 \partial_t \mathcal{V}_+ = & -e^{i\Omega t} \Big(\mathcal{V}_+ \cdot \nabla \mathcal{V}_+ - \mathfrak P_0( \mathcal{V}_+ \cdot \nabla \mathcal{V}_+ + (\nabla \cdot \mathcal{V}_+) \mathcal{V}_+) - (\int_0^z \nabla\cdot \mathcal{V}_+(\boldsymbol{x},s)ds ) \partial_z \mathcal{V}_+ \Big)  \\
 & -  \Big(\overline{\mathcal{V}} \cdot \nabla \mathcal{V}_+  + \frac{1}{2}(\mathcal{V}_+ \cdot \nabla)(\overline{\mathcal{V}}+i\overline{\mathcal{V}}^\perp) \Big) + \nu \partial_{zz} \mathcal{V}_+ - e^{-2i\Omega t} \frac{1}{2} (\mathcal{V}_- \cdot \nabla)(\overline{\mathcal{V}}+i\overline{\mathcal{V}}^\perp)  \\
& - e^{-i\Omega t} \Big(\mathcal{V}_- \cdot \nabla \mathcal{V}_+ - \mathfrak P_0( \mathcal{V}_- \cdot \nabla \mathcal{V}_+ + (\nabla \cdot \mathcal{V}_-) \mathcal{V}_+) - (\int_0^z \nabla\cdot \mathcal{V}_-(\boldsymbol{x},s)ds ) \partial_z \mathcal{V}_+ \Big).
\end{aligned}
\end{equation}
Similarly, the $\mathcal{V}_-$ part of \eqref{baroclinic-evolution-3} can be written as
\begin{equation}\label{um1}
\begin{aligned}
\partial_t \mathcal{V}_- = & -e^{-i\Omega t} \Big(\mathcal{V}_- \cdot \nabla \mathcal{V}_- - \mathfrak P_0( \mathcal{V}_- \cdot \nabla \mathcal{V}_- + (\nabla \cdot \mathcal{V}_-) \mathcal{V}_-) - (\int_0^z \nabla\cdot \mathcal{V}_-(\boldsymbol{x},s)ds ) \partial_z \mathcal{V}_- \Big)  \\
& -  \Big(\overline{\mathcal{V}} \cdot \nabla \mathcal{V}_-  + \frac{1}{2}(\mathcal{V}_- \cdot \nabla)(\overline{\mathcal{V}}-i\overline{\mathcal{V}}^\perp) \Big) + \nu \partial_{zz} \mathcal{V}_- - e^{2i\Omega t} \frac{1}{2} (\mathcal{V}_+ \cdot \nabla)(\overline{\mathcal{V}}-i\overline{\mathcal{V}}^\perp)  \\
& - e^{i\Omega t} \Big(\mathcal{V}_+ \cdot \nabla \mathcal{V}_- - \mathfrak P_0( \mathcal{V}_+ \cdot \nabla \mathcal{V}_- + (\nabla \cdot \mathcal{V}_+) \mathcal{V}_-) - (\int_0^z \nabla\cdot \mathcal{V}_+(\boldsymbol{x},s)ds ) \partial_z \mathcal{V}_- \Big).
\end{aligned}
\end{equation}

In addition,
\eqref{barotropic-evolution-1} can be written as
\begin{equation*}\label{barotropic-evolution-3}
\begin{gathered}
 \partial_t \overline{\mathcal{V}} + \overline{\mathcal{V}}\cdot \nabla \overline{\mathcal{V}} + e^{2i\Omega t} \mathfrak P_0\Big(\mathcal{V}_+ \cdot \nabla \mathcal{V}_+  + (\nabla \cdot \mathcal{V}_+) \mathcal{V}_+\Big)  + e^{-2i\Omega t} \mathfrak P_0\Big(\mathcal{V}_- \cdot \nabla \mathcal{V}_-  + (\nabla \cdot \mathcal{V}_-) \mathcal{V}_-\Big)  \\
 + \nabla p + \mathfrak  P_0 \Big(\mathcal{V}_+ \cdot \nabla \mathcal{V}_-  + \mathcal{V}_- \cdot \nabla \mathcal{V}_+ + (\nabla \cdot \mathcal{V}_+) \mathcal{V}_- + (\nabla \cdot \mathcal{V}_-) \mathcal{V}_+  \Big) = 0.
\end{gathered}
\end{equation*}
Recalling \eqref{Pplus+Pminus} and \eqref{def:V+-}, i.e., $\mathcal{V}_\pm = e^{\mp i\Omega t} \mathfrak P_\pm \mathcal{V} = \frac{1}{2} e^{\mp i\Omega t} (\widetilde{\mathcal{V}} \pm i \widetilde{\mathcal{V}}^\perp)$,
we rewrite the last term of the above equation as
\begin{eqnarray*}
&&\hskip-.8in
\mathfrak P_0 \Big(\mathcal{V}_+ \cdot \nabla \mathcal{V}_-  + \mathcal{V}_- \cdot \nabla \mathcal{V}_+ + (\nabla \cdot \mathcal{V}_+) \mathcal{V}_- + (\nabla \cdot \mathcal{V}_-) \mathcal{V}_+ \Big)  \\
&&\hskip-.9in
= \frac{1}{2} \mathfrak P_0  \Big(\widetilde{\mathcal{V}} \cdot \nabla \widetilde{\mathcal{V}} + \widetilde{\mathcal{V}}^\perp \cdot \nabla \widetilde{\mathcal{V}}^\perp + (\nabla \cdot \widetilde{\mathcal{V}}) \widetilde{\mathcal{V}} + (\nabla \cdot \widetilde{\mathcal{V}}^\perp) \widetilde{\mathcal{V}}^\perp\Big) = \frac{1}{2}  \mathfrak P_0  (\nabla |\widetilde{\mathcal{V}}|^2) = \nabla (\frac{1}{2} \mathfrak P_0 |\widetilde{\mathcal{V}}|^2),
\end{eqnarray*}
which can be combined with $ \nabla p $. Therefore, with abuse of notation, one can rewrite \eqref{barotropic-evolution-1} as
\begin{equation}\label{barotropic-evolution-4}
\begin{aligned}
    \partial_t \overline{\mathcal{V}} + (\overline{\mathcal{V}}\cdot \nabla \overline{\mathcal{V}}) + \nabla p &+  e^{2i\Omega t} \mathfrak P_0 \Big(\mathcal{V}_+ \cdot \nabla \mathcal{V}_+  + (\nabla \cdot \mathcal{V}_+) \mathcal{V}_+ \Big)
    \\
    +&e^{-2i\Omega t} \mathfrak P_0\Big(\mathcal{V}_- \cdot \nabla \mathcal{V}_-  + (\nabla \cdot \mathcal{V}_-) \mathcal{V}_-\Big) = 0.
\end{aligned}
\end{equation}

\section{Local Well-posedness}\label{section-local}
In sections \ref{subsec:local-a-priori} and \ref{subsec:well-posedness}, below, we will establish the local well-posedness, i.e., the existence, the uniqueness, and the continuous dependency on initial data, of weak solutions to system \eqref{PE-system}, defined as below:

\begin{definition}\label{definition-weak-solution}
Let $\mathcal{T}>0$, $r>2$, $\tau_0>0$, and suppose that the initial data $\mathcal{V}_0\in \mathcal{S}_{r,0,\tau_0}\cap H$.
We say $\mathcal{V}$ is a Leray-Hopf type weak solution to system \eqref{PE-system} with initial and boundary conditions \eqref{PE-IC}--\eqref{PE-BC}
if
\begin{enumerate}[label=\arabic*)]
\item there exists $ \tau(t) > 0 $, for $ t \in [0,\mathcal T]$, such that
$$\mathcal{V}\in L^\infty\big(0,\mathcal{T}; \mathcal{S}_{r,0,\tau(t)}
\big) \cap L^2\big(0,\mathcal{T}; V\cap \mathcal{S}_{r,1,\tau(t)}\cap \mathcal{S}_{r+\frac{1}{2},0,\tau(t)}\big), $$
$$\partial_t \mathcal V , A^{r-\frac{1}{2}} e^{\tau A} \partial_t \mathcal{V} \in L^2\big(0,\mathcal{T}; V'\big), $$
\item system \eqref{PE-system} is satisfied in the distribution sense,
\item and moreover, the following energy inequality holds:
\begin{gather*}
    \| \mathcal{V}(t)\|_{r,0,\tau(t)}^2 + 2 \int_0^t \Big(  \nu \|\partial_z \mathcal{V}(s)\|_{r,0,\tau(s)}^2   + \|A^{r+\frac{1}{2}} e^{\tau(s) A}  \mathcal{V}(s)\|^2 \Big) ds\leq \| \mathcal{V}_0\|_{r,0,\tau_0}^2.
 \end{gather*}

\end{enumerate}
\end{definition}
The following theorem is the main result in this section.

\begin{theorem}\label{theorem-local}
Assume $\mathcal{V}_0\in \mathcal{S}_{r,0,\tau_0}\cap H$ with $r>2$ and $\tau_0>0$. Let $\Omega \in \mathbb{R}$ be arbitrary and fixed. Then there exist a positive time $\mathcal T>0$ and a positive function $\tau(t) >0$ given in \eqref{time-T} and \eqref{tau-inequality}, below, respectively,
such that $ \mathcal V $ is a Leray-Hopf type weak solution, as in Definition \ref{definition-weak-solution}, to system  \eqref{PE-system} with \eqref{PE-IC} and \eqref{PE-BC}  in $ [0,\mathcal{T}]$. In particular, $\tau(t) $ and $ \mathcal T $ are independent of $ \Omega $. Moreover, $\mathcal{V}$ is unique and depends continuously on the initial data, in the sense of \eqref{continuous-dependence}, below.
\end{theorem}

Notice that we do not need to assume \eqref{mean-zero-2} in Theorem \ref{theorem-local}. 
Throughout the rest of this section, we assume that $ (\mathcal V, p) $ satisfies \eqref{PE-system}--\eqref{PE-BC} and is smooth enough such that the following calculation makes sense. The rigid justification can be established through Galerkin approximation arguments (see, e.g., \cite{GILT20,LT10}). In particular, in section \ref{subsec:local-a-priori}, we establish the {\it a priori} estimates of solutions to system \eqref{PE-system} with \eqref{PE-BC}.
In section \ref{subsec:well-posedness}, we finish the proof of Theorem \ref{theorem-local} by establishing the uniqueness and continuous dependency on initial data. In section \ref{section-radius}, we show that the weak solution immediately becomes analytic in $z$, and the radius of analyticity in $z$ increases as long as the solution exists.

\subsection{\textit{A Priori} Estimates}\label{subsec:local-a-priori}
Direct calculation of $ \langle \eqref{PE-1}, \mathcal V \rangle + \langle A^r e^{\tau A} \eqref{PE-1},A^r e^{\tau A} \mathcal V\rangle $, after applying integration by parts, \eqref{PE-3}, and \eqref{PE-BC}, shows that
\begin{equation}\label{apri-001}
\begin{split}
    \frac{1}{2}\frac{d}{dt} \| \mathcal{V}\|_{r,0,\tau}^2 + &\nu  \|\partial_z \mathcal{V}\|_{r,0,\tau}^2 - \dot{\tau} \|A^{r+\frac{1}{2}} e^{\tau A}  \mathcal{V}\|^2 =
    -\Big\langle A^r e^{\tau A}(\mathcal{V}\cdot \nabla \mathcal{V}), A^r e^{\tau A}  \mathcal{V}\Big\rangle\\
    +&\Big\langle A^r e^{\tau A} \Big[ \Big(\int_0^z \nabla\cdot \mathcal{V}(\boldsymbol{x},s)ds \Big) \partial_z \mathcal{V} \Big], A^r e^{\tau A}  \mathcal{V} \Big\rangle =: I_1 + I_2.
\end{split}
\end{equation}
By virtue of Lemma \ref{lemma-type1}, the Sobolev inequality, and the H\"older inequality, we have
\begin{equation*}
\begin{split}
    |I_1| &\leq \Big|\Big\langle A^r e^{\tau A} ( \mathcal{V}\cdot \nabla \mathcal{V}), A^r e^{\tau A}  \mathcal{V} \Big\rangle\Big| \\
    &\leq \int_0^1 C_r \Big(\|A^r e^{\tau A} \mathcal{V}(z)\|_{L^2(\mathbb{T}^2)} + \| \mathcal{V}(z)\|_{L^2(\mathbb{T}^2)} \Big) \|A^{r+\frac{1}{2}} e^{\tau A} \mathcal{V}(z)\|_{L^2(\mathbb{T}^2)}^2  dz\\
    &\leq C_r (\| \mathcal{V}\|_{r,0,\tau} + \| \partial_z \mathcal{V}\|_{r,0,\tau}) \|A^{r+\frac{1}{2}} e^{\tau A}  \mathcal{V}\|^2.
\end{split}
\end{equation*}
Applying Lemma \ref{lemma-type2} to $ I_2 $ leads to
\begin{equation*}
    |I_2| \leq C_r \|\partial_z \mathcal{V}\|_{r,0,\tau} \|A^{r+\frac{1}{2}} e^{\tau A}  \mathcal{V}\|^2.
\end{equation*}
Thus from \eqref{apri-001}, one has
\begin{equation}\label{apri-002}
\begin{gathered}
     \frac{1}{2}\frac{d}{dt} \| \mathcal{V}\|_{r,0,\tau}^2 + \nu  \|\partial_z \mathcal{V}\|_{r,0,\tau}^2 + \|A^{r+\frac{1}{2}} e^{\tau A}  \mathcal{V}\|^2 \leq \Big(\dot{\tau}+1+C_r (\| \mathcal{V}\|_{r,0,\tau} + \| \partial_z \mathcal{V}\|_{r,0,\tau})\Big) \\
     \qquad  \times \|A^{r+\frac{1}{2}} e^{\tau A}  \mathcal{V}\|^2
        \leq \Big(\dot{\tau}+C_r (1+\| \mathcal{V}\|^2_{r,0,\tau} + \| \partial_z \mathcal{V}\|^2_{r,0,\tau})\Big) \|A^{r+\frac{1}{2}} e^{\tau A}  \mathcal{V}\|^2  .
\end{gathered}
\end{equation}
Choose $ \tau $ such that
\begin{equation}\label{def:tau-local}
\dot{\tau}+1+C_r (\| \mathcal{V}\|_{r,0,\tau} + \| \partial_z \mathcal{V}\|_{r,0,\tau}) = 0 .
\end{equation}
Then, one has
\begin{equation*}
    \frac{1}{2}\frac{d}{dt} \| \mathcal{V}\|_{r,0,\tau}^2 + \nu  \|\partial_z \mathcal{V}\|_{r,0,\tau}^2 + \|A^{r+\frac{1}{2}} e^{\tau A}  \mathcal{V}\|^2 \leq 0.
\end{equation*}
For $ \mathcal T > 0 $, to be determined, and $t \in [0, \mathcal T ] $, one has, after integrating \eqref{apri-002} in the $ t $-variable,
\begin{equation}\label{apri-003}
    \| \mathcal{V}(t)\|_{r,0,\tau(t)}^2 + 2 \int_0^t \Big(  \nu \|\partial_z \mathcal{V}(s)\|_{r,0,\tau(s)}^2   + \|A^{r+\frac{1}{2}} e^{\tau(s) A}  \mathcal{V}(s)\|^2 \Big) ds\leq \| \mathcal{V}_0\|_{r,0,\tau_0}^2.
\end{equation}
On the other hand, integrating \eqref{def:tau-local} yields
\begin{equation}\label{tau-inequality}
\begin{aligned}
    \tau(t) =& \tau_0 - t - C_r \int_0^t \bigl( \| \mathcal{V}(s)\|_{r,0,\tau(s)} + \| \partial_z \mathcal{V}(s)\|_{r,0,\tau(s)} \bigr) ds\\
    \geq & \tau_0 - (1+C_r \| \mathcal{V}_0\|_{r,0,\tau_0}) t - \frac{C_r}{\sqrt{2\nu}} \| \mathcal{V}_0\|_{r,0,\tau_0} \sqrt{t}.
\end{aligned}
\end{equation}
Consider, for $ C_r > 0 $ as in \eqref{tau-inequality}, that
\begin{equation}\label{time-T}
    \mathcal{T} := \Big(\frac{\sqrt{\frac{C_r^2 \| \mathcal{V}_0\|_{r,0,\tau_0}^2}{2\nu} + 2\tau_0(1+C_r \| \mathcal{V}_0\|_{r,0,\tau_0})} - \frac{C_r \| \mathcal{V}_0\|_{r,0,\tau_0}}{\sqrt{2\nu}}}{2(1+C_r \| \mathcal{V}_0\|_{r,0,\tau_0})}\Big)^2 > 0,
\end{equation}
which solves
\begin{equation*}
   (1+C_r \| \mathcal{V}_0\|_{r,0,\tau_0}) \mathcal{T} - \frac{C_r}{\sqrt{2\nu}} \| \mathcal{V}_0\|_{r,0,\tau_0} \sqrt{\mathcal{T}} = \frac{\tau_0}{2}.
\end{equation*}
Then one has
\begin{equation*}
    \tau(t) \geq \tau_0/2 > 0 \quad \text{for} \quad t \in [0,\mathcal T].
\end{equation*}
Consequently, \eqref{apri-003} implies that
\begin{equation}\label{regularity-v}
    \mathcal{V}\in L^\infty\big(0,\mathcal{T};  \mathcal{S}_{r,0,\tau(t)}\big) \cap L^2\big(0,\mathcal{T}; V\cap \mathcal{S}_{r,1,\tau(t)}\cap \mathcal{S}_{r+\frac{1}{2},0,\tau(t)}\big)
\end{equation}
with $ \mathcal T > 0 $ given as in \eqref{time-T} and $ \tau(t) $ given as in \eqref{tau-inequality} (or equivalently \eqref{def:tau-local}).

Next, in order to obtain the estimate of $\partial_t \mathcal{V}$, testing \eqref{PE-1} with $ \forall \phi \in \mathscr{V} $ (see \eqref{def:func-space}) leads to
\begin{eqnarray}\label{apri-004}
&&\hskip-.8in
\Big\langle \partial_t \mathcal{V}, \phi \Big\rangle +  \Big\langle \mathcal{V} \cdot \nabla \mathcal{V} - \Big(\int_0^z \nabla\cdot \mathcal{V}(\boldsymbol{x},s)ds \Big) \partial_z \mathcal{V}  +  \Omega \mathcal{V}^\perp - \nu \partial_{zz}\mathcal{V},  \phi   \Big \rangle  = 0 .
\end{eqnarray}
where we have substituted, thanks to \eqref{PE-2} and \eqref{def:func-space}, $ \langle \nabla p , \phi \rangle = - \langle  p , \nabla \cdot \phi \rangle = 0$.
Since $r>2$, thanks to the H\"older inequality and the Sobolev inequality, we obtain that
\begin{equation*}
   \Big| \Big\langle \mathcal{V} \cdot \nabla \mathcal{V} ,  \phi \Big \rangle \Big| \leq   C \|\mathcal{V}\|_{L^\infty_{\boldsymbol{x}}L^2_z} \|\nabla\mathcal{V}\|_{L^2_{\boldsymbol{x}}L^2_z} \| \phi\|_{L^2_{\boldsymbol{x}}L^\infty_z} \leq C_r \|\mathcal{V}\|_{r,0,\tau}^2 \|\phi\|_{V}
\end{equation*}
and
\begin{equation*}
    \begin{split}
        &\Big| \Big\langle \Big(\int_0^z \nabla\cdot \mathcal{V}(\boldsymbol{x},s)ds \Big) \partial_z \mathcal{V} ,  \phi   \Big \rangle \Big| \leq \int_{\mathbb{T}^2} \Big(\int_0^1 |\nabla\cdot \mathcal{V}| dz\Big) \Big(\int_0^1 |\partial_z \mathcal{V}||\phi| dz\Big) d\boldsymbol{x} \\
         \leq C & \int_{\mathbb{T}^2} \|\nabla \mathcal{V}\|_{L^2_z} \|\partial_z \mathcal{V}\|_{L^2_z} \|\phi\|_{L^2_z} d\boldsymbol{x} \leq C  \|\nabla \mathcal{V}\|_{L^2_{\boldsymbol{x}}L^2_z}\|\partial_z \mathcal{V}\|_{L^4_{\boldsymbol{x}}L^2_z}\|\phi\|_{L^4_{\boldsymbol{x}}L^2_z} \leq  C_r \|\mathcal{V}\|_{r,1,\tau} \|\mathcal{V}\|_{r,0,\tau} \|\phi\|_{V}.
    \end{split}
\end{equation*}
After applying integration by parts, one has
\begin{equation*}
    \Big|\Big\langle   \Omega \mathcal{V}^\perp - \nu \partial_{zz}\mathcal{V},  \phi   \Big \rangle\Big|
    = \Big|\Big\langle \Omega \mathcal{V}^\perp,  \phi \Big\rangle + \nu \Big\langle \partial_z \mathcal V, \partial_z \phi \Big\rangle   \Big|
    \leq C_{\nu,  \Omega} \|\mathcal{V}\|_{r,1,\tau}\|\phi\|_{V}.
\end{equation*}
Therefore, one has
\begin{eqnarray*}
&&\hskip-.8in
\Big|\Big\langle \partial_t \mathcal{V}, \phi \Big\rangle \Big| \leq C_{\nu,r,\Omega} \Big(\| \mathcal{V}\|_{r,0,\tau}^2 + (1+\| \mathcal{V}\|_{r,0,\tau})\| \mathcal{V}\|_{r,1,\tau} \Big) \|\phi\|_{V}.
\end{eqnarray*}
Since $\mathscr{V}$ is dense in $V$, one has
\begin{equation*}
    \| \partial_t \mathcal{V}\|_{V'} = \sup\limits_{\|\phi\|_{V}=1} \Big|\Big\langle \partial_t \mathcal{V}, \phi \Big\rangle \Big| \leq C_{\nu,r,\Omega} \Big(\| \mathcal{V}\|_{r,0,\tau}^2 + (1+\| \mathcal{V}\|_{r,0,\tau})\| \mathcal{V}\|_{r,1,\tau} \Big) .
\end{equation*}
Thanks to \eqref{regularity-v}, we have
\begin{equation}\label{regularity-vt-1}
\begin{split}
&\partial_t \mathcal{V}\in L^2(0,\mathcal{T}; V')  \qquad \text{and}
\\
   & \|\partial_t \mathcal V \|_{L^2(0,\mathcal{T}; V')}  \leq C_{\nu,r,\Omega} \Big(C_{\mathcal T}\|  \mathcal{V}\|_{L^\infty(0,\mathcal T; \mathcal S_{r,0,\tau})}^2 + (1+\| \mathcal{V}\|_{L^\infty(0,\mathcal T; \mathcal S_{r,0,\tau})})\| \mathcal{V}\|_{L^2(0,\mathcal T; \mathcal S_{r,1,\tau})} \Big) < \infty.
\end{split}
\end{equation}

Meanwhile,
for $A^{r-\frac{1}{2}} e^{\tau A}\partial_t \mathcal{V}$, one has, similarly as in \eqref{apri-004},
\begin{equation*}
\begin{split}
    \Big\langle A^{r-\frac{1}{2}} e^{\tau A}\partial_t \mathcal{V}, \phi \Big\rangle +  \Big\langle & A^{r-\frac{1}{2}} e^{\tau A}\big(\mathcal{V} \cdot \nabla \mathcal{V}\big) - A^{r-\frac{1}{2}} e^{\tau A} \Big(\big(\int_0^z \nabla\cdot \mathcal{V}(\boldsymbol{x},s)ds \big) \partial_z \mathcal{V}\Big) \\
    & +  \Omega A^{r-\frac{1}{2}} e^{\tau A}\mathcal{V}^\perp - \nu \partial_{zz}A^{r-\frac{1}{2}} e^{\tau A}\mathcal{V} ,  \phi   \Big \rangle  = 0 .
\end{split}
\end{equation*}
With $r>2$, thanks to Lemma \ref{lemma-banach-algebra},  the H\"older inequality, and the Sobolev inequality, we obtain that
\begin{equation*}
    \begin{split}
        \Big| \Big\langle  A^{r-\frac{1}{2}} e^{\tau A}\big(\mathcal{V} \cdot \nabla \mathcal{V}\big),  \phi   \Big \rangle \Big| &\leq \|A^{r-\frac{1}{2}} e^{\tau A}\mathcal{V} \cdot \nabla \mathcal{V}\|_{L^2_{\boldsymbol{x}}L^1_z} \|\phi\|_{L^2_{\boldsymbol{x}} L^\infty_z} \\
        &
        \leq C_r \|\mathcal{V}\|_{r+\frac{1}{2},0,\tau} \|\mathcal{V}\|_{r,0,\tau} \|\phi\|_{V}.
    \end{split}
\end{equation*}
After applying integration by parts in the $ z $-variable and the H\"older inequality, one has
\begin{equation*}
    \begin{split}
       &\Big| \Big\langle  A^{r-\frac{1}{2}} e^{\tau A} \Big(\big(\int_0^z \nabla\cdot \mathcal{V}(\boldsymbol{x},s)ds \big) \partial_z \mathcal{V}\Big) ,  \phi   \Big \rangle   \Big|
       \leq  \Big| \Big\langle A^{r-\frac{1}{2}} e^{\tau A} \Big(  (\nabla\cdot \mathcal{V})  \mathcal{V} \Big),   \phi   \Big \rangle   \Big|\\
    &\qquad\qquad + \Big| \Big\langle A^{r-\frac{1}{2}} e^{\tau A} \Big( \big(\int_0^z \nabla\cdot \mathcal{V}(\boldsymbol{x},s)ds \big)  \mathcal{V} \Big) ,   \partial_z \phi   \Big \rangle   \Big|
        \leq  C_r \|\mathcal{V}\|_{r+\frac{1}{2},0,\tau} \|\mathcal{V}\|_{r,0,\tau} \|\phi\|_{V},
    \end{split}
\end{equation*}
and similarly,
\begin{equation*}
    \begin{split}
       \Big| \Big\langle  \Omega A^{r-\frac{1}{2}} e^{\tau A}\mathcal{V}^\perp - \nu \partial_{zz}A^{r-\frac{1}{2}} e^{\tau A}\mathcal{V} ,  \phi   \Big \rangle \Big| \leq C_{\nu,  \Omega} \|\mathcal{V}\|_{r,1,\tau}\|\phi\|_{V}.
    \end{split}
\end{equation*}
Therefore, one has
\begin{eqnarray*}
&&\hskip-.8in
\Big|\Big\langle A^{r-\frac{1}{2}} e^{\tau A}\partial_t \mathcal{V}, \phi \Big\rangle \Big| \leq C_{\nu,r,\Omega} \Big( \|\mathcal{V}\|_{r+\frac{1}{2},0,\tau} \|\mathcal{V}\|_{r,0,\tau} + \|\mathcal{V}\|_{r,1,\tau}   \Big) \|\phi\|_{V}.
\end{eqnarray*}
Since $\mathscr{V}$ is dense in $V$, one has
\begin{equation*}
    \|A^{r-\frac{1}{2}} e^{\tau A}\partial_t \mathcal{V}\|_{V'} = \sup\limits_{\|\phi\|_{V}=1} \Big|\Big\langle A^{r-\frac{1}{2}} e^{\tau A}\partial_t \mathcal{V}, \phi \Big\rangle \Big| \leq C_{\nu,r,\Omega} \Big( \|\mathcal{V}\|_{r+\frac{1}{2},0,\tau} \|\mathcal{V}\|_{r,0,\tau} + \|\mathcal{V}\|_{r,1,\tau}   \Big).
\end{equation*}
Thanks to \eqref{regularity-v}, we have
\begin{equation}\label{regularity-vt-2}
    \begin{gathered}
    A^{r-\frac{1}{2}} e^{\tau A} \partial_t \mathcal{V}\in L^2\big(0,\mathcal{T}; V'\big)
    \qquad \text{and}\qquad \\
    \|A^{r-\frac{1}{2}} e^{\tau A} \partial_t \mathcal{V}\|_{L^2\big(0,\mathcal{T}; V'\big)}
    \leq C_{\nu,r,\Omega} \Big( \|\mathcal{V}\|_{L^\infty(0,\mathcal T; \mathcal S_{r,0,\tau})} \|\mathcal{V}\|_{L^2(0,\mathcal T; \mathcal S_{r+\frac{1}{2},0,\tau})}  + \|\mathcal{V}\|_{L^2(0,\mathcal T; \mathcal S_{r,1,\tau})}   \Big)< \infty.
    \end{gathered}
\end{equation}

\subsection{Uniqueness and continuous dependence on the initial data}\label{subsec:well-posedness}
In this section, we show the uniqueness of solutions and the continuous dependence on the initial data. Let $\mathcal{V}_1$ and $\mathcal{V}_2$ be two weak solutions with initial data $(\mathcal{V}_0)_1$ and $(\mathcal{V}_0)_2$, respectively. Assume the radius of analyticity of $(\mathcal{V}_0)_1$ and $(\mathcal{V}_0)_2$ is $ \tau_0 $. By virtue of \eqref{tau-inequality} and \eqref{time-T}, for $i=1,2$, let
\begin{equation}\label{tau-T-i}
\begin{gathered}
    \tau_i(t) := \tau_0 - t - C_{r,i} \int_0^t \bigl( \| \mathcal{V}_i(s)\|_{r,0,\tau_i(s)} + \| \partial_z \mathcal{V}_i(s)\|_{r,0,\tau_i(s)} \bigr) ds, \\
    \text{and} \qquad \mathcal{T}_i :=
    \Big(\frac{\sqrt{\frac{C_{r,i} ^2 \| (\mathcal{V}_0)_i\|_{r,0,\tau_0}^2}{2\nu} + 2\tau_0(1+C_{r,i}  \| (\mathcal{V}_0)_i\|_{r,0,\tau_0})} - \frac{C_{r,i}  \| (\mathcal{V}_0)_i\|_{r,0,\tau_0}}{\sqrt{2\nu}}}{2(1+C_{r,i}  \| (\mathcal{V}_0)_i\|_{r,0,\tau_0})}\Big)^2
\end{gathered}
\end{equation}
such that, according to \eqref{apri-003}, \eqref{regularity-v}, \eqref{regularity-vt-1}, and \eqref{regularity-vt-2},
\begin{equation*}
     \| \mathcal{V}_i(t)\|_{r,0,\tau_i(t)}^2 + 2 \int_0^t   \Big( \nu\|\partial_z \mathcal{V}_i(s)\|_{r,0,\tau_i(s)}^2 +  \|A^{r+\frac{1}{2}} e^{\tau_i(s) A}  \mathcal{V}_i(s)\|^2 \Big) ds \leq \|(\mathcal{V}_0)_i\|^2_{r,0,\tau_{i0}},
\end{equation*}
for $t\in[0,\mathcal{T}_i]$, and
\begin{equation*}
    \mathcal{V}_i\in L^\infty\big(0,\mathcal{T}_i;  \mathcal{S}_{r,0,\tau_i(t)}\big) \cap L^2\big(0,\mathcal{T}_i; V\cap \mathcal{S}_{r,1,\tau_i(t)}\cap \mathcal{S}_{r+\frac{1}{2},0,\tau_i(t)}\big),
\end{equation*}
\begin{equation*}
    \partial_t \mathcal{V}_i \quad \text{and} \quad A^{r-\frac{1}{2}} e^{\tau_i A} \partial_t \mathcal{V}_i\in L^2\big(0,\mathcal{T}_i; V'\big).
\end{equation*}
We remind readers that $ C_{r,i}, i =1,2, $ are independent of $ \Omega $ and $ \tau_0 $.

Let 
\begin{equation}\label{M}
  M := \max\Big\{ \|(\mathcal{V}_0)_1\|_{r,0,\tau_{0}}, \|(\mathcal{V}_0)_2\|_{r,0,\tau_{0}}\Big\}.
\end{equation}
Denote by $\delta \mathcal{V} := \mathcal{V}_1- \mathcal{V}_2$ and $\delta p := p_1 -p_2$. Let
\begin{equation}\label{tilde-tau-T}
\begin{gathered}
 \widetilde{\tau}(t) := \tau_0 - t - C_r \sum\limits_{i=1}^2\int_0^t \bigl( \| \mathcal{V}_i(s)\|_{r,0,\tau_i(s)} +
 \| \mathcal{V}_i(s)\|_{r,0,\tau_i(s)}^2 + \| \partial_z \mathcal{V}_i(s)\|_{r,0,\tau_i(s)} \bigr) ds, \\
    \text{and} \quad  \widetilde{\mathcal{T}} := \Big(\frac{\sqrt{\frac{2C_{r}^2 M^2}{\nu} + 2\tau_0\big(1+2C_{r}  (M^2+M)\big)} - \frac{ \sqrt{2}C_{r}  M}{\sqrt{\nu}}}{2\big(1+2C_{r}  (M^2+M)\big)}\Big)^2,
\end{gathered}
\end{equation}
where $ C_r $ is a positive constant, to be determined later, satisfying
\begin{equation}\label{unq-def:cnst-c-r}
    C_r \geq \max\lbrace C_{r,1}, C_{r,2}\rbrace .
\end{equation}
In particular, \eqref{tilde-tau-T} and \eqref{unq-def:cnst-c-r} imply that $ \tilde{\tau}(t) \leq \tau_i(t) $ and $ \widetilde{\mathcal T} \leq \widetilde{\mathcal T}_i $ for $ i \in \lbrace 1, 2 \rbrace $ and $ t \in (0,\widetilde{\mathcal T} ] $. Therefore, for $i=1,2$,
\begin{gather}
\label{regularity-v-uniqueness}
    \delta \mathcal{V} \quad \text{and} \quad  \mathcal{V}_i\in L^\infty\big(0,\widetilde{\mathcal{T}};  \mathcal{S}_{r,0,\widetilde{\tau}(t)}) \cap L^2\big(0,\widetilde{\mathcal{T}}; V\cap \mathcal{S}_{r,1,\widetilde{\tau}(t)}\cap \mathcal{S}_{r+\frac{1}{2},0,\widetilde{\tau}(t)}),\\
\label{regularity-vt}
    \partial_t \delta \mathcal{V}  \quad \text{and} \quad  A^{r-\frac{1}{2}} e^{\widetilde{\tau} A} \partial_t \delta \mathcal{V}\in L^2\big(0,\widetilde{\mathcal{T}}; V'\big),
\end{gather}
and
\begin{equation*}
     \| \mathcal{V}_i(t)\|_{r,0,\widetilde{\tau}(t)}^2 + 2 \int_0^t \Big(   \nu\|\partial_z \mathcal{V}_i(s)\|_{r,0,\widetilde{\tau}(s)}^2 +  \|A^{r+\frac{1}{2}} e^{\widetilde{\tau}(s) A}  \mathcal{V}_i(s)\|^2 \Big) ds \leq M^2,
\end{equation*}
for $t\in[0,\widetilde{\mathcal{T}}]$.

From system \eqref{PE-system}, it is clear that
\begin{equation*}
    \begin{gathered}
        \partial_t \delta \mathcal{V} + \delta \mathcal{V}\cdot \nabla \mathcal{V}_1 + \mathcal{V}_2\cdot \nabla \delta \mathcal{V} - \Big(\int_0^z \nabla\cdot \delta \mathcal{V}(\boldsymbol{x},s)ds \Big) \partial_z \mathcal{V}_1
        - \Big(\int_0^z \nabla\cdot \mathcal{V}_2(\boldsymbol{x},s)ds \Big) \partial_z \delta \mathcal{V}\\
        + \Omega \delta \mathcal{V}^\perp -\nu \partial_{zz} \delta \mathcal{V} + \nabla \delta p =0 \qquad \text{and} \qquad \partial_z \delta p = 0.
    \end{gathered}
\end{equation*}
Notice that from \eqref{regularity-v-uniqueness}, one has that $A^{r-\frac{1}{2}} e^{\widetilde{\tau} A} \delta \mathcal{V} \in L^2\big(0,\widetilde{\mathcal{T}}; V\big)$. Thanks to \eqref{regularity-vt}, similar calculation as in \eqref{apri-001} leads to
\begin{equation}\label{apri-005}
    \begin{split}
    &\frac{1}{2}\frac{d}{dt} \| \delta \mathcal{V}\|_{r-\frac{1}{2},0,\widetilde{\tau}}^2 + \nu  \|\partial_z \delta \mathcal{V}\|_{r-\frac{1}{2},0,\widetilde{\tau}}^2 - \dot{\widetilde{\tau}} \|A^{r} e^{\widetilde{\tau} A}  \delta \mathcal{V}\|^2 \\
    =& -\Big\langle \delta \mathcal{V}\cdot \nabla \mathcal{V}_1 + \mathcal{V}_2\cdot \nabla \delta \mathcal{V} - \Big(\int_0^z \nabla\cdot \delta \mathcal{V}(\boldsymbol{x},s)ds \Big) \partial_z \mathcal{V}_1 - \Big(\int_0^z \nabla\cdot \mathcal{V}_2(\boldsymbol{x},s)ds \Big) \partial_z \delta \mathcal{V} ,\delta \mathcal{V} \Big\rangle \\
    &
    -\Big\langle A^{r-\frac{1}{2}} e^{\widetilde{\tau} A}(\delta \mathcal{V}\cdot \nabla \mathcal{V}_1), A^{r-\frac{1}{2}} e^{\widetilde{\tau} A}  \delta \mathcal{V}\Big\rangle
    +\Big\langle A^{r-\frac{1}{2}} e^{\widetilde{\tau} A} \Big[ \Big(\int_0^z \nabla\cdot \delta \mathcal{V}(\boldsymbol{x},s)ds \Big) \partial_z \mathcal{V}_1 \Big], A^{r-\frac{1}{2}} e^{\widetilde{\tau} A}  \delta \mathcal{V} \Big\rangle \\
    & -\Big\langle A^{r-\frac{1}{2}} e^{\widetilde{\tau} A}(\mathcal{V}_2\cdot \nabla \delta \mathcal{V}), A^{r-\frac{1}{2}} e^{\widetilde{\tau} A}  \delta \mathcal{V}\Big\rangle
    +\Big\langle A^{r-\frac{1}{2}} e^{\widetilde{\tau} A} \Big[ \Big(\int_0^z \nabla\cdot \mathcal{V}_2(\boldsymbol{x},s)ds \Big) \partial_z \delta \mathcal{V} \Big], A^{r-\frac{1}{2}} e^{\widetilde{\tau} A}  \delta \mathcal{V} \Big\rangle.
\end{split}
\end{equation}
After applying integration by parts, the H\"older inequality, the Young inequality, and the Sobolev inequality, since $r>2$, one has
\begin{equation*}
\begin{split}
    &\Big| \Big\langle \delta \mathcal{V}\cdot \nabla \mathcal{V}_1 + \mathcal{V}_2\cdot \nabla \delta \mathcal{V} - \Big(\int_0^z \nabla\cdot \delta \mathcal{V}(\boldsymbol{x},s)ds \Big) \partial_z \mathcal{V}_1 - \Big(\int_0^z \nabla\cdot \mathcal{V}_2(\boldsymbol{x},s)ds \Big) \partial_z \delta \mathcal{V} ,\delta \mathcal{V} \Big\rangle \Big| \\
    = & \Big| \Big\langle \delta \mathcal{V}\cdot \nabla \mathcal{V}_1 - \Big(\int_0^z \nabla\cdot \delta \mathcal{V}(\boldsymbol{x},s)ds \Big) \partial_z \mathcal{V}_1 ,\delta \mathcal{V} \Big\rangle \Big| \leq C_{r-\frac{1}{2}} \| \mathcal{V}_1\|_{r,1,\widetilde{\tau}} \| \delta \mathcal{V}\|_{r-\frac{1}{2},0,\widetilde{\tau}}^2.
\end{split}
\end{equation*}
Thanks to Lemmas \ref{lemma-type1} and \ref{lemma-type2}, the H\"older inequality, the Young inequality, and the Sobolev inequality, since $r>2$, one has
\begin{align*}
&
\begin{aligned}
     &\Big| \Big\langle A^{r-\frac{1}{2}} e^{\widetilde{\tau} A}(\delta \mathcal{V}\cdot \nabla \mathcal{V}_1), A^{r-\frac{1}{2}} e^{\widetilde{\tau} A}  \delta \mathcal{V}\Big\rangle    \Big| \\
     &\quad \leq \int_0^1 C_{r-\frac{1}{2}}\Big[ (\|A^{r-\frac{1}{2}} e^{\widetilde{\tau} A} \delta \mathcal{V}(z)\|_{L^2(\mathbb{T}^2)} + \|\delta \mathcal{V}(z)\|_{L^2(\mathbb{T}^2)} ) \|A^{r} e^{\widetilde{\tau} A} \mathcal{V}_1(z)\|_{L^2(\mathbb{T}^2)} \|A^{r} e^{\widetilde{\tau} A} \delta \mathcal{V}(z)\|_{L^2(\mathbb{T}^2)} \\
       & \qquad + \|A^{r} e^{\widetilde{\tau} A} \delta \mathcal{V}(z)\|_{L^2(\mathbb{T}^2)}  \|A^{r} e^{\widetilde{\tau} A} \mathcal{V}_1(z)\|_{L^2(\mathbb{T}^2)} \|A^{r-\frac{1}{2}} e^{\widetilde{\tau} A} \delta \mathcal{V}(z)\|_{L^2(\mathbb{T}^2)}\Big] dz\\
       &\quad \leq   C_{r-\frac{1}{2}}  \|\mathcal{V}_1\|_{r,1,\widetilde{\tau}} (\|\delta \mathcal{V}\|_{r-\frac{1}{2},0,\widetilde{\tau}}^2+ \|A^{r} e^{\widetilde{\tau} A}\delta \mathcal{V}\|^2 ),
\end{aligned}\\
&
    \begin{aligned}
        &\Big| \Big\langle A^{r-\frac{1}{2}} e^{\widetilde{\tau} A}(\mathcal{V}_2\cdot \nabla \delta \mathcal{V}), A^{r-\frac{1}{2}} e^{\widetilde{\tau} A}  \delta \mathcal{V}\Big\rangle    \Big| \\
        & \quad \leq \int_0^1 C_{r-\frac{1}{2}}\Big[ (\|A^{r-\frac{1}{2}} e^{\widetilde{\tau} A} \mathcal{V}_2(z)\|_{L^2(\mathbb{T}^2)} + \|\mathcal{V}_2(z)\|_{L^2(\mathbb{T}^2)} ) \|A^{r} e^{\widetilde{\tau} A} \delta \mathcal{V}(z)\|_{L^2(\mathbb{T}^2)} \|A^{r} e^{\widetilde{\tau} A} \delta \mathcal{V}(z)\|_{L^2(\mathbb{T}^2)} \\
       &\qquad + \|A^{r} e^{\widetilde{\tau} A} \mathcal{V}_2(z)\|_{L^2(\mathbb{T}^2)}  \|A^{r} e^{\widetilde{\tau} A} \delta \mathcal{V}(z)\|_{L^2(\mathbb{T}^2)} \|A^{r-\frac{1}{2}} e^{\widetilde{\tau} A} \delta \mathcal{V}(z)\|_{L^2(\mathbb{T}^2)}\Big] dz\\
       & \quad \leq  C_{r-\frac{1}{2}}  \|\mathcal{V}_2\|_{r,1,\widetilde{\tau}}  \|A^{r} e^{\widetilde{\tau} A}\delta \mathcal{V}\|^2,
    \end{aligned}\\
&
    \Big| \Big\langle A^{r-\frac{1}{2}} e^{\widetilde{\tau} A} \Big[ \Big(\int_0^z \nabla\cdot \delta \mathcal{V}(\boldsymbol{x},s)ds \Big) \partial_z \mathcal{V}_1 \Big], A^{r-\frac{1}{2}} e^{\widetilde{\tau} A}  \delta \mathcal{V} \Big\rangle \Big| \leq C_{r-\frac{1}{2}} \|\mathcal{V}_1\|_{r,1,\widetilde{\tau}}  \|A^{r} e^{\widetilde{\tau} A}\delta \mathcal{V}\|^2,
\end{align*}
and
\begin{equation}\label{apri-006}
\begin{aligned}
    &\Big| \Big\langle A^{r-\frac{1}{2}} e^{\widetilde{\tau} A} \Big[ \Big(\int_0^z \nabla\cdot \mathcal{V}_2(\boldsymbol{x},s)ds \Big) \partial_z \delta\mathcal{V} \Big], A^{r-\frac{1}{2}} e^{\widetilde{\tau} A}  \delta\mathcal{V} \Big\rangle \Big| \\
    &\quad \leq C_{r-\frac{1}{2}} \|A^{r} e^{\widetilde{\tau} A}\mathcal{V}_2\| \|\partial_z\delta \mathcal{V}\|_{r-\frac{1}{2},0,\widetilde{\tau}}  \|A^{r} e^{\widetilde{\tau} A}\delta\mathcal{V}\|\leq \frac{\nu}{2} \|\partial_z\delta \mathcal{V}\|_{r-\frac{1}{2},0,\widetilde{\tau}}^2 + C_{\nu,r-\frac{1}{2}} \|\mathcal{V}_2\|_{r,0,\widetilde{\tau}}^2 \|A^{r} e^{\widetilde{\tau} A}\delta\mathcal{V}\|^2.
    \end{aligned}
\end{equation}
Consequently, combining the calculations between \eqref{apri-005} and \eqref{apri-006} yields
\begin{equation*}
    \begin{split}
    &\frac{1}{2}\frac{d}{dt} \| \delta \mathcal{V}\|_{r-\frac{1}{2},0,\widetilde{\tau}}^2 + \frac{1}{2} \nu  \|\partial_z \delta \mathcal{V}\|_{r-\frac{1}{2},0,\widetilde{\tau}}^2  \\
    \leq & \Big(\dot{\widetilde{\tau}} + C_{\nu,r-\frac{1}{2}} \|\mathcal{V}_2\|_{r,0,\widetilde{\tau}}^2 + C_{r-\frac{1}{2}}(\|\mathcal{V}_1\|_{r,1,\widetilde{\tau}} + \|\mathcal{V}_2\|_{r,1,\widetilde{\tau}})  \Big) \|A^{r} e^{\widetilde{\tau} A}  \delta\mathcal{V}\|^2 + C_{r-\frac{1}{2}} \| \mathcal{V}_1\|_{r,1,\widetilde{\tau}} \| \delta \mathcal{V}\|_{r-\frac{1}{2},0,\widetilde{\tau}}^2 .
\end{split}
\end{equation*}
In addition, from \eqref{tilde-tau-T}, and \eqref{unq-def:cnst-c-r}, and the fact that $\tau_i(t) \geq \widetilde{\tau}(t)$, $ i=1,2$, one can derive that
\begin{equation*}
\begin{split}
    &\dot{\widetilde{\tau}} + C_{\nu,r-\frac{1}{2}} \|\mathcal{V}_2\|_{r,0,\widetilde{\tau}}^2 + C_{r-\frac{1}{2}}(\|\mathcal{V}_1\|_{r,1,\widetilde{\tau}} + \|\mathcal{V}_2\|_{r,1,\widetilde{\tau}}) \\
    = &  - 1 - C_r \sum\limits_{i=1}^2 \bigl(\| \mathcal{V}_i(t)\|_{r,0,\tau_i(t)}+ \| \mathcal{V}_i(t)\|^2_{r,0,\tau_i(t)} + \| \partial_z \mathcal{V}_i(t)\|_{r,0,\tau_i(t)} \bigr)
    \\
    &\qquad \qquad + C_{\nu,r-\frac{1}{2}} \|\mathcal{V}_2\|_{r,0,\widetilde{\tau}}^2 + C_{r-\frac{1}{2}}(\|\mathcal{V}_1\|_{r,1,\widetilde{\tau}} + \|\mathcal{V}_2\|_{r,1,\widetilde{\tau}})\\
    \leq & \big(\widetilde{C}_{\nu,r-\frac{1}{2}}-C_r \big) \sum\limits_{i=1}^2 \bigl(\| \mathcal{V}_i(t)\|_{r,0,\widetilde{\tau}(t)}+ \| \mathcal{V}_i(t)\|^2_{r,0,\widetilde{\tau}(t)} + \| \partial_z \mathcal{V}_i(t)\|_{r,0,\widetilde{\tau}(t)} \bigr)  \leq 0,
\end{split}
\end{equation*}
where we have chosen
\begin{equation}\label{unq-def:c-r}
    C_r := \max \lbrace \widetilde{C}_{\nu,r-\frac{1}{2}}, C_{r,1}, C_{r,2} \rbrace.
\end{equation}

In conclusion, with $ C_r $ satisfying \eqref{unq-def:c-r}, one has
\begin{equation}\label{apri:007}
    \begin{split}
    \frac{1}{2}\frac{d}{dt} \| \delta \mathcal{V}\|_{r-\frac{1}{2},0,\widetilde{\tau}}^2 + \frac{1}{2} \nu  \|\partial_z\delta  \mathcal{V}\|_{r-\frac{1}{2},0,\widetilde{\tau}}^2  \leq  C_{r-\frac{1}{2}} \| \mathcal{V}_1\|_{r,1,\widetilde{\tau}} \| \delta \mathcal{V}\|_{r-\frac{1}{2},0,\widetilde{\tau}}^2 .
\end{split}
\end{equation}

Applying the Gr\"onwall inequality to \eqref{apri:007} results in
\begin{equation}\label{continuous-dependence}
    \| \delta \mathcal{V}(t)\|_{r-\frac{1}{2},0,\widetilde{\tau}(t)}^2 \leq \| \delta \mathcal{V}(0)\|_{r-\frac{1}{2},0,\tau_0}^2 \exp(\int_0^t 2C_{r-\frac{1}{2}} \| \mathcal{V}_1(s)\|_{r,1,\widetilde{\tau}(s)}ds)
\end{equation}
for $t\in [0,\widetilde{\mathcal{T}}]$, which establishes the continuous dependence on the initial data as well as the uniqueness of the weak solutions. This, together with section \ref{subsec:local-a-priori}, finishes the proof of Theorem \ref{theorem-local}.
\subsection{Instantaneous analyticity in the $z$-variable}\label{section-radius}
In this section, we will show that the weak solution obtained in Theorem \ref{theorem-local} immediately becomes analytic in the $z$-variable (and thus analytic in all variables) when $t>0$. Moreover, the radius of analyticity in the $z$-variable increases as long as the solution exists. For simplicity, we consider the even extension for $\mathcal V$ in the $z$-variable, which is compatible with \eqref{PE-BC}, and work in the unit three-dimensional torus $\mathbb{T}^3$ instead of $\mathcal D$. With abuse of notations, we use $ \mathcal V $ to represent both $ \mathcal V $ in $ \mathcal D $ and its even extension with respect to the $ z $-variable in $ \mathbb T^3 $.


We first introduce the following notations that are only used in this subsection. For $f\in L^2(\mathbb{T}^3)$ even with respect to the $ z $-variable, we consider the following functional space
\begin{equation*}
    \mathcal{S}_{r,s,\tau,\eta} := \Big\{ f\in L^2(\mathbb{T}^3), \|f\|_{r,s,\tau,\eta} < \infty, ~ f ~ \text{even with respect to the $ z $-variable}  \Big\},
\end{equation*}
where
\begin{gather*}
   \|f\|_{r,s,\tau,\eta}^2 := \sum\limits_{\boldsymbol{k}\in  2\pi \mathbb{Z}^2, k_3 \in 2\pi \mathbb Z} \Big(1 + (|\boldsymbol{k}|^{2r} + |k_3|^{2s}) e^{2\tau |\boldsymbol{k}|} e^{2\eta |k_3|}\Big) |\hat{f}_{\boldsymbol{k},k_3}|^2 \\
   \text{and} \qquad  \hat{f}_{\boldsymbol{k},k_3} :=\int_{\mathbb T^3} e^{-i\boldsymbol{k} \cdot \boldsymbol{x}- ik_3 z} f(\boldsymbol{x},z)\, d\boldsymbol{x} dz.
\end{gather*}
Denote by
\begin{equation*}
    A_h := \sqrt{-\Delta_h}, \quad A_z :=\sqrt{-\partial_{zz}},
\end{equation*}
subject to periodic boundary condition, defined by, in terms of the Fourier coefficients,
\begin{equation*}
    (\widehat{A_h^r f})_{\boldsymbol{k},k_3} := |\boldsymbol{k}|^r \hat{f}_{\boldsymbol{k},k_3}, \qquad (\widehat{A_z^sf})_{\boldsymbol{k},k_3} := |k_3|^s \hat{f}_{\boldsymbol{k},k_3}, \qquad (\boldsymbol k ,k_3)  \in 2\pi (\mathbb Z^2 \times \mathbb Z), ~ r,s\geq 0.
\end{equation*}
Accordingly, one has
\begin{equation*}
    \|f\|_{r,s,\tau,\eta}^2 = \|f\|^2 +  \|A_h^r e^{\tau A_h} e^{\eta A_z} f\|^2 + \|A_z^s e^{\tau A_h} e^{\eta A_z} f\|^2.
\end{equation*}

With such notations, we establish the following theorem:
\begin{theorem}\label{theorem-radius}
Assume $\mathcal{V}_0\in \mathcal{S}_{r,0,\tau_0,0}$ with $r>2$ and $\tau_0>0$. Let $\Omega \in \mathbb{R}$ be arbitrary and fixed. Then there exist $\mathcal T>0$ defined in \eqref{T-radius},
 $\tau(t)>0$ given in  \eqref{tau-radius}, below, and $\eta(t) = \frac{\nu}{2}t$, such that there exists a unique solution $\mathcal V$ to system  \eqref{PE-system} with \eqref{PE-IC} and \eqref{PE-BC}  in $ [0,\mathcal{T}]$ satisfying
 \begin{equation*}
    \mathcal{V}\in L^\infty\big(0,\mathcal{T};  \mathcal{S}_{r,0,\tau(t),\eta(t)}\big) \cap L^2\big(0,\mathcal{T};  \mathcal{S}_{r,1,\tau(t),\eta(t)}\big),
\end{equation*}
and depending continuously on the initial data. In particular, $\mathcal V$ immediately becomes analytic in all spatial variables for $t>0$.
\end{theorem}

\begin{remark}
    After restricting $ \mathcal V_0 $ and $ \mathcal V $ in $ \mathbb T^2 \times (0,1) $, the solutions in Theorem \ref{theorem-radius} are the same to the ones in Theorem \ref{theorem-local}, thanks to the uniqueness of solutions. Therefore, the gain of analyticity in the $ z $-variable of Theorem \ref{theorem-radius} can be regarded as a property to solutions in Theorem \ref{theorem-local}.
\end{remark}

{
\begin{remark}
Theorem \ref{theorem-radius} states the gain of analyticity in the $ z $-variable for solutions to system \eqref{PE-system}. One can then apply the result from \cite{GILT20} to study the effect of rotation on the lifespan of solutions after a initial time layer. However, in order to achieve a longer lifespan, the result from \cite{GILT20} requests smallness of Sobolev norms in the baroclinic mode. In this paper, thanks to the effect of viscosity, we are able to relax the condition by requiring smallness in a larger functional space. See Theorem \ref{theorem-main} and remark \ref{rm:small-baroclinic}, below, for more details.
\end{remark}
}

\begin{proof}[Sketch of proof]
Here we only show the {\it a priori} estimates. Direct calculation of $$\langle \eqref{PE-1}, \mathcal V \rangle + \langle A_h^r e^{\tau A_h}e^{\eta A_z} \eqref{PE-1},A_h^r e^{\tau A_h}e^{\eta A_z} \mathcal V\rangle + \langle  e^{\tau A_h}e^{\eta A_z} \eqref{PE-1},e^{\tau A_h} e^{\eta A_z} \mathcal V\rangle , $$ after applying integration by parts, \eqref{PE-3}, and \eqref{PE-BC}, shows that
\begin{equation*}
\begin{split}
    &\frac{1}{2} \frac{d}{dt} \|\mathcal V\|_{r,0,\tau,\eta}^2 + \nu \|\partial_z \mathcal V\|_{r,0,\tau,\eta}^2 - \dot{\tau}\Big( \|A_h^{r+\frac{1}{2}}  e^{\tau A_h} e^{\eta A_z} \mathcal V\|^2 +  \|A_h^{\frac{1}{2}} e^{\tau A_h} e^{\eta A_z} \mathcal V \|^2\Big)
    \\
    &\qquad \qquad \qquad \qquad - \dot{\eta} \Big(\|A_z^{\frac{1}{2}} A_h^r  e^{\tau A_h}  e^{\eta A_z} \mathcal V\|^2  + \|A_z^{\frac{1}{2}}  e^{\tau A_h}  e^{\eta A_z} \mathcal V\|^2  \Big)
    \\
    + & \Big\langle A_h^r e^{\tau A_h} e^{\eta A_z}(\mathcal V \cdot\nabla \mathcal V), A_h^r e^{\tau A_h} e^{\eta A_z} \mathcal V \Big\rangle  + \Big\langle e^{\tau A_h} e^{\eta A_z}(\mathcal V \cdot\nabla \mathcal V), e^{\tau A_h} e^{\eta A_z} \mathcal V \Big\rangle
    \\
    + &\Big\langle A_h^r e^{\tau A_h} e^{\eta A_z} \Big((\int_0^z \nabla\cdot \mathcal V ds) \partial_z \mathcal V\Big), A_h^r e^{\tau A_h} e^{\eta A_z}\mathcal V \Big\rangle
    \\
    + &\Big\langle  e^{\tau A_h} e^{\eta A_z} \Big((\int_0^z \nabla\cdot \mathcal V ds) \partial_z \mathcal V\Big),  e^{\tau A_h} e^{\eta A_z}\mathcal V \Big\rangle= 0.
\end{split}
\end{equation*}
Denote by
\begin{equation*}
    E: = \|\mathcal V\|_{r,0,\tau,\eta}^2 = \sum\limits_{(\boldsymbol{k},k_3)\in  2\pi \mathbb{Z}^3} \Big(1 + (|\boldsymbol{k}|^{2r} +1) e^{2\tau |\boldsymbol{k}|} e^{2\eta |k_3|} \Big) |\hat{\mathcal V}_{\boldsymbol{k},k_3}|^2,
\end{equation*}
\begin{equation*}
    F := \|\partial_z \mathcal V\|_{r,0,\tau,\eta}^2 = \sum\limits_{(\boldsymbol{k},k_3)\in  2\pi \mathbb{Z}^3} |k_3|^2 \Big(1 + (|\boldsymbol{k}|^{2r} +1)  e^{2\tau |\boldsymbol{k}|} e^{2\eta |k_3|} \Big) |\hat{\mathcal V}_{\boldsymbol{k},k_3}|^2,
\end{equation*}
\begin{equation*}
    G: = \|A_h^{r+\frac{1}{2}}  e^{\tau A_h} e^{\eta A_z} \mathcal V\|^2 +  \|A_h^{\frac{1}{2}} e^{\tau A_h} e^{\eta A_z} \mathcal V \|^2= \sum\limits_{(\boldsymbol{k},k_3)\in  2\pi \mathbb{Z}^3}  (|\boldsymbol{k}|^{2r+1} + |\boldsymbol{k}|^{\frac{1}{2}})  e^{2\tau |\boldsymbol{k}|} e^{2\eta |k_3|} |\hat{\mathcal V}_{\boldsymbol{k},k_3}|^2,
\end{equation*}
\begin{equation*}
    H:= \|A_z^{\frac{1}{2}} A_h^r  e^{\tau A_h}  e^{\eta A_z} \mathcal V\|^2  + \|A_z^{\frac{1}{2}}  e^{\tau A_h}  e^{\eta A_z} \mathcal V\|^2 = \sum\limits_{(\boldsymbol{k},k_3)\in  2\pi \mathbb{Z}^3} (|k_3| |\boldsymbol{k}|^{2r} + |k_3|)  e^{2\tau |\boldsymbol{k}|} e^{2\eta |k_3|} |\hat{\mathcal V}_{\boldsymbol{k},k_3}|^2.
\end{equation*}
Observe that $H\leq F$. After setting $\dot{\eta} = \frac{\nu}{2}$, one obtains that
\begin{equation*}
\begin{split}
    &\frac{1}{2} \frac{d}{dt} E + \frac{1}{2}\nu G - \dot{\tau} G
    \\
    + & \Big\langle A_h^r e^{\tau A_h} e^{\eta A_z}(\mathcal V \cdot\nabla \mathcal V), A_h^r e^{\tau A_h} e^{\eta A_z} \mathcal V \Big\rangle  + \Big\langle e^{\tau A_h} e^{\eta A_z}(\mathcal V \cdot\nabla \mathcal V), e^{\tau A_h} e^{\eta A_z} \mathcal V \Big\rangle
    \\
    + &\Big\langle A_h^r e^{\tau A_h} e^{\eta A_z} \Big((\int_0^z \nabla\cdot \mathcal V ds) \partial_z \mathcal V\Big), A_h^r e^{\tau A_h} e^{\eta A_z}\mathcal V \Big\rangle
    \\
    + &\Big\langle  e^{\tau A_h} e^{\eta A_z} \Big((\int_0^z \nabla\cdot \mathcal V ds) \partial_z \mathcal V\Big),  e^{\tau A_h} e^{\eta A_z}\mathcal V \Big\rangle \leq 0.
\end{split}
\end{equation*}
For the nonlinear terms, by applying similar calculations as in Lemma \ref{lemma-type1} and Lemma \ref{lemma-type2} (we also refer the readers to \cite{GILT20} for detailed calculations in $\mathbb{T}^3$), one can obtain that
\begin{equation*}
\begin{split}
    \Big| \Big\langle A_h^r e^{\tau A_h} e^{\eta A_z}(\mathcal V \cdot\nabla \mathcal V), A_h^r e^{\tau A_h} e^{\eta A_z} \mathcal V \Big\rangle\Big|  + \Big| \Big\langle e^{\tau A_h} e^{\eta A_z}(\mathcal V \cdot\nabla \mathcal V), e^{\tau A_h} e^{\eta A_z} \mathcal V \Big\rangle \Big|
    \leq  C_r \Big(E^{\frac{1}{2}}+F^{\frac{1}{2}} \Big)G,
\end{split}
\end{equation*}
\begin{equation*}
  \begin{split}
      \Big|\Big\langle A_h^r e^{\tau A_h} e^{\eta A_z} \Big((\int_0^z \nabla\cdot \mathcal V ds) \partial_z \mathcal V\Big), A_h^r e^{\tau A_h} e^{\eta A_z}\mathcal V \Big\rangle\Big| \leq C_r F^{\frac{1}{2}} G,
  \end{split}
\end{equation*}
and thanks to the Young inequality,
\begin{equation*}
   \Big| \Big\langle  e^{\tau A_h} e^{\eta A_z} \Big((\int_0^z \nabla\cdot \mathcal V ds) \partial_z \mathcal V\Big),  e^{\tau A_h} e^{\eta A_z}\mathcal V \Big\rangle   \Big| \leq C_r F^{\frac{1}{2}}E^{\frac{1}{2}}G^{\frac{1}{2}} \leq \frac{C_r}{\nu}E G + \frac{\nu}{4}F.
\end{equation*}
Therefore, combining all the estimates above leads to
\begin{equation}\label{apri-radius}
   \begin{split}
    & \frac{d}{dt} E + \frac{1}{2}\nu F \leq \Big(\dot{\tau}+ C_{r}(E^{\frac{1}{2}} + F^{\frac{1}{2}} + \frac{1}{\nu}E)  \Big)G .
\end{split}
\end{equation}
By taking $\dot{\tau}+ C_{r}(E^{\frac{1}{2}} + F^{\frac{1}{2}} + \frac{1}{\nu}E)=0$, one obtains
\begin{equation*}
   E(t) + \frac{1}{2}\nu \int_0^t F(s) ds \leq E(0).
\end{equation*}
Integrating in time for $\dot{\tau}+ C_{r}(E^{\frac{1}{2}} + F^{\frac{1}{2}} + \frac{1}{\nu}E)=0$, we have
\begin{equation}\label{tau-radius}
    \tau(t) = \tau_0 - \int_0^t C_r (E^{\frac{1}{2}}(s) + F^{\frac{1}{2}}(s) + \frac{1}{\nu}E(s)) ds \geq \tau_0 - C_r \Big( E^{\frac{1}{2}}(0) (t + \sqrt{\frac{2 t}{\nu}} ) + E(0)\frac{t}{\nu} \Big).
\end{equation}
Since $E(0) = \|\mathcal V\|_{r,0,\tau,0}^2$, we denote by
\begin{equation}\label{T-radius}
    \mathcal T := \Big( \frac{\sqrt{\frac{2}{\nu}\|\mathcal V\|_{r,0,\tau,0}^2 + \frac{2\tau_0}{C_r}(\frac{\|\mathcal V\|_{r,0,\tau,0}^2}{\nu} + \|\mathcal V\|_{r,0,\tau,0})} - \sqrt{\frac{2}{\nu}} \|\mathcal V\|_{r,0,\tau,0}}{2(\frac{\|\mathcal V\|_{r,0,\tau,0}^2}{\nu} + \|\mathcal V\|_{r,0,\tau,0})}   \Big)^{\frac{1}{2}}>0,
\end{equation}
which solves
\begin{equation*}
    \|\mathcal V\|_{r,0,\tau,0}(\mathcal T + \sqrt{\mathcal T}) + \frac{1}{\nu} \|\mathcal V\|_{r,0,\tau,0}^2 \mathcal T = \frac{\tau_0}{2C_r}.
\end{equation*}
Then one has
\begin{equation*}
    \tau(t) \geq \tau_0/2 > 0 \quad \text{for} \quad t \in [0,\mathcal T].
\end{equation*}
Notice that the radius of analyticity in the $z$ variable satisfies $\eta = \frac{\nu}{2}t$. Therefore, \eqref{apri-radius} implies that
\begin{equation*}
    \mathcal{V}\in L^\infty\big(0,\mathcal{T};  \mathcal{S}_{r,0,\tau(t),\eta(t)}\big) \cap L^2\big(0,\mathcal{T};  \mathcal{S}_{r,1,\tau(t),\eta(t)}\big).
\end{equation*}

Based on the estimates above, one is able to show the existence, uniqueness, and continuous dependence on the initial data of the solution $\mathcal V$. We omit the details.

\end{proof}

\section{The limit resonant system}\label{section-limit-system}
In this section, we derive the formal limit resonant system, i.e., the limit system of system \eqref{PE-system} (or, equivalently, system \eqref{barotropic-baroclinic}) as $|\Omega|\rightarrow \infty$, and discuss some properties of the limit resonant system. Recall that from \eqref{up3}, we have
\begin{equation}\label{upl}
\begin{aligned}
 \partial_t \mathcal{V}_+ = & -e^{i\Omega t} \Big(\underbrace{\mathcal{V}_+ \cdot \nabla \mathcal{V}_+ - \mathfrak P_0( \mathcal{V}_+ \cdot \nabla \mathcal{V}_+ + (\nabla \cdot \mathcal{V}_+) \mathcal{V}_+) - (\int_0^z \nabla\cdot \mathcal{V}_+(\boldsymbol{x},s)ds ) \partial_z \mathcal{V}_+}_{=:I_1} \Big)  \\
&-  \Big[\underbrace{\Big(\overline{\mathcal{V}} \cdot \nabla \mathcal{V}_+  + \frac{1}{2}(\mathcal{V}_+ \cdot \nabla)(\overline{\mathcal{V}}+i\overline{\mathcal{V}}^\perp) \Big) - \nu \partial_{zz} \mathcal{V}_+}_{=:I_0}\Big]  \\
& - e^{-i\Omega t} \Big(\underbrace{\mathcal{V}_- \cdot \nabla \mathcal{V}_+ - \mathfrak P_0( \mathcal{V}_- \cdot \nabla \mathcal{V}_+ + (\nabla \cdot \mathcal{V}_-) \mathcal{V}_+) - (\int_0^z \nabla\cdot \mathcal{V}_-(\boldsymbol{x},s)ds ) \partial_z \mathcal{V}_+}_{=:I_{-1}} \Big) \\
&
- e^{-2i\Omega t} \underbrace{\frac{1}{2} (\mathcal{V}_- \cdot \nabla)(\overline{\mathcal{V}}+i\overline{\mathcal{V}}^\perp)}_{=:I_{-2}}.
\end{aligned}
\end{equation}

We can further rewrite \eqref{upl} as
\begin{eqnarray*}
&&\hskip-.58in \partial_t \Big[\mathcal{V}_+ - \frac{i}{\Omega} \Big(e^{i\Omega t}I_1 - e^{-i\Omega t} I_{-1} - \frac{1}{2} e^{-2i\Omega t} I_{-2} \Big) \Big] = - \frac{i}{\Omega} \Big(e^{i\Omega t}\partial_t I_1 - e^{-i\Omega t}\partial_t I_{-1} - \frac{1}{2}e^{-2i\Omega t}\partial_t I_{-2}\Big) - I_0.
\label{up4}
\end{eqnarray*}
Denote by the formal limits of $\mathcal{V}_+, \mathcal{V}_- $, and $ \overline{\mathcal{V}}$ to be $V_+, V_-$, and $ \overline{V}$, respectively. By taking limit $\Omega \rightarrow \infty $, we obtain the limit resonant equation for $\mathcal{V}_+$ is
\begin{equation}
 \partial_t V_+ = - (\overline{V} \cdot \nabla)V_+ - \frac{1}{2} (V_+ \cdot \nabla)(\overline{V}+i\overline{V}^\perp) + \nu \partial_{zz} V_+  .
\label{uplimit}
\end{equation}
Similarly, one has
\begin{equation}
\partial_t V_- = - (\overline{V} \cdot \nabla)V_- - \frac{1}{2} (V_- \cdot \nabla)(\overline{V}-i\overline{V}^\perp) + \nu \partial_{zz} V_- ,
\label{umlimit}
\end{equation}
and
\begin{equation}
 \partial_t \overline{V} + \overline{V}\cdot \nabla \overline{V} + \nabla p =0, \qquad \nabla \cdot \overline{V} = 0, \qquad \partial_z p = 0.
\label{vbarlimit}
\end{equation}
Notice that \eqref{vbarlimit} is nothing but the 2D Euler equations.
Accordingly, we consider the initial conditions
\begin{equation}\label{initial:+-bar}
    (\overline{V}_0, (V_+)_0, (V_-)_0) = (\overline{\mathcal{V}}_0, \frac{1}{2}(\widetilde{\mathcal{V}}_0 + i\widetilde{\mathcal{V}}_0^\perp), \frac{1}{2}(\widetilde{\mathcal{V}}_0 - i\widetilde{\mathcal{V}}_0^\perp) )
\end{equation}
for equations \eqref{uplimit}--\eqref{vbarlimit}. Since $ \overline{V}_0$,  $ \overline{\mathcal V}_0 $, and $ \widetilde{\mathcal V}_0 $ are real valued, one has that $ (V_+)_0 = (V_-)_0^* $, $ (V_+)_0+ (V_-)_0 = i ((V_+)_0 - (V_-)_0)^\perp = \widetilde{\mathcal V}_0 $, and, thanks to \eqref{vbarlimit}, $ \overline{V} $ is real valued. Thanks to \eqref{uplimit} and \eqref{umlimit}, one has
\begin{gather}
\begin{gathered}
    \partial_t ( V_+ - V_-^* ) = - ( \overline{V} \cdot \nabla )( V_+ - V_-^* ) - \dfrac{1}{2} \bigl[( V_+ - V_-^* ) \cdot \nabla \bigr] ( \overline V + i \overline V^\perp) + \nu \partial_{zz} ( V_+ - V_-^* ),
\end{gathered}\\
\begin{gathered}
\partial_t \bigl[(V_+ + V_-) - i ( V_+ - V_-)^\perp \bigr] = - ( \overline V \cdot \nabla )\bigl[(V_+ + V_-) - i ( V_+ - V_-)^\perp \bigr] - \dfrac{1}{2}  \\
+ \nu \partial_{zz} \bigl[(V_+ + V_-) - i ( V_+ - V_-)^\perp \bigr].
\end{gathered}
\end{gather}
Therefore, provided solutions exist and are well-posed, one has $ V_+ \equiv V_-^* $ and $ V_+ + V_- \equiv i ( V_+ - V_-)^\perp $. Let
\begin{equation}\label{def:limit-V-tilde}
	\widetilde V := V_+ + V_- .
\end{equation}
Notice that, according to \eqref{def:V+-}, $ \widetilde V $ is the formal limit of $ \mathcal V_+ + \mathcal V_- = e^{-i\Omega t} \mathfrak P_+ \mathcal V  + e^{i\Omega t} \mathfrak P_- \mathcal V $, as $ \Omega \rightarrow \infty $.
It is easy to verify that \begin{equation}\label{apri:107}
    V_\pm = \dfrac{1}{2} ( \widetilde V \pm i \widetilde V^\perp),
\end{equation} and
\begin{equation*}
\partial_t \widetilde{V} + (\overline{V}\cdot \nabla)\widetilde{V} + \frac{1}{2}(\widetilde{V}\cdot \nabla \overline{V} -  \widetilde{V}^\perp\cdot \nabla \overline{V}^\perp) - \nu \partial_{zz} \widetilde{V} =0,
\end{equation*}
or, thanks to $\nabla \cdot \overline{V} = 0$, equivalently,
\begin{equation}\label{vtildelimit}
\partial_t \widetilde{V} + \overline{V} \cdot \nabla \widetilde{V} + \frac{1}{2} \widetilde{V}^\perp (\nabla^\perp \cdot \overline{V}) - \nu \partial_{zz} \widetilde{V} = 0.
\end{equation}

In summary, to solve the limit equations \eqref{uplimit}--\eqref{vbarlimit} with \eqref{initial:+-bar} is equivalent to solve the following equations:
\begin{subequations}\label{system-Vbar-Vtilde}
\begin{gather}
    \partial_t \overline{V} + \overline{V}\cdot \nabla \overline{V} + \nabla p = 0,\label{systemVbar-1}
    \\
     \nabla\cdot\overline{V} = 0, \qquad \partial_z p = 0 ,\label{systemVbar-2}
    \\
    \partial_t \widetilde{V} + \overline{V} \cdot \nabla \widetilde{V} + \frac{1}{2} \widetilde{V}^\perp (\nabla^\perp \cdot \overline{V}) - \nu \partial_{zz} \widetilde{V} = 0 , \label{systemVtilde}
    \\
    \partial_z\widetilde{V}\big|_{z=0,1} = 0, \qquad \overline{V}(0) = \overline{\mathcal{V}}_0, \quad \text{and} \quad \widetilde{V}(0) = \widetilde{\mathcal V}_0. \label{systemlimit-initial}
\end{gather}
\end{subequations}
Notice that, thanks to our choice of $ \overline{ \mathcal V}_0 $ and $\widetilde{\mathcal V}_0$, one has $\mathfrak P_0 \overline{V} = \overline{V}$ and $\mathfrak P_0 \widetilde{V} = 0$.
In addition, \eqref{systemVbar-1}--\eqref{systemVbar-2} is the $2D$ Euler system, and \eqref{systemVtilde} is a linear transport equation with a  stretching term and vertical dissipation. In the rest of this section, we summarize the well-posedness theory of \eqref{system-Vbar-Vtilde}.

\subsection{Well-posedness theory of \eqref{systemVbar-1} and \eqref{systemVbar-2}}
The global well-posedness of solutions to the $2D$ Euler system \eqref{systemVbar-1}--\eqref{systemVbar-2} in Sobolev spaces $H^r(\mathbb{T}^2) = \mathcal{S}_{r,0,0}$ with $r > 3$ is a classical result (see, e.g., \cite{BM02}). Moreover, from equation (3.84) in \cite{BM02}, for $r > 3$, we have
\begin{equation}\label{Euler-Sobolev-estimate-1}
    \frac{d}{dt}\|\overline{V}\|_{r,0,0} \leq C_r \|\overline{V}\|_{r,0,0}(1+\ln^+\|\overline{V}\|_{r,0,0}).
\end{equation}
Let $\|\overline{V}_0\|_{r,0,0} \leq  M$ for some $M\geq 0$.
Denote by $W(t):=\|\overline{V}(t)\|_{r,0,0} + e$.
Thanks to $\ln^+ x + 1 \leq 2\ln(x+e)$, from \eqref{Euler-Sobolev-estimate-1}, we have
\begin{equation*}
    \frac{d}{dt} W \leq C_r W \ln W.
\end{equation*}
Therefore, one can obtain that
\begin{equation}\label{Euler-Sobolev-estimate-2}
    \|\overline{V}(t)\|_{r,0,0} \leq W(t) \leq W(0)^{e^{C_r t}} = (\|\overline{V}_0\|_{r,0,0} +e)^{e^{C_r t}} \leq (M+e)^{e^{C_r t}} =: \theta_{M,r}(t).
\end{equation}

The authors in \cite{LO97} proved the global existence of solutions to system \eqref{systemVbar-1}--\eqref{systemVbar-2} for initial data in the space of analytic functions. For completion, we state it here, with slight modifications to meet our settings. See also \cite{GILT20}.
\begin{proposition}\label{global-Euler}
    Assume $\overline{V}_0 \in \mathcal{S}\cap \mathcal{S}_{r,0,\tau_0}$ with $r>3$ and $\tau_0 >0$, and suppose that $\|\overline{V}_0\|_{r,0,\tau_0} \leq M$ for some $M\geq 0$. Let
    \begin{equation*}
    \tau(t) := \tau_0 \exp\Big( -C_r \int_0^t    h(s)ds\Big), 
    \end{equation*}
    where
    \begin{equation*}
      h^2(t) := \|\overline{V}_0\|_{r,0,\tau_0}^2 + C_r\int_0^t \theta_{M,r}^3(s)ds,
    \end{equation*}
    with $\theta_{M,r}(t)$ defined in \eqref{Euler-Sobolev-estimate-2}. Then for any given time $\mathcal{T}>0$, there exists a unique solution
    \begin{equation*}
        \overline{V}\in L^\infty(0, \mathcal{T}; \mathcal{S}\cap \mathcal{S}_{r,0,\tau(t)} )
    \end{equation*}
    to system \eqref{systemVbar-1}--\eqref{systemVbar-2}. Moreover, there exist constants $C_M>1$ and $C_r>1$ such that
\begin{equation*}
 \|\overline{V}(t)\|_{r,0,\tau(t)}^2 \leq h^2(t) \leq C_M^{\exp(C_r t)}.
\end{equation*}
The solution is continuously depending on the initial data.
\end{proposition}

\subsection{Global well-posedness of system \eqref{system-Vbar-Vtilde}}

In this subsection, we establish the global well-posedness of limit resonant system \eqref{system-Vbar-Vtilde} in both Sobolev spaces $\mathcal{S}_{r,s,0}$ and analytic-Sobolev spaces $\mathcal{S}_{r,s,\tau}$.

\begin{proposition}\label{global-limit} Let $ r > 2 $ and $ s \in \lbrace 0,1 \rbrace $.
    Assume that $\overline{V}_0 \in  \mathcal{S}\cap \mathcal{S}_{r+1,0,0}$ and $\widetilde{V}_0 \in \mathcal{S}\cap \mathcal{S}_{r,s,0}$. Let $M\geq 0$ be the constant such that $\|\overline{V}_0\|_{r+1,0,0} \leq M$. Then there exists a function $K(t) := C_M^{\exp(C_r t)}$ with constants $C_M>1$ and $C_r>1$, such that for any given time $\mathcal{T}>0$,  there exists a unique solution $( \overline{V}, \widetilde{V} )\in  L^\infty(0,\mathcal{T}; \mathcal{S} \cap \mathcal{S}_{r+1,0,0}) \times L^\infty(0,\mathcal{T}; \mathcal{S}\cap \mathcal{S}_{r,s,0})$ of system \eqref{system-Vbar-Vtilde}, which satisfies
    \begin{equation}\label{sobolev-growth-limit}
    \|\overline{V}(t)\|_{r+1,0,0} \leq  K(t)  \qquad \text{and} \qquad \|\widetilde{V}(t)\|^2_{r,s,0} + 2\nu \int_0^t \|\partial_z \widetilde{V}(\xi)\|^2_{r,s,0} d\xi  \leq  \|\widetilde{V}_0\|^2_{r,s,0} \; e^{\int_0^t K(s) \,ds}.
    \end{equation}
    On the other hand, suppose that $\overline{V}_0 \in \mathcal{S}\cap \mathcal{S}_{r+1,0,\tau_0}$ and $\widetilde{V}_0 \in \mathcal{S}\cap \mathcal{S}_{r,s,\tau_0}$ with $\tau_0 >0$, and that $\|\overline{V}_0\|_{r+1,0,\tau_0}\leq M$. Let
    \begin{equation}\label{tau-limit-system}
            \tau(t) := \tau_0 \exp(-\int_0^t K(s)ds).
    \end{equation}
    Then for any given time $\mathcal{T}>0$, there exists a unique solution $(\overline{V},\widetilde{V}) \in L^\infty(0,\mathcal{T}; \mathcal{S} \cap \mathcal{S}_{r+1,0,\tau}) \times L^\infty(0,\mathcal{T}; \mathcal{S} \cap \mathcal{S}_{r,s,\tau})$ of system \eqref{system-Vbar-Vtilde} such that
    \begin{equation}\label{analytic-growth-limit}
            \| \overline{V}(t)\|_{r+1,0,\tau(t)} \leq K(t)
           \qquad \text{and} \qquad  \|\widetilde{V}(t)\|^2_{r,s,\tau(t)} +  2\nu \int_0^t \|\partial_z \widetilde{V}(\xi)\|^2_{r,s,\tau(\xi)} d\xi \leq \|\widetilde{V}_0\|^2_{r,s,\tau_0} e^{\int_0^t K(s)\,ds}.
    \end{equation}
    The solutions continuously depend on the initial data.
\end{proposition}

\begin{proof}[Sketch of proof]
We will consider the case when $s=1$ and only show the {\it a priori} estimates. The construction of solutions,  uniqueness, and continuous dependency of solutions on initial data, as well as the case when $s=0$, are left to readers as exercises. The global well-posedness of the $2D$ Euler equations in Sobolev spaces and corresponding growth estimate have been reviewed in the previous subsection. From \eqref{Euler-Sobolev-estimate-2}, we obtain that
\begin{equation}\label{apri:105}
    \|\overline{V}\|_{r+1,0,0} \leq K_1(t)
\end{equation}
for some function $K_1(t) := C_{M,1}^{\exp(C_{r,1}t)}$ with some constants $ C_{M,1}, C_{r,1} > 1 $.

Denote by $ \mathfrak I $ the identity map.
For the growth of $\|\widetilde{V}\|_{H^{r}}$, after calculating $ 2 \langle \eqref{systemVtilde}, (\mathfrak I - \partial_{zz}) \widetilde V \rangle + 2 \langle A^r  \eqref{systemVtilde},  (\mathfrak I - \partial_{zz}) A^r   \widetilde V \rangle $ and applying integration by parts to the resultant, one has, thanks to $\partial_z \overline{V} = 0 $, $ \nabla\cdot \overline{V}=0$, and $r>\frac{5}{2}$, for some constant $ C_{r,s} > 0 $,
\begin{equation}\label{estimate-limit-sobolev}
    \frac{d}{dt}\|\widetilde{V}\|^2_{r,1,0} + 2\nu \|\partial_z \widetilde{V}\|^2_{r,1,0} \leq C_{r,s}\|\overline{V}\|_{r+1,0,0} \|\widetilde{V}\|^2_{r,1,0}.
\end{equation}
After applying the Gr\"onwall inequality to the above,  by virtue of \eqref{apri:105}, we obtain
\begin{equation}\label{apri:106}
    \|\widetilde{V}(t)\|^2_{r,1,0} + 2\nu \int_0^t \|\partial_z \widetilde{V}(\xi)\|^2_{r,1,0} d\xi \leq \|\widetilde{V}_0\|^2_{r,1,0}\exp\Big(C_{r,s} \int_0^t K_1(\xi)d\xi\Big).
\end{equation}

On the other hand,
the global well-posedness of the $2D$ Euler equations in the space of analytic functions and the corresponding growth estimate are summarized in Proposition \ref{global-Euler}. 
We can first choose some suitable function $K_2(t) := C_{M,2}^{\exp(C_{r,2}t)}$, with $ C_{M,2}, C_{r,2} > 1 $, such that $\| \overline{V}(t)\|_{r+1,0,\tau_E(t)} \leq K_2(t)$ with $\tau_E(t) := \tau_0 \exp(-\int_0^t K_2(s)ds)$.

Let $ \tau = \tau(t) $ to be determined.
For $\widetilde{V}$, after calculating $ \langle \eqref{systemVtilde}, (\mathfrak I - \partial_{zz}) \widetilde V \rangle + \langle A^r e^{\tau A} \eqref{systemVtilde},  (\mathfrak I - \partial_{zz}) A^r e^{\tau A} \widetilde V \rangle $ and applying integration by parts, the H\"older inequality, the Sobolev inequality, Lemma \ref{lemma-banach-algebra}, and Lemma \ref{lemma-difference-type1} to the resultant,
since $r>2$, one has, for some constant $ C_{r,a} >0 $,
\begin{equation}\label{apri:101}
    \begin{split}
        &\frac{1}{2} \frac{d}{dt} \| \widetilde{V} \|_{r,1,\tau}^2 + \nu \| \partial_z \widetilde{V} \|_{r,1,\tau}^2 - \dot{\tau} \Big( \|A^{r+\frac{1}{2}} e^{\tau A} \widetilde{V} \|^2 +  \|A^{r+\frac{1}{2}} e^{\tau A} \partial_z \widetilde{V} \|^2 \Big)
        \\
        = & \underbrace{ - \Big\langle \overline{V}\cdot \nabla \widetilde{V}, \widetilde{V} \Big\rangle - \frac{1}{2} \Big\langle  (\nabla^\perp \cdot \overline{V}) \widetilde{V}^\perp, \widetilde{V} \Big\rangle - \Big\langle \overline{V}\cdot \nabla \partial_z \widetilde{V}, \partial_z \widetilde{V} \Big\rangle - \frac{1}{2} \Big\langle  (\nabla^\perp \cdot \overline{V}) \partial_z \widetilde{V}^\perp, \partial_z \widetilde{V} \Big\rangle }_{=0}
        \\
        & - \Big\langle A^r e^{\tau A} (\overline{V}\cdot \nabla \widetilde{V}), A^r e^{\tau A} \widetilde{V} \Big\rangle - \frac{1}{2} \Big\langle A^r e^{\tau A} \Big( (\nabla^\perp \cdot \overline{V}) \widetilde{V}^\perp \Big), A^r e^{\tau A} \widetilde{V} \Big\rangle
        \\
        & - \Big\langle A^r e^{\tau A} (\overline{V}\cdot \nabla \partial_z \widetilde{V}), A^r e^{\tau A} \partial_z \widetilde{V} \Big\rangle - \frac{1}{2} \Big\langle A^r e^{\tau A} \Big( (\nabla^\perp \cdot \overline{V}) \partial_z \widetilde{V}^\perp \Big), A^r e^{\tau A}\partial_z \widetilde{V} \Big\rangle
        \\
        = & - \Big(\Big\langle A^r e^{\tau A} (\overline{V}\cdot \nabla \widetilde{V}), A^r e^{\tau A} \widetilde{V} \Big\rangle - \underbrace{\Big\langle \overline{V}\cdot \nabla A^r e^{\tau A}\widetilde{V}, A^r e^{\tau A} \widetilde{V} \Big\rangle }_{=0}\Big)
        \\
        &- \Big(\Big\langle A^r e^{\tau A} (\overline{V}\cdot \nabla \partial_z \widetilde{V}), A^r e^{\tau A} \partial_z \widetilde{V} \Big\rangle - \underbrace{\Big\langle \overline{V}\cdot \nabla A^r e^{\tau A} \partial_z \widetilde{V}, A^r e^{\tau A} \partial_z \widetilde{V} \Big\rangle }_{=0}\Big)
        \\
        & - \frac{1}{2} \Big\langle A^r e^{\tau A} \Big( (\nabla^\perp \cdot \overline{V}) \widetilde{V}^\perp \Big) , A^r e^{\tau A} \widetilde{V} \Big\rangle - \frac{1}{2} \Big\langle A^r e^{\tau A} \Big( (\nabla^\perp \cdot \overline{V}) \partial_z \widetilde{V}^\perp\Big), A^r e^{\tau A}\partial_z \widetilde{V} \Big\rangle
        \\
        \leq &  C_{r,a}\tau \|\overline{V} \|_{r+1,0,\tau}  \Big( \|A^{r+\frac{1}{2}} e^{\tau A} \widetilde{V} \|^2 + \|A^{r+\frac{1}{2}} e^{\tau A} \partial_z \widetilde{V} \|^2 \Big) +  C_{r,a}\|\overline{V} \|_{r+1,0,\tau} \|\widetilde{V} \|_{r,1,\tau}^2  .
    \end{split}
\end{equation}

Now, let
\begin{equation}\label{apri:102}
    K(t) := \max\lbrace (1 +  C_{r,s} ) K_1(t), (1 + C_{r,a}) K_2(t) \rbrace \qquad \text{and} \qquad \tau = \tau(t) := \tau_0 \exp (- \int_0^t K(s) \,ds ).
\end{equation} Then $ \tau(t) $ satisfies
\begin{equation*}
    \tau(t) \leq \tau_E(t) \qquad \text{and} \qquad
    \dot{\tau} + C_{r,a} \tau \|\overline{V} \|_{r+1,0,\tau} \leq \dot{\tau} + C_{r,a} \tau \|\overline{V} \|_{r+1,0,\tau_E} \leq \dot\tau + C_{r,a}\tau K_2 \leq 0.
\end{equation*}
Therefore,
\begin{equation}\label{apri:103}
    \|  \overline{V} (t)\|_{r+1,0,\tau(t)} \leq \|  \overline{V} (t)\|_{r+1,0,\tau_E(t)} \leq K_2(t) \leq K(t), \end{equation}
and, after applying the Gr\"onwall inequality to \eqref{apri:101}, we have
\begin{equation}\label{apri:104}
    \begin{gathered}
        \| \widetilde{V}(t) \|_{r,1,\tau(t)}^2 + 2 \nu \int_0^t  \| \partial_z \widetilde{V}(\xi) \|_{r,1,\tau(\xi)}^2 d\xi \leq \| \widetilde{V}_0 \|_{r,1,\tau_0}^2 \exp\Big( \int_0^t C_{r,a} \|  \overline{V} (\xi)\|_{r+1,0,\tau(\xi)}d\xi \Big)
        \\
        \leq \| \widetilde{V}_0 \|_{r,1,\tau_0}^2 e^{\int_0^t C_{r,a} K_2(s)\,ds} \leq  \| \widetilde{V}_0 \|_{r,1,\tau_0}^2 e^{\int_0^t K(s)\,ds}.
    \end{gathered}
\end{equation}
Consequently, according to \eqref{apri:105}, \eqref{apri:106}, \eqref{apri:103}, and \eqref{apri:104}, $ K $ and $\tau$ as in \eqref{apri:102} verify \eqref{sobolev-growth-limit} and \eqref{analytic-growth-limit}.
\end{proof}
\begin{remark}
From Proposition \ref{global-limit}, one can see that the growth of $\|\overline{V}(t)\|_{r+1,0,0}$ and  $\|\overline{V}(t)\|_{r+1,0,\tau(t)}$ are double exponential in time, while the growth of $\|\widetilde{V}(t)\|_{r,s,0}$ and  $\|\widetilde{V}(t)\|_{r,s,\tau(t)}$ are triple exponential in time.
\end{remark}

\begin{remark}\label{remark-equivalence-upm-vtilde}
Thanks to \eqref{apri:107},
similarly as in \eqref{vtilde-upm-analytic}, we have
\begin{equation*}
   \|V_+\|_{r,s,\tau}^2 = \|V_-\|_{r,s,\tau}^2 = \frac{1}{2}\|\widetilde{V}\|_{r,s,\tau} ^2,
\end{equation*}
whose growths are also triple exponential.
\end{remark}

\begin{remark}\label{special-situation-limit}
Proposition \ref{global-limit} is for the general initial data. However, by considering special solutions to the 2D Euler equations, one has the following:
\begin{itemize}

\item Supposed that $\overline{V}$ is uniformly-in-time bounded in $\mathcal{S}_{r+1,0,\tau}$, i.e., $ \sup_{0\leq t < \infty }\|\overline{V}(t)\|_{r+1,0,\tau} \leq C_{M,r}$ for some positive constant $C_{M,r}$, then one can control the growth of $\|\widetilde{V}(t)\|_{r,1,\tau}$ by one exponential in time.

\item Supposed that $\sup_{0\leq t < \infty} \|\overline{V}(t)\|_{r+1,0,\tau} \leq \frac{\nu}{4C_{r,\alpha}}$ is small enough, by applying the Poincar\'e inequality and with $\tau$ chosen suitably, from \eqref{apri:101} one can derive that
\begin{equation*}
    \frac{d}{dt}\|\widetilde{V}\|^2_{r,1,\tau} + \frac{1}{2}\nu \|\partial_z \widetilde{V}\|^2_{r,1,\tau} \leq -\nu \|\widetilde{V}\|^2_{r,1,\tau}.
\end{equation*}
After applying the Gr\"onwall inequality to the above, we obtain
\begin{equation*}
    \|\widetilde{V}(t)\|_{r,1,\tau(t)}^2 e^{\nu t} + \frac{1}{2}\nu \int_0^t \|\partial_z \widetilde{V}(\xi)\|^2_{r,1,\tau(\xi)} e^{\nu \xi} d\xi \leq \|\widetilde{V}_0\|_{r,1,\tau_0}^2.
\end{equation*}
In particular, the estimate above holds when $\overline{V} \equiv 0$, i.e., zero solutions to the 2D Euler equations.

\end{itemize}
\end{remark}


\section{Effect of fast rotation}\label{section-rotation}
In this section, we investigate the effect of rotation on the lifespan $\mathcal{T}$ of solutions to system \eqref{PE-system}. We show that the existing time of the solution in $\mathcal{S}_{r,0,\tau(t)}$ can be prolonged for large $ |\Omega| $ provided that the Sobolev norm $\|\widetilde{\mathcal{V}}_0\|_{\frac{5}{2}+\delta,1,0}$ is small, while the analytic-Sobolev norm $\|\widetilde{\mathcal{V}}_0\|_{r,0,\tau_0}$ can be large. Such initial data is referred to as ``well-prepared" initial data.

\begin{theorem}\label{theorem-main}
	Let $ \delta \in (0,\frac{1}{2}) $ be a constant. Let $|\Omega| \geq |\Omega_0| > 1$ and $|\Omega_0|$ be large enough such that condition \eqref{constraint-1} below holds.
Assume $\overline{\mathcal{V}}_0 \in \mathcal{S}\cap \mathcal{S}_{r+3,0,\tau_0}$, $\widetilde{\mathcal{V}}_0 \in \mathcal{S}\cap\mathcal{S}_{r+2,0,\tau_0}\cap \mathcal{S}_{r+1,1,\tau_0} $ with $r>2$ and $\tau_0>0$.
Let $M\geq 0$ be such that
\begin{equation}\label{main-bound-initial}
	\| \overline {\mathcal V}_0 \|_{r+3,0,\tau_0}^2 + \| \widetilde  {\mathcal V}_0 \|_{r+2,0,\tau_0}^2 + \| \widetilde  {\mathcal V}_0 \|_{r+1,1,\tau_0}^2 \leq M,
\end{equation}
and
{
\begin{equation}\label{constraint-tilde-V}
    \|\widetilde{\mathcal{V}}_0\|_{\frac{3}{2}+\delta,0,0} \leq  \frac{M}{|\Omega_0|^{1/2}}.
\end{equation}
}
Then there exists a time $\mathcal{T} = \mathcal{T}(\tau_0, |\Omega_0|, M, r, \nu)$ satisfying
{
\begin{equation}\label{constraint-1}
     \mathcal T = \frac{1}{C_{\tau_0, M, r} } \log\left[e^{-C_{\tau_0, M, r}}\log\left[\log \left[C_{\tau_0, M, r,\nu}
       \log \left(C_{\tau_0, M, r,\nu} |\Omega_0|\right)\right]
     \right]\right] > 0,
\end{equation}
}
for some positive constant $ C_{\tau_0, M, r,\nu}>0 $ depending only on $ \tau_0, M $, and $ r $, such that the unique solution $\mathcal{V}$ obtained in Theorem \ref{theorem-local} satisfies  \begin{equation}\label{main-regularity}
    \mathcal{V} \in  L^\infty(0,\mathcal{T};\mathcal{S} \cap \mathcal{S}_{r,0,\tau(t)}),
\end{equation}
with $ \tau(t) > 0 $, $ t \in [0,\mathcal T] $, satisfying \eqref{est:tau}, below.
In particular, from \eqref{constraint-1}, $\mathcal{T} \rightarrow \infty$ as $ |\Omega_0|^{1/2}\rightarrow \infty$, for any fixed $ \nu $.
\end{theorem}
{
\begin{remark}\label{rm:small-baroclinic}
Recall that the result in \cite{GILT20} requires the initial baroclinic mode $ \widetilde{\mathcal{V}}_0 $ to be small in the $ H^{3+\delta} $ space instead of \eqref{constraint-tilde-V} in Theorem \ref{theorem-main}. This relaxation on the requirement of $ \widetilde{\mathcal{V}}_0 $ is due to the vertical viscosity.
\end{remark}
}

In Theorem \ref{theorem-main}, we consider general initial data $\overline{\mathcal V}_0$ for the barotropic mode, {where the vertical viscosity helps relax the requirement on the initial baroclinic mode, but does not help prolong the lifespan}. By virtue of Remark \ref{special-situation-limit}, when the solution $\overline{V}$ to the $2D$ Euler equations with initial condition $\overline{\mathcal V}_0$ satisfies certain conditions, the smallness condition \eqref{constraint-tilde-V} can be relaxed and the result \eqref{constraint-1} can be improved. The following theorem is the summary of these results:

\begin{theorem}\label{theorem-small-barotropic}
 With the same assumptions as in Theorem \ref{theorem-main}, let $\overline{V}(t)$ be the solution to the $2D$ Euler equations with initial condition $\overline{V}_0 = \overline{\mathcal V}_0$. Then
\begin{enumerate}
    \item if  $\|\overline{V}(t)\|_{r+3,0,\tau(t)} \leq C_{M,r}$, the result \eqref{constraint-1} can be improved to $\mathcal T = \frac{1}{C_{\tau_0, M, r, \nu} } \log(\log (|\Omega_0|) )$;

    \item if $\|\overline{V}(t)\|_{r+3,0,\tau(t)} \leq \frac{\nu}{4C_{r,\alpha}}$ which is small enough, then \eqref{constraint-tilde-V} can be relaxed and replaced by $\|\widetilde{\mathcal{V}}_0\|_{\frac{3}{2}+\delta,0,0} \leq  \frac{\tau_0}{C_{r,\nu,M}}$,  and \eqref{constraint-1} can be improved to $\mathcal T = \frac{1}{C_{\tau_0, M, r, \nu} } \log (|\Omega_0|) $;

    \item finally, if the initial condition satisfies $\|\overline{\mathcal V}_0\|_{r+3,0,\tau_0} \leq \frac{M}{|\Omega_0|}$, \eqref{constraint-tilde-V} can be relaxed and replaced by $\|\widetilde{\mathcal{V}}_0\|_{\frac{3}{2}+\delta,0,0} \leq  \frac{\tau_0}{C_{r,\nu,M}}$,  and \eqref{constraint-1} can be improved to
    $\mathcal T = \frac{|\Omega_0|^{\frac{1}{2}}}{C_{\tau_0, M, r, \nu} } $.
\end{enumerate}
\end{theorem}

\begin{remark}
 Compared to \cite{GILT20}, the main improvement in Theorem \ref{theorem-main} is that the initial data is analytic in the horizontal variables but only $L^2$ in the vertical variable. The main improvements in Theorem \ref{theorem-small-barotropic} are points $(ii)$ and $(iii)$, where the smallness assumption does not depend on $\Omega_0$, and the lifespan is growing faster with respect to $\Omega_0.$ For more details, we refer readers to R3 and R4 in the introduction (pages \pageref{R3} and \pageref{R4}).
\end{remark}

In this section, we focus on equations \eqref{up3}--\eqref{barotropic-evolution-4}, which are equivalent to system \eqref{PE-system}. To prove Theorem \ref{theorem-main}, in section \ref{section-difference-system},
we rewrite \eqref{up3}--\eqref{barotropic-evolution-4} as the perturbation of \eqref{uplimit}--\eqref{vbarlimit}.
In section \ref{section-proof-main}, we establish a series of {\it a priori} estimates on the solutions to the perturbation system. This together with Proposition \ref{global-limit} will finish the proof of Theorem \ref{theorem-main}. In section \ref{section-smallbaro}, the proof of Theorem \ref{theorem-small-barotropic} is provided.

\begin{remark}\label{remark-analytic-in-z-main}
  In this section, we only focus on the long-time existence of the weak solution. By virtue of Theorem \ref{theorem-radius}, the weak solution is analytic in all spatial variables.
\end{remark}


\subsection{The perturbation system}\label{section-difference-system}
Denote by
\begin{equation}\label{def:perturbation}
    \overline{\phi} := \overline{\mathcal{V}} - \overline{V} \qquad \text{and} \qquad \phi_\pm := \mathcal{V}_\pm - V_\pm.
\end{equation}
Calculating the difference between \eqref{up3}, \eqref{um1} , \eqref{barotropic-evolution-4}  and \eqref{uplimit}, \eqref{umlimit}, \eqref{vbarlimit}, respectively, leads to
\begin{gather}
\begin{gathered}
	\partial_t \phi_+ + \overline{\phi} \cdot \nabla V_+ + \overline{\phi} \cdot \nabla \phi_+ + \overline{V} \cdot \nabla \phi_+  + \frac{1}{2}(\phi_+ \cdot \nabla)(\overline{V}+i\overline{V}^\perp) + \frac{1}{2}(\phi_+ \cdot \nabla)(\overline{\phi}+i\overline{\phi}^\perp) \\
	+ \frac{1}{2}(V_+ \cdot \nabla)(\overline{\phi}+i\overline{\phi}^\perp) - \nu \partial_{zz} \phi_+ + e^{i\Omega t}\Big(Q_{1,+,+} - \mathfrak P_0 Q_{1,+,+} - \mathfrak P_0 Q_{2,+,+} - Q_{3,+,+}\Big)  \\
	+ e^{-i\Omega t} \Big(Q_{1,-,+} - \mathfrak P_0 Q_{1,-,+} - \mathfrak P_0 Q_{2,-,+} - Q_{3,-,+}\Big) + e^{-2i\Omega t} Q_{4,-,+}  = 0 , \label{difference-phip}
\end{gathered}
\\
\nonumber
\linebreak
\\
\begin{gathered}
	\partial_t \phi_- + \overline{\phi} \cdot \nabla V_- + \overline{\phi} \cdot \nabla \phi_- + \overline{V} \cdot \nabla \phi_-  + \frac{1}{2}(\phi_- \cdot \nabla)(\overline{V}-i\overline{V}^\perp) + \frac{1}{2}(\phi_- \cdot \nabla)(\overline{\phi}-i\overline{\phi}^\perp) \\
	+ \frac{1}{2}(V_- \cdot \nabla)(\overline{\phi}-i\overline{\phi}^\perp) - \nu \partial_{zz} \phi_- + e^{-i\Omega t}\Big(Q_{1,-,-} - \mathfrak P_0 Q_{1,-,-} - \mathfrak P_0 Q_{2,-,-} - Q_{3,-,-}\Big)  \\
	+ e^{i\Omega t} \Big(Q_{1,+,-} - \mathfrak P_0 Q_{1,+,-} - \mathfrak P_0 Q_{2,+,-} - Q_{3,+,-}\Big) + e^{2i\Omega t}Q_{4,+,-}  = 0 , \label{difference-phim}
	\\
	\nonumber \linebreak\\
	\nabla \cdot  \overline\phi = 0 , \qquad \partial_z p = 0,
\end{gathered}
\\
\nonumber
\linebreak
\\
\begin{gathered}
	 \partial_t \overline{\phi} + \overline{\phi}\cdot \nabla \overline{V} + \overline{\phi}\cdot \nabla \overline{\phi} + \overline{V}\cdot \nabla \overline{\phi}  +   e^{2i\Omega t} \mathfrak P_0\Big( Q_{1,+,+} + Q_{2,+,+} \Big)\\
	+ e^{-2i\Omega t} \mathfrak P_0\Big( Q_{1,-,-} + Q_{2,-,-} \Big) + \nabla p = 0,\label{difference-phibar}
\end{gathered}
\end{gather}
where
\begin{align*}
Q_{1,\pm,\mp} := & \phi_\pm \cdot \nabla V_\mp + \phi_\pm \cdot \nabla \phi_\mp + V_\pm \cdot \nabla \phi_\mp + V_\pm \cdot \nabla V_\mp, \nonumber
\\
Q_{2,\pm,\mp} := & (\nabla\cdot \phi_\pm) V_\mp + (\nabla\cdot \phi_\pm) \phi_\mp + (\nabla\cdot V_\pm) \phi_\mp + (\nabla\cdot V_\pm) V_\mp, \nonumber
\\
Q_{3,\pm,\mp} := & (\int_0^z \nabla\cdot \phi_\pm (\boldsymbol{x},s)ds) \partial_z V_\mp +(\int_0^z \nabla\cdot \phi_\pm (\boldsymbol{x},s)ds) \partial_z \phi_\mp \nonumber
\\
& +(\int_0^z \nabla\cdot V_\pm (\boldsymbol{x},s)ds) \partial_z \phi_\mp+(\int_0^z \nabla\cdot V_\pm (\boldsymbol{x},s)ds) \partial_z V_\mp, \nonumber
\\
Q_{4,\pm,\mp} := & \frac{1}{2}\Big[(\phi_\pm \cdot \nabla)(\overline{V} \mp i\overline{V}^\perp) + (\phi_\pm \cdot \nabla)(\overline{\phi} \mp i\overline{\phi}^\perp)
\nonumber\\
& + (V_\pm \cdot \nabla)(\overline{\phi} \mp i\overline{\phi}^\perp) + (V_\pm \cdot \nabla)(\overline{V} \mp i\overline{V}^\perp) \Big]. \nonumber
\end{align*}
Recalling that $ ( \overline{\mathcal V}, \mathcal V_\pm) $ and $ ( \overline V, V_\pm ) $ are complemented with the same initial data. Hence, we have
\begin{equation}\label{difference-ic}
  \overline{\phi}|_{t=0}=0 \qquad \text{and} \qquad \phi_\pm|_{t=0} = 0.
\end{equation}


\subsection{Proof of Theorem \ref{theorem-main}}\label{section-proof-main}
In this subsection, we prove Theorem \ref{theorem-main}. Thanks to Proposition \ref{global-limit},  for any given $ \mathcal T \in (0,\infty)$, let $V_\pm$ and $\overline{V}$ be the global solution to equations \eqref{uplimit}-\eqref{vbarlimit} in $L^\infty(0,\mathcal T; \mathcal{S}\cap \mathcal{S}_{r+2,1,\tau(t)})$ and $L^\infty(0,\mathcal T; \mathcal{S}\cap \mathcal{S}_{r+3,0,\tau(t)})$ for some $ \tau = \tau(t), t \in [0,\mathcal T) $, respectively. Next, we provide the energy estimate in the space $\mathcal{S}_{r,0,\tau(t)}$ for equations \eqref{difference-phip}--\eqref{difference-phibar}.

After applying similar calculation as in \eqref{apri-001}, we obtain that
\begin{align}
& \frac{1}{2}\frac{d}{dt} (\| \phi_+ \|_{r,0,\tau}^2 + \|\phi_- \|_{r,0,\tau}^2) + \nu (\|\partial_z \phi_+ \|_{r,0,\tau}^2 + \|\partial_z \phi_- \|_{r,0,\tau}^2)= \dot{\tau} (\|A^{r+\frac{1}{2}} e^{\tau A} \phi_+ \|^2 + \|A^{r+\frac{1}{2}} e^{\tau A} \phi_- \|^2)  \nonumber
\\
\quad &
- \Big\langle \overline{\phi} \cdot \nabla V_+ + \overline{\phi} \cdot \nabla \phi_+ + \overline{V} \cdot \nabla \phi_+  + \frac{1}{2}(\phi_+ \cdot \nabla)(\overline{V}+i\overline{V}^\perp) + \frac{1}{2}(\phi_+ \cdot \nabla)(\overline{\phi}+i\overline{\phi}^\perp)  \nonumber
\\
\quad &
+ \frac{1}{2}(V_+ \cdot \nabla)(\overline{\phi}+i\overline{\phi}^\perp)  + e^{i\Omega t}\Big(Q_{1,+,+}  - Q_{3,+,+}\Big)
+ e^{-i\Omega t} \Big(Q_{1,-,+}  - Q_{3,-,+}\Big) + e^{-2i\Omega t}Q_{4,-,+} , \phi_+  \Big\rangle \nonumber
 \\
\quad &
- \Big\langle \overline{\phi} \cdot \nabla V_- + \overline{\phi} \cdot \nabla \phi_- + \overline{V} \cdot \nabla \phi_-  + \frac{1}{2}(\phi_- \cdot \nabla)(\overline{V}-i\overline{V}^\perp) + \frac{1}{2}(\phi_- \cdot \nabla)(\overline{\phi}-i\overline{\phi}^\perp) \nonumber
\\
\quad &
+ \frac{1}{2}(V_- \cdot \nabla)(\overline{\phi}-i\overline{\phi}^\perp)  + e^{-i\Omega t}\Big(Q_{1,-,-}  - Q_{3,-,-}\Big)
+ e^{i\Omega t} \Big(Q_{1,+,-}  - Q_{3,+,-}\Big) + e^{2i\Omega t}Q_{4,+,-}   , \phi_- \Big\rangle  \nonumber
\\
\quad &
- \underbrace{\Big\langle A^r e^{\tau A} (\overline{\phi}\cdot \nabla V_+), A^r e^{\tau A} \phi_+ \Big\rangle}_\mathrm{Tp2} - \underbrace{\Big\langle A^r e^{\tau A} (\overline{\phi}\cdot \nabla V_-), A^r e^{\tau A} \phi_- \Big\rangle}_\mathrm{Tp2}  \nonumber
\\
\quad &
- \underbrace{\Big\langle A^r e^{\tau A} (\overline{\phi}\cdot \nabla \phi_+), A^r e^{\tau A} \phi_+ \Big\rangle }_\mathrm{Tp1}- \underbrace{\Big\langle A^r e^{\tau A} (\overline{\phi}\cdot \nabla \phi_-), A^r e^{\tau A} \phi_- \Big\rangle}_\mathrm{Tp1} \nonumber
\\
\quad &
- \underbrace{\Big\langle A^r e^{\tau A} (\overline{V}\cdot \nabla \phi_+), A^r e^{\tau A} \phi_+ \Big\rangle}_\mathrm{Tp4} - \underbrace{\Big\langle A^r e^{\tau A} (\overline{V}\cdot \nabla \phi_-), A^r e^{\tau A} \phi_- \Big\rangle}_\mathrm{Tp4} \nonumber
\\
\quad &
- \underbrace{\Big\langle A^r e^{\tau A} (\phi_+\cdot \nabla (\overline{V}+i\overline{V}^\perp)), A^r e^{\tau A} \phi_+ \Big\rangle}_\mathrm{Tp2} - \underbrace{\Big\langle A^r e^{\tau A} (\phi_- \cdot \nabla (\overline{V}-i\overline{V}^\perp)), A^r e^{\tau A} \phi_- \Big\rangle}_\mathrm{Tp2}  \nonumber
\\
\quad &
- \underbrace{\Big\langle A^r e^{\tau A} (\phi_+\cdot \nabla (\overline{\phi}+i\overline{\phi}^\perp)), A^r e^{\tau A} \phi_+ \Big\rangle}_\mathrm{Tp1} - \underbrace{\Big\langle A^r e^{\tau A} (\phi_- \cdot \nabla (\overline{\phi}-i\overline{\phi}^\perp)), A^r e^{\tau A} \phi_- \Big\rangle}_\mathrm{Tp1} \nonumber
\\
\quad &
- \underbrace{\Big\langle A^r e^{\tau A} (V_+\cdot \nabla (\overline{\phi}+i\overline{\phi}^\perp)), A^r e^{\tau A} \phi_+ \Big\rangle}_\mathrm{Tp4} - \underbrace{\Big\langle A^r e^{\tau A} (V_- \cdot \nabla (\overline{\phi}-i\overline{\phi}^\perp)), A^r e^{\tau A} \phi_- \Big\rangle}_\mathrm{Tp4} \nonumber
\\
\quad &
- \underbrace{e^{i\Omega t}\Big(\Big\langle A^r e^{\tau A}(Q_{1,+,+} - Q_{3,+,+}), A^r e^{\tau A} \phi_+ \Big\rangle + \Big\langle A^r e^{\tau A}(Q_{1,+,-} - Q_{3,+,-}), A^r e^{\tau A} \phi_-\Big\rangle\Big) }_{\mathrm{Tp1}, \cdots , \mathrm{Tp5}} \nonumber
\\
\quad &
- \underbrace{e^{-i\Omega t}\Big(\Big\langle A^r e^{\tau A}(Q_{1,-,+} - Q_{3,-,+}), A^r e^{\tau A} \phi_+ \Big\rangle + \Big\langle A^r e^{\tau A}(Q_{1,-,-} - Q_{3,-,-}), A^r e^{\tau A} \phi_- \Big\rangle\Big) }_{\mathrm{Tp1} , \cdots , \mathrm{Tp5}} \nonumber
\\
\quad &
-\underbrace{e^{2i\Omega t}\Big\langle A^r e^{\tau A}Q_{4,+,-}, A^r e^{\tau A} \phi_- \Big\rangle}_{\mathrm{Tp1}, \mathrm{Tp2},\mathrm{Tp4},\mathrm{Tp5}} - \underbrace{e^{-2i\Omega t}\Big\langle A^r e^{\tau A}Q_{4,-,+}, A^r e^{\tau A} \phi_+ \Big\rangle}_{\mathrm{Tp1} ,\mathrm{Tp2},\mathrm{Tp4},\mathrm{Tp5}}, \label{energy-difference-phipm}
\end{align}
and
	\begin{align}
		\frac{1}{2} \frac{d}{dt} \|A^r e^{\tau A} \overline{\phi} \|^2 = & \dot{\tau} \|A^{r+\frac{1}{2}} e^{\tau A} \overline{\phi} \|^2 \nonumber
		\\
		&
		- \underbrace{\Big\langle A^r e^{\tau A} (\overline{\phi}\cdot \nabla \overline{V}), A^r e^{\tau A} \overline{\phi} \Big\rangle}_{\mathrm{Tp2}} - \underbrace{\Big\langle A^r e^{\tau A} (\overline{\phi}\cdot \nabla \overline{\phi}), A^r e^{\tau A} \overline{\phi} \Big\rangle}_\mathrm{Tp1} \nonumber
		\\
		&
		- \underbrace{\Big\langle A^r e^{\tau A} (\overline{V}\cdot \nabla \overline{\phi}), A^r e^{\tau A} \overline{\phi} \Big\rangle}_\mathrm{Tp4}  - \underbrace{e^{2i\Omega t}\Big\langle A^r e^{\tau A}(Q_{1,+,+} + Q_{2,+,+}), A^r e^{\tau A} \overline{\phi}\Big\rangle}_{\mathrm{Tp1} , \mathrm{Tp2}, \mathrm{Tp4},\mathrm{Tp5} } \nonumber
		\\
		&
		- \underbrace{e^{-2i\Omega t}\Big\langle A^r e^{\tau A}(Q_{1,-,-} + Q_{2,-,-}), A^r e^{\tau A} \overline{\phi}\Big\rangle}_{\mathrm{Tp1} , \mathrm{Tp2}, \mathrm{Tp4},\mathrm{Tp5} },  \label{energy-difference-phibar}
	\end{align}
where we have applied Lemmas \ref{lemma-orthogonal-decomposition}--\ref{lemma-projection}. It is easy to verify from \eqref{difference-phibar} and \eqref{difference-ic} that
\begin{equation*}
    \int_{\mathbb{T}^2} \overline{\phi}(\boldsymbol{x},t) \, d\boldsymbol{x} = \int_{\mathbb{T}^2} \overline{\phi}(\boldsymbol{x},t)|_{t=0} \, d\boldsymbol{x}   = 0,
\end{equation*}
and therefore, applying the Poincar\'e inequality yields
\begin{equation}\label{ineq:poincare-1}
    \|\overline{\phi}\| \leq \|A^r e^{\tau A}\overline{\phi}\| \qquad \text{and} \qquad
    \|\overline{\phi}\|_{r,0,\tau} \leq C \|A^r e^{\tau A}\overline{\phi}\|.
\end{equation}

In \eqref{energy-difference-phipm}  and \eqref{energy-difference-phibar}, we have labeled five types of terms by $ \mathrm{Tp1},\cdots, \mathrm{Tp5} $, which we will present the estimates. The rest lower order terms can be estimated in a similar manner and will be omitted.
Temporally, let $ V $ denote $ V_\pm $ and  $ \overline{V} $, and $ \phi $ denote $ \phi_\pm $ and $ \overline{\phi} $. The aforementioned five types of terms are described in the following:
\begin{itemize}
    \item Type 1 (labeled as $ \mathrm{Tp1} $): terms that are trilinear in $\phi$, e.g.,
    	\begin{gather*} e^{\mathrm{j} i \Omega t} \Big\langle A^r e^{\tau A} \bigl( ({\phi}\cdot \nabla) {\phi}\bigr), A^r e^{\tau A} {\phi} \Big\rangle, \qquad e^{\mathrm{j} i \Omega t} \Big\langle A^r e^{\tau A} \bigl( (\nabla \cdot {\phi}) \phi \bigr), A^r e^{\tau A} {\phi} \Big\rangle,
    		\\
    	\text{and}  \qquad  e^{\mathrm{j} i \Omega t}\Big\langle A^r e^{\tau A} \bigl( \int_0^z (\nabla \cdot \phi(\boldsymbol{x},s)) \,ds \partial_z \phi\bigr) , A^r e^{\tau A} \phi  \Big\rangle  , \qquad \mathrm{j} = 0,\pm 1, \pm 2; \end{gather*}
    \item Type 2 (labeled as $ \mathrm{Tp2} $): terms that are bilinear in $\phi$ with no derivative of $\phi$, e.g.,
    \begin{gather*}  e^{\mathrm{j} i \Omega t} \Big\langle A^r e^{\tau A} \bigl( ({\phi}\cdot \nabla) {V}\bigr), A^r e^{\tau A} {\phi} \Big\rangle \qquad  \text{and} \\
    	  e^{\mathrm{j} i \Omega t} \Big\langle A^r e^{\tau A} \bigl( (\nabla \cdot {V}) \phi \bigr), A^r e^{\tau A} {\phi} \Big\rangle,
    	  \qquad \mathrm{j} = 0,\pm 1, \pm 2; \end{gather*}
    \item Type 3 (labeled as $ \mathrm{Tp3} $): terms that are bilinear in $\phi$ and a vertical derivative of $\phi$, e.g., \begin{gather*}   e^{\mathrm{j} i \Omega t}\Big\langle A^r e^{\tau A} \bigl( \int_0^z (\nabla \cdot V(\boldsymbol{x},s)) \,ds \partial_z \phi\bigr) , A^r e^{\tau A} \phi  \Big\rangle  , \qquad \mathrm{j} = 0,\pm 1, \pm 2; \end{gather*}
    \item Type 4 (labeled as $ \mathrm{Tp4} $): terms that are bilinear in $\phi$ and a horizontal derivative of $\phi$, e.g., \begin{gather*} e^{\mathrm{j} i \Omega t} \Big\langle A^r e^{\tau A} \bigl( ({V}\cdot \nabla) {\phi}\bigr), A^r e^{\tau A} {\phi} \Big\rangle, \qquad e^{\mathrm{j} i \Omega t} \Big\langle A^r e^{\tau A} \bigl( (\nabla \cdot {\phi}) V \bigr), A^r e^{\tau A} {\phi} \Big\rangle,
    	\\
    	\text{and}  \qquad  e^{\mathrm{j} i \Omega t}\Big\langle A^r e^{\tau A} \bigl( \int_0^z (\nabla \cdot \phi(\boldsymbol{x},s)) \,ds \partial_z V \bigr) , A^r e^{\tau A} \phi  \Big\rangle  , \qquad \mathrm{j} = 0,\pm 1, \pm 2; \end{gather*}
    \item Type 5 (labeled as $ \mathrm{Tp5} $): terms that are linear in $\phi$, e.g.,
     	\begin{gather*} e^{\mathrm{j} i \Omega t} \Big\langle A^r e^{\tau A} \bigl( ({V}\cdot \nabla) {V}\bigr), A^r e^{\tau A} {\phi} \Big\rangle, \qquad e^{\mathrm{j} i \Omega t} \Big\langle A^r e^{\tau A} \bigl( (\nabla \cdot {V}) V \bigr), A^r e^{\tau A} {\phi} \Big\rangle,
    	\\
    	\text{and}  \qquad e^{\mathrm{j} i \Omega t} \Big\langle A^r e^{\tau A} \bigl( \int_0^z (\nabla \cdot V(\boldsymbol{x},s)) \,ds \partial_z V\bigr) , A^r e^{\tau A} \phi  \Big\rangle  , \qquad \mathrm{j} = \pm 1, \pm 2. \end{gather*}
\end{itemize}

\subsubsection{Estimates of Type 1 -- Type 4 terms}\label{subsec:tp1-4}
We start with Type 1 terms. Applying Lemmas \ref{lemma-type1}--\ref{lemma-type3} yields
\begin{align*}
	|\mathrm{Tp1}| \leq & C_r \int_0^1   \underbrace{\| A^{r+\frac{1}{2}} e^{\tau A} \phi (z) \|_{L^2(\mathbb T^2)}^2}_{L^1 ~ \text{in} ~ z} \underbrace{ \bigl(\| A^{r} e^{\tau A} \phi(z) \|_{L^2(\mathbb T^2)} + \| \phi(z) \|_{L^2(\mathbb T^2)} \bigr)}_{L^\infty ~ \text{in} ~ z }  \,dz \\
	& + C_r \| A^{r+\frac{1}{2}}e^{\tau A}\phi \|^2 \| \partial_z \phi \|_{r,0, \tau}  \leq C_r \|A^{r+\frac{1}{2}}e^{\tau A} \phi\|^2 \|\phi\|_{r,1,\tau},
\end{align*}
where we have used the embedding $ L^\infty_z\hookrightarrow H^1_z  $ in the $ z $-variable and the H\"older inequality. Notice that, for $ \phi = \overline \phi $, the estimate is similar with obvious modification. Therefore, hereafter, unless pointed out explicitly, we omit the estimates in the case of $ \phi = \overline \phi $ and, similarly, $ V = \overline V $.

Similarly, applying Lemma \ref{lemma-banach-algebra} to Types 2 and 3 terms yields
\begin{align*}
	|\mathrm{Tp2}| \leq & C_r \int_0^1 \underbrace{\bigl(  \| A^{r+1}e^{\tau A} V (z) \|_{L^2(\mathbb T^2)} + \| V (z)) \|_{L^2(\mathbb T^2)} \bigr)}_{L^\infty ~ \text{in} ~ z}  \underbrace{ \bigl(\| A^{r}e^{\tau A} \phi (z) \|_{L^2(\mathbb T^2)} + \| \phi (z)) \|_{L^2(\mathbb T^2)} \bigr)^2}_{L^1 ~ \text{in} ~ z} \,dz \\
	 \leq & C_r \| V \|_{r+1, 1, \tau} \| \phi \|_{r,0,\tau}^2 \qquad \qquad \qquad  \text{and}
	 \\
	 |\mathrm{Tp3} | \leq & C_r \int_0^1 \biggl\lbrack \underbrace{\biggl( \int_0^z \| A^{r+1} e^{\tau A} V (s) \|_{L^2(\mathbb T^2)} +  \| V (s) \|_{L^2(\mathbb T^2)} \,ds \biggr)}_{L^\infty ~ \text{in} ~ z } \underbrace{ \biggl( \| A^{r} e^{\tau A} \partial_z \phi (z) \|_{L^2(\mathbb T^2)} + \|\partial_z \phi (z) \|_{L^2(\mathbb T^2)} \biggr)}_{L^2 ~ \text{in} ~ z} \\
	 &  \times \underbrace{\biggl( \| A^{r} e^{\tau A} \phi (z) \|_{L^2(\mathbb T^2)}  +   \|  \phi (z) \|_{L^2(\mathbb T^2)} \biggr)}_{L^2 ~ \text{in} ~ z}  \biggr\rbrack \,dz \leq C_r \| V\|_{r+1,0,\tau} \| \partial_z \phi \|_{r,0,\tau} \|\phi \|_{r,0,\tau} \\
	 \leq & \frac{\nu}{2}  \| \partial_z \phi \|_{r,0,\tau}^2 + { \dfrac{C_{r}}{\nu}} \| V\|_{r+1,0,\tau}^2 \| \phi \|_{r,0,\tau} ^2.
\end{align*}
In order to estimate Type 4 terms, notice that $ \mathrm{Tp4} $ can be written as, with abuse of notations,
\begin{equation*}
	\mathrm{Tp4} =  \mathrm{Tp4}_1 + \mathrm{Tp4}_2,
\end{equation*}
where
\begin{align*}
	& \begin{aligned}
		\mathrm{Tp4}_1 := &  e^{\mathrm j i\Omega t} \Big\langle (V \cdot \nabla) A^r e^{\tau A} \phi, A^r e^{\tau A} \phi \Big\rangle
		 + e^{\mathrm j i\Omega t} \Big\langle  (\nabla \cdot A^r e^{\tau A} \phi) V, A^r e^{\tau A} \phi \Big\rangle \\
		&+ e^{\mathrm j i\Omega t} \Big\langle \int_0^z \bigl(\nabla \cdot A^r e^{\tau A} \phi(s)\bigr) \,ds \partial_z V, A^r e^{\tau A} \phi \Big\rangle,
	\end{aligned}\\
	& \begin{aligned}
		\mathrm{Tp4}_2 := & e^{\mathrm j i\Omega t} \Big\langle A^r e^{\tau A} \bigl( ( V \cdot \nabla ) \phi \bigr) - (V \cdot \nabla) A^r e^{\tau A} \phi, A^r e^{\tau A} \phi \Big\rangle \\
		& + e^{\mathrm j i\Omega t} \Big\langle A^r e^{\tau A} \bigl( ( \nabla \cdot \phi ) V \bigr) - (\nabla \cdot A^r e^{\tau A} \phi) V, A^r e^{\tau A} \phi \Big\rangle \\
		&+ e^{\mathrm j i\Omega t} \Big\langle A^r e^{\tau A} \bigl( \int_0^z( \nabla \cdot  \phi (s) ) \,ds \partial_z V  \bigr) - \int_0^z \bigl(\nabla \cdot A^r e^{\tau A} \phi(s)\bigr) \,ds \partial_z V, A^r e^{\tau A} \phi \Big\rangle.
	\end{aligned}
\end{align*}
Observing from \eqref{energy-difference-phipm} and \eqref{energy-difference-phibar}, only for $ V = V_\pm $, $ \mathrm{Tp4}_1 $ is nontrivial.
Therefore,
after substituting the inequality $ |\alpha|^{\frac{1}{2}} \leq |\beta|^{\frac{1}{2}} + |\xi|^{\frac{1}{2}} $ for $ \alpha + \beta = \xi $ in the Fourier representation of $ \mathrm{Tp4}_1 $ (see the proof of Lemma \ref{lemma-type2} in the appendix), one can obtain that, for any $ \delta \in (0,1) $,
\begin{align*}
	& | \mathrm{Tp4}_1| \leq  \Big|  \Big\langle (A^{\frac{1}{2}} V_\pm \cdot \nabla) A^{r-\frac{1}{2}} e^{\tau A} \phi, A^r e^{\tau A} \phi \Big\rangle    \Big| + \Big|  \Big\langle (V_\pm \cdot \nabla) A^{r-\frac{1}{2}} e^{\tau A} \phi, A^{r+\frac{1}{2}} e^{\tau A} \phi \Big\rangle   \Big| \\
	& \quad + \Big|  \Big\langle  (\nabla \cdot A^{r-\frac{1}{2}} e^{\tau A} \phi)  A^{\frac{1}{2}} V_\pm, A^r e^{\tau A} \phi \Big\rangle \Big| + \Big|  \Big\langle  (\nabla \cdot A^{r-\frac{1}{2}} e^{\tau A} \phi) V_\pm, A^{r+\frac{1}{2}} e^{\tau A} \phi \Big\rangle \Big| \\
	& \quad  + \Big| \Big\langle \int_0^z \bigl(\nabla \cdot A^{r-\frac{1}{2}} e^{\tau A} \phi(s)\bigr) \,ds \partial_z A^{\frac{1}{2}} V_\pm, A^r e^{\tau A} \phi \Big\rangle \Big| + \Big| \Big\langle \int_0^z \bigl(\nabla \cdot A^{r-\frac{1}{2}} e^{\tau A} \phi(s)\bigr) \,ds \partial_z V_\pm, A^{r+\frac{1}{2}} e^{\tau A} \phi \Big\rangle \Big|\\
	& \leq C_r \int_0^1 \underbrace{\bigl( \|A^{\frac{1}{2}}V_\pm(z)\|_{H^{1+\delta}(\mathbb T^2)}+\|V_\pm(z)\|_{H^{1+\delta}(\mathbb T^2)} \bigr)}_{L^\infty ~ \text{in} ~ z}  \underbrace{\| A^{r+\frac{1}{2}} e^{\tau A} \phi (z)  \|_{L^2(\mathbb T^2)}^2}_{L^1 ~ \text{in} ~ z} \,dz \\
	& \quad + C_r \int_0^1 \biggl\lbrack \underbrace{\int_0^z \| A^{r+\frac{1}{2}} e^{\tau A} \phi(s) \|_{L^2(\mathbb T^2)} \,ds}_{L^\infty ~ \text{in} ~ z} \times \underbrace{\| A^{r+\frac{1}{2}} e^{\tau A} \phi(z) \|_{L^2(\mathbb T^2)}}_{L^2 ~ \text{in} ~ z} \\
	& \qquad\qquad \times   \underbrace{\bigl( \| \partial_z A^{\frac{1}{2}} V_\pm (z) \|_{H^{1+\delta}(\mathbb T^2)}  + \|\partial_z V_\pm (z)\|_{H^{1+\delta}(\mathbb T^2)} \bigr)}_{L^2 ~ \text{in} ~ z} \biggr\rbrack
	\,dz  \leq C_r  \| V_\pm \|_{\frac{3}{2}+\delta,1,0} \| A^{r+\frac{1}{2}} e^{\tau A} \phi  \|^2,
\end{align*}
where we have applied the Sobolev embedding inequality and the H\"older inequality. Meanwhile, applying Lemmas \ref{lemma-difference-type1}--\ref{lemma-difference-type4} to $ \mathrm{Tp4}_2 $ yields
\begin{align*}
	|\mathrm{Tp4}_2| \leq & C_r \int_0^1 \| A^r \phi(z) \|_{L^2(\mathbb T^2)}^2 \| A^r V(z) \|_{L^2(\mathbb T^2)} \,dz \\
	& + C_r \tau \int_0^1 \|A^{r+\frac{1}{2}}e^{\tau A} \phi(z)  \|_{L^2(\mathbb T^2)}^2 \|A^{r+\frac{1}{2}}e^{\tau A} V(z)  \|_{L^2(\mathbb T^2)}  \,dz \\
	& + C_r \| A^r \partial_z V \| \| A^r \phi \|^2 + C_r \tau  \| A^{r+\frac{1}{2}}e^{\tau A} \partial_z V \| \| A^{r+\frac{1}{2}} e^{\tau A} \phi \|^2 \\
	\leq & C_r \| V \|_{r,1,\tau} \| \phi \|_{r,0,\tau}^2 + C_r \tau \| V\|_{r+\frac{1}{2},1,\tau} \| A^{r+\frac{1}{2}} e^{\tau A} \phi \|^2.
\end{align*}

\begin{remark}
For the interested readers, we refer to \cite{GILT20} for an alternative estimate of $ \mathrm{Tp4}_1 $, where some cancellations are taking care of. However, in this paper, such cancellations are not necessary and thus omitted.  Notably, the terms $\| V_\pm \|_{\frac{3}{2}+\delta,1,0}$ in the estimate of $\mathrm{Tp4}_1$ is the reason for the requirement \eqref{constraint-tilde-V}.
\end{remark}
\subsubsection{Estimates of Type 5 terms}\label{subsec:tp5}
In this case, $ \mathrm j \neq 0 $ and $ e^{\mathrm j \Omega i t} = \dfrac{1}{\mathrm j \Omega i} \dfrac{d}{dt} e^{\mathrm j \Omega i t} $. Therefore, $ \mathrm{Tp5} $ can be written as, with abuse of notations,
\begin{equation*}
	\mathrm{Tp5} = \dfrac{1}{\Omega} \partial_t  N + \dfrac{1}{\Omega} R,
\end{equation*}
with
\begin{align}
	& \label{def:N} \begin{aligned}
		N := & \dfrac{e^{\mathrm{j}\Omega i t}}{\mathrm j i} \biggl\lbrack  \Big\langle A^r e^{\tau A} \bigl( ( V \cdot \nabla ) V \bigr) , A^r e^{\tau A} \phi  \Big\rangle +  \Big\langle A^r e^{\tau A} \bigl( ( \nabla  \cdot V ) V \bigr) , A^r e^{\tau A} \phi  \Big\rangle \\
		& +  \Big\langle A^r e^{\tau A} \bigl( \int_0^z (\nabla \cdot V (s)) \,ds  \partial_z V \bigr) , A^r e^{\tau A} \phi  \Big\rangle \biggr\rbrack,
	\end{aligned} \\
	& \label{def:R} \begin{aligned}
		R := & \dfrac{e^{\mathrm{j}\Omega i t}}{\mathrm j i} \biggl\lbrack  \underbrace{\partial_t \Big\langle A^r e^{\tau A} \bigl( ( V \cdot \nabla ) V \bigr) , A^r e^{\tau A} \phi  \Big\rangle}_{=:R_1} + \underbrace{\partial_t \Big\langle A^r e^{\tau A} \bigl( ( \nabla  \cdot V ) V \bigr) , A^r e^{\tau A} \phi  \Big\rangle}_{=:R_2} \\
		& + \underbrace{\partial_t \Big\langle A^r e^{\tau A} \bigl( \int_0^z (\nabla \cdot V (s)) \,ds  \partial_z V \bigr) , A^r e^{\tau A} \phi  \Big\rangle}_{=:R_3} \biggr\rbrack.
	\end{aligned}
\end{align}
It is straightforward to check that
\begin{equation}\label{est:N}
N \leq C_r \|V \|_{r,1,\tau} \|V \|_{r+1,0,\tau}\|\phi \|_{r,0,\tau}.
\end{equation}
Meanwhile, one has
\begin{align*}
	R_1 = & 2 \dot\tau \Big\langle A^{r+1} e^{\tau A} \bigl( (V \cdot \nabla ) V\bigr), A^r e^{\tau A} \phi  \Big\rangle + \Big\langle A^r e^{\tau A} \partial_t \bigl( ( V\cdot \nabla ) V \bigr) , A^r e^{\tau A}\phi \Big\rangle \\
	&  +  \Big\langle A^r e^{\tau A}  \bigl( ( V\cdot \nabla ) V \bigr) , A^r e^{\tau A}\partial_t\phi \Big\rangle =: R_{1,1} + R_{1,2} +  R_{1,3}.
\end{align*}
It follows that, thanks to Lemma \ref{lemma-banach-algebra} and similar arguments as in section \ref{subsec:tp1-4},
\begin{equation}\label{est:R-1-1}
	R_{1,1} \leq C_r |\dot \tau| \| V \|_{r+1,1,\tau}  \| V \|_{r+2,0,\tau} \| \phi \|_{r,0,\tau}.
\end{equation}
After applying the Leray projection \eqref{leray-projection} to \eqref{vbarlimit}, together with \eqref{uplimit} and \eqref{umlimit},
for $ V = V_\pm $ or $ \overline V$, one has
\begin{equation}\label{eq:dt-V}
	\partial_t V - \underbrace{\nu \partial_{zz}V}_{\text{for} ~ V = V_\pm} = \mathcal B( V, \nabla V ).
\end{equation}
Here we use $ \mathcal B $ to represent a generic bilinear term with respect to both of its arguments.
With such notations, after applying integration by parts, one can derive
\begin{equation}\label{est:R-1-2}
\begin{aligned}
	R_{1,2}  = &  - 2 \nu \Big\langle A^r e^{\tau A} \bigl( (\partial_z V \cdot \nabla) \partial_z V       \bigr)  , A^r e^{\tau A}  \phi  \Big\rangle - \nu \Big\langle A^r e^{\tau A} \bigl( (\partial_z V \cdot \nabla) V +( V \cdot \nabla) \partial_z V      \bigr)  , A^r e^{\tau A}  \partial_z \phi  \Big\rangle
		\\
	& + \Big\langle A^r e^{\tau A} \bigl(  (\mathcal B(V,\nabla V) \cdot \nabla) V + (V \cdot \nabla) \mathcal B(V, \nabla V)   \bigr)  , A^r e^{\tau A}  \phi  \Big\rangle
	 \\
	\leq & C_{r}{(1+\nu)}\Big( \| V \|_{r,1,\tau} \| V \|_{r+1,1,\tau}  \|\phi \|_{r,1,\tau} + \| V \|_{r,0,\tau} \| V \|_{r+1,1,\tau}^2 \|\phi\|_{r,0,\tau}
	+ \| V\|_{r+1,1,\tau}^2 \| V\|_{r+2,0,\tau}  \|\phi\|_{r,0,\tau}\Big),
\end{aligned}
\end{equation}
where we have applied Lemma \ref{lemma-banach-algebra} and similar arguments as in section \ref{subsec:tp1-4}.
Similarly, according to \eqref{difference-phip}--\eqref{difference-phibar}, for $ \phi = \phi_\pm $ or $ \overline\phi $, one has, with abuse of notations
\begin{equation}\label{eq:dt-phi}
	\partial_t \phi - \underbrace{\bigl( \nu \partial_{zz} \phi + e^{\mathrm{j}\Omega i t }(\int_0^z \nabla\cdot  (\phi + V)(s) \,ds) \partial_z (\phi + V) \bigr)}_{\text{for} \quad  \phi = \phi_\pm \quad \text{and} \quad  V = V_\pm } = \sum_{ \mathrm A, \mathrm B \in \lbrace \phi,  V \rbrace}\mathcal B( \mathrm A, \nabla \mathrm B ).
\end{equation}
Therefore, $ R_{1,3} $ can be estimated as
\begin{equation}\label{est:R-1-3}
	\begin{aligned}
		R_{1,3} = & - \nu \Big\langle A^r e^{\tau A} \bigl( (\partial_z V \cdot \nabla) V + (V \cdot \nabla) \partial_z V  \bigr), A^r e^{\tau A} \partial_z \phi  \Big\rangle  \\
		& - e^{\mathrm j\Omega i t} \Big\langle A^{r+1} e^{\tau A} \bigl( (V \cdot \nabla ) V \bigr), A^{r-1} e^{\tau A} \bigl( (\int_0^z \nabla \cdot (\phi + V)(s) \,ds) \partial_z ( \phi + V)  \bigr) \Big\rangle \\
		& - \sum_{\mathrm A, \mathrm B \in \lbrace \phi, V \rbrace} \Big\langle A^{r+1} e^{\tau A} \bigl( (V \cdot \nabla ) V \bigr), A^{r-1} e^{\tau A} \mathcal B( \mathrm A,\nabla \mathrm B ) \Big\rangle \\
		\leq  & C_{r}{\nu} \| V \|_{r,1,\tau} \| V\|_{r+1,1,\tau} \|\partial_z \phi\|_{r,0,\tau}  \\
		& + C_r \| V \|_{r+1,1,\tau} \| V \|_{r+2,0,\tau} (\| \phi \|_{r-1,1,\tau} + \| V\|_{r-1,1,\tau}  )( \| \phi \|_{r,0,\tau} + \| V \|_{r,0, \tau}  ).
	\end{aligned}
\end{equation}
The estimate of $ R_2 $ is the same as $ R_1 $ (see \eqref{est:R-1-1}, \eqref{est:R-1-2}, and \eqref{est:R-1-3}). To estimate $ R_3 $, one has, after applying integration by parts,
\begin{align*}
	R_3 = & 2 \dot\tau \Big\langle A^{r+1} e^{\tau A} \bigl( \int_0^z (\nabla \cdot V (s) ) \,ds \partial_z V \bigr), A^r e^{\tau A} \phi \Big\rangle - \Big\langle A^{r} e^{\tau A} \partial_t \bigl( (\nabla \cdot V ) V \bigr), A^r e^{\tau A} \phi \Big\rangle \\
	&  - \Big\langle A^{r} e^{\tau A} \partial_t \bigl( \int_0^z (\nabla \cdot V (s) ) \,ds  V \bigr), A^r e^{\tau A} \partial_z\phi \Big\rangle + \Big\langle A^r e^{\tau A} \bigl( \int_0^z (\nabla \cdot V (s) ) \,ds \partial_z V \bigr), A^r e^{\tau A} \partial_t \phi \Big\rangle\\
	 =: &  R_{3,1} + R_{3,2} + R_{3,3} + R_{3,4}.
\end{align*}
As before,
\begin{equation}\label{est:R-3-1}
	R_{3,1} \leq C_r |\dot\tau| \| V \|_{r+2,0,\tau}  \| V\|_{r+1,1,\tau} \| \phi \|_{r,0,\tau}.
\end{equation}
The estimate of $ R_{3,2} $ is the same as that of $ R_{1,2} $ in \eqref{est:R-1-2}. Meanwhile,
substituting representation \eqref{eq:dt-V} in $ R_{3,3} $ leads to
\begin{equation}\label{est:R-3-3}
	\begin{aligned}
		R_{3,3} = & - \Big\langle A^r e^{\tau A}  \bigl( \int_0^z (\nabla \cdot \partial_t V (s)) \,ds V \bigr), A^r e^{\tau A} \partial_z \phi \Big\rangle - \Big\langle A^r e^{\tau A}  \bigl( \int_0^z (\nabla \cdot  V (s)) \,ds \partial_t V \bigr), A^r e^{\tau A} \partial_z \phi \Big\rangle \\
		= & - \Big\langle A^r e^{\tau A}  \bigl( \int_0^z (\nabla \cdot (\nu \partial_{zz}V + \mathcal B(V,\nabla V) ) (s)) \,ds V \bigr), A^r e^{\tau A} \partial_z \phi \Big\rangle \\
		&  - \Big\langle A^r e^{\tau A}  \bigl( \int_0^z (\nabla \cdot  V (s)) \,ds (\nu \partial_{zz}V + \mathcal B(V,\nabla V) ) \bigr), A^r e^{\tau A} \partial_z \phi \Big\rangle\\
		\leq & C_{r} \bigl(  {\nu} \|V\|_{r+1,0,\tau} \|V\|_{r,2,\tau} + \|V\|_{r+1,0,\tau} \| V\|_{r,1,\tau} \|V\|_{r+2,0,\tau}   \bigr) \| \partial_z \phi \|_{r,0,\tau}.
	\end{aligned}
\end{equation}
After substituting \eqref{eq:dt-phi}, $ R_{3,4} $ can be estimated as
\begin{equation}\label{est:R-3-4}
	\begin{aligned}
		R_{3,4} = & - \nu \Big\langle A^r e^{\tau A} \bigl( \int_0^z (\nabla \cdot V (s) ) \,ds \partial_{zz} V \bigr), A^r e^{\tau A} \partial_z \phi \Big\rangle - \nu  \Big\langle A^r e^{\tau A} \bigl( (\nabla \cdot V )  \partial_z V \bigr), A^r e^{\tau A} \partial_z \phi \Big\rangle \\
		& -  e^{\mathrm j \Omega i t } \Big\langle A^{r+1} e^{\tau A} \bigl( \int_0^z (\nabla \cdot V (s) ) \,ds \partial_z V \bigr), A^{r-1} e^{\tau A} \bigl\lbrack \bigl( \int_0^z \nabla \cdot (\phi + V )(s) \,ds \bigr) \partial_z (\phi + V ) \bigr\rbrack \Big\rangle \\
		& -  \sum_{\mathrm A, \mathrm B \in \lbrace \phi, V \rbrace }\Big\langle A^{r+1} e^{\tau A} \bigl( \int_0^z (\nabla \cdot V (s) ) \,ds \partial_z V \bigr), A^{r-1} e^{\tau A} \mathcal B(\mathrm A, \nabla \mathrm B) \Big\rangle \\
		\leq & C_{r} {\nu}  \| V \|_{r+1,0,\tau} \| V\|_{r,2,\tau} \| \partial_z \phi \|_{r,0,\tau} \\
		& + C_r \| V \|_{r+2,0,\tau} \| V \|_{r+1,1,\tau}  (\|\phi \|_{r,0,\tau} + \|V \|_{r,0,\tau} )(\|\phi\|_{r-1,1,\tau} + \|V \|_{r-1,1,\tau} ).
	\end{aligned}
\end{equation}
We emphasize that, in the estimates above, we do not distinguish $ V_\pm $ and $ \overline V $, $ \phi_\pm $  and $ \overline\phi $, i.e., we treat all $ V $ and $ \phi $ as if they are three-dimensional. The estimates in the case when they are two-dimensional are similar with obvious modifications, and thus omitted.
Consequently, combining \eqref{est:R-1-1}--\eqref{est:R-3-4} leads to the estimate of $ R $.

\subsubsection{Finishing of proof of Theorem \ref{theorem-main}}
Without loss of generality, we assume $ | \Omega | >  1 $.
Combining the estimates in subsections \ref{subsec:tp1-4} and \ref{subsec:tp5}, from \eqref{energy-difference-phipm} and \eqref{energy-difference-phibar}, yields, thanks to \eqref{ineq:poincare-1} and the Young inequality,
\begin{equation}\label{main-estimate-1}
    \begin{aligned}
        &\frac{d}{dt} F  + \nu H
        \leq  \Big[ \dot{\tau} + C_r K^{\frac{1}{2}} \tau  + C_r\big( \| V_+\|_{\frac{3}{2}+\delta,1,0} + \|  V_-\|_{\frac{3}{2}+\delta,1,0}  \big) \\
        & \qquad \qquad  + C_r F^{\frac{1}{2}} + C_r H^{\frac{1}{2}}  \Big]
        \times G
         + C_{r}{(1+\nu+\dfrac{1}{\nu})}\Big( K^2+ 1 \Big) F
        \\
        & \quad + \frac{C_{r}({1+\nu})}{|\Omega|}  K H
         + \frac{C_{r}}{|\Omega|} \Big( {\nu \|\partial_z V_+\|_{r,1,\tau} + \nu \|\partial_z V_-\|_{r,1,\tau}} \Big) K^{\frac{1}{2}} H^{\frac{1}{2}}
        \\
        & \quad + \frac{C_{r}{(1+\nu)}}{|\Omega|}\Big( |\dot{\tau}|^2 + K^2 + 1 \Big) +  \dfrac{{1}}{|\Omega|} \partial_t N.
    \end{aligned}
\end{equation}
where $ \delta \in (0,\frac{1}{2}) $ and
\begin{align}
	F:= &  \|A^r e^{\tau A} \overline{\phi} \|^2 + \|\phi_+ \|_{r,0,\tau}^2 + \|\phi_- \|_{r,0,\tau}^2 , \label{F} \\
	G := &  \|A^{r+\frac{1}{2}} e^{\tau A} \overline{\phi} \|^2 + \|A^{r+\frac{1}{2}} e^{\tau A} \phi_+ \|^2 + \|A^{r+\frac{1}{2}} e^{\tau A} \phi_- \|^2 , \\
	H:= &  \|\partial_z \phi_+ \|_{r,0,\tau}^2 + \|\partial_z \phi_- \|_{r,0,\tau}^2, \label{H}\\
	K:= & \|\overline V\|_{r+2,0,\tau}^2 + \|V_+\|_{r+2,0,\tau}^2 + \|V_-\|_{r+2,0,\tau}^2 +  \|V_+\|_{r+1,1,\tau}^2+ \|V_-\|_{r+1,1,\tau}^2 \label{K} .
\end{align}

Assume that, for the moment, we have
\begin{equation}\label{eq:tau-main}
	\dot{\tau} + C_r K^{\frac{1}{2}} \tau  + C_r\big(  \| V_+\|_{\frac{3}{2}+\delta,1,0} + \|  V_-\|_{\frac{3}{2}+\delta,1,0}  \big)  + C_r F^{\frac{1}{2}} + C_r H^{\frac{1}{2}} = 0,
\end{equation}
which implies $ \tau \leq \tau_0 $ and
\begin{equation*}
	|\dot \tau|^2 \leq C_r (\tau_0^2 + 1) K + C_r (F + H).
\end{equation*}
On the other hand, recalling $ M $ as in \eqref{main-bound-initial},
then according to Proposition \ref{global-limit}, \eqref{def:limit-V-tilde}, and \eqref{apri:107}, there exist $ C_{M}, C_r > 1 $ such that
{
\begin{align}
		& \label{growing-1}
		K + \int_0^t \Big(\nu \|\partial_z V_+ (s) \|_{r,1,\tau}^2 + \nu \|\partial_z V_-(s) \|_{r,1,\tau}^2 \Big) \,ds
		\leq \exp[ \exp[\exp(C_r t + C_{M})] ] =: \mathcal K(t) , \\
		& \qquad  \text{and} \nonumber \\
		& \label{growing-2}
		\int_0^t \bigl(  \| V_+(s) \|_{\frac{3}{2}+\delta,1,0}^2 +  \|  V_-(s) \|_{\frac{3}{2}+\delta,1,0}^2 \bigr)  \,ds \leq (1+\dfrac{1}{\nu}) \| \widetilde{V}_0 \|_{\frac{3}{2}+\delta, 0, 0}^2 \mathcal K(t). 
\end{align}
}
Under these conditions, from \eqref{main-estimate-1}, one can derive that
{
\begin{equation}\label{main-estimate-2}
     \begin{aligned}
    	&\frac{d}{dt} F  + \dfrac{\nu}{2} H
    	\leq  C_{r}(1+\nu+\dfrac{1}{\nu})\Big( \mathcal K^2+ 1 \Big) F + \frac{C_{r}(1+\nu)}{|\Omega|}  (\mathcal K + 1 ) H
    	+ \dfrac{C_{r}}{|\Omega|^2} \Big(\nu \|\partial_z V_+\|_{r,1,\tau}^2 + \nu \|\partial_z V_-\|_{r,1,\tau}^2 \Big) \mathcal K
    	\\
    	& \quad + \frac{C_{r}(1+\nu)}{|\Omega|}\Big( \mathcal K^2 + \tau_0^4  + 1 \Big) +  \dfrac{1}{|\Omega|} \partial_t N.
    \end{aligned}
\end{equation}
Therefore, multiplying \eqref{main-estimate-2} with $ e^{- C_{r}(1+\nu+\frac{1}{\nu})\int_0^t (\mathcal K^2 + 1)(s) \,ds} $ leads to
\begin{align*}
	& \dfrac{d}{dt} \bigl( F e^{- C_{r}(1+\nu+\frac{1}{\nu})\int_0^t (\mathcal K^2 + 1)(s) \,ds} \bigr) + \bigl[ \dfrac{\nu}{2} - \frac{C_{r}(1+\nu)}{|\Omega|}  (\mathcal K + 1 )  \bigr] H e^{- C_{r}(1+\nu+\frac{1}{\nu})\int_0^t (\mathcal K^2 + 1)(s) \,ds} \\
	& \leq \dfrac{C_{r}}{|\Omega|^2} \Big(\nu \|\partial_z V_+\|_{r,1,\tau}^2 + \nu \|\partial_z V_-\|_{r,1,\tau}^2 \Big) \mathcal K e^{- C_{r}(1+\nu+\frac{1}{\nu})\int_0^t (\mathcal K^2 + 1)(s) \,ds}
    	\\
    & \quad + \frac{C_{r}(1+\nu)}{|\Omega|}\Big( \mathcal K^2 + \tau_0^4  + 1 \Big) e^{- C_{r}(1+\nu+\frac{1}{\nu})\int_0^t (\mathcal K^2 + 1)(s) \,ds} +  \dfrac{1}{|\Omega|} \partial_t N \times e^{- C_{r}(1+\nu+\frac{1}{\nu})\int_0^t (\mathcal K^2 + 1)(s) \,ds}.
\end{align*}
}
Integrating the above equation in time and recalling that $F(t=0)=0$, one obtains
{
\begin{equation}\label{est:301}
\begin{aligned}
	& \bigl( F(t) e^{- C_{r}(1+\nu+\frac{1}{\nu})\int_0^t (\mathcal K^2 + 1)(s) \,ds} \bigr) + \int_0^t \bigl[ \dfrac{\nu}{2} - \frac{C_{r}(1+\nu)}{|\Omega|}  (\mathcal K (t') + 1 )  \bigr] H (t') e^{- C_{r}(1+\nu+\frac{1}{\nu})\int_0^{t'} (\mathcal K^2 + 1)(s) \,ds} \,dt' \\
	& \leq \int_0^t \dfrac{C_{r}}{|\Omega|^2} \Big(\nu \|\partial_z V_+(t') \|_{r,1,\tau}^2 + \nu \|\partial_z V_-(t') \|_{r,1,\tau}^2 \Big) \mathcal K e^{- C_{r}(1+\nu+\frac{1}{\nu})\int_0^{t'} (\mathcal K^2 + 1)(s) \,ds} \,dt' \\
	& \quad + \int_0^t \frac{C_{r}(1+\nu)}{|\Omega|}\Big( \mathcal K^2 (t') + \tau_0^4  + 1 \Big) e^{- C_{r}(1+\nu+\frac{1}{\nu})\int_0^{t'} (\mathcal K^2 + 1)(s) \,ds} \,dt' \\
	& \quad  +  \int_0^t  \dfrac{1}{|\Omega|} \partial_t N(t') e^{- C_{r}(1+\nu+\frac{1}{\nu})\int_0^{t'} (\mathcal K^2 + 1)(s) \,ds} \,dt' \\
	& \leq \dfrac{C_{r}}{|\Omega|^2 } \mathcal K(t) + \dfrac{C_{r}(1+\nu)}{|\Omega|} \int_0^t ( \mathcal K(t') + \tau_0^4 + 1 ) \,dt' +  \dfrac{1}{|\Omega|} \int_0^t  \partial_t N(t') e^{- C_{r}(1+\nu+\frac{1}{\nu})\int_0^{t'} (\mathcal K^2 + 1)(s) \,ds} \,dt',
\end{aligned}
\end{equation}
}
where we have applied \eqref{growing-1} and, thanks to the definition of $ \mathcal K $,
\begin{equation}\label{est:e-ineq}
\mathcal K(t') e^{- C_{r}(1+\nu+\frac{1}{\nu})\int_0^{t'} (\mathcal K^2 + 1)(s) \,ds} < C,
\end{equation}
for some constant $ C \in (0,\infty) $. On the other hand, thanks to \eqref{est:N}, \eqref{growing-1}, and \eqref{est:e-ineq}, since $N(t=0)=0$, one can derive that
\begin{equation}\label{N}
    \begin{split}
        & \int_0^t  \partial_t N(t') e^{- C_{r}(1+\nu+\frac{1}{\nu})\int_0^{t'} (\mathcal K^2 + 1)(s) \,ds} \,dt' = N(t) e^{- C_{r}(1+\nu+\frac{1}{\nu})\int_0^{t} (\mathcal K^2 + 1)(s) \,ds} \\
    & \qquad + C_{r}(1+\nu + \frac{1}{\nu}) \int_0^t N(t') (\mathcal K^2(t') + 1) e^{- C_{r}(1+\nu+\frac{1}{\nu})\int_0^{t'} (\mathcal K^2 + 1)(s) \,ds} \,dt'\\
    & \leq \mathcal K(t) e^{- C_{r}(1+\nu+\frac{1}{\nu})\int_0^{t} (\mathcal K^2 + 1)(s) \,ds} F^{\frac{1}{2}}(t) + C_{r}(1+\nu+\dfrac{1}{\nu}) \int_0^t (\mathcal K^2 (t') + 1 ) \mathcal K(t') F^{\frac{1}{2}}(t') e^{- C_{r}(1+\nu+\frac{1}{\nu})\int_0^{t'} (\mathcal K^2 + 1)(s) \,ds} \,dt' \\
    & \leq C_{r}F^{\frac{1}{2}}(t) + C_{r}(1+\nu + \frac{1}{\nu}) \int_0^t (\mathcal K^2 (t') + 1 )F^{\frac{1}{2}}(t') \,dt'.
    \end{split}
\end{equation}
Hence, \eqref{est:301} implies that, for $ t \in [0,\mathcal T] $, since $|\Omega|>1$, after applying the young inequality,
{
\begin{equation}\label{est:302}
    \begin{aligned}
    & F(t) + \dfrac{\nu}{4}\int_0^t H (t')  \,dt' \leq \dfrac{C_{r}}{|\Omega|} \mathcal K(t) e^{C_{r}(1+\nu+\frac{1}{\nu})\int_0^t (\mathcal K^2 + 1)(s) \,ds} \\
    & \qquad + \dfrac{C_{r}(1+\nu+\frac{1}{\nu})^2}{|\Omega|} \biggl( \int_0^t \bigl( \mathcal K^2(t') + 1 \bigr) \,dt'\biggr)^2 e^{C_{r}(1+\nu+\frac{1}{\nu})\int_0^t (\mathcal K^2 + 1)(s) \,ds} \\
    & \qquad + \dfrac{C_{r}(1+\nu)}{|\Omega|} \int_0^t \bigl( \mathcal K(t') + \tau_0^4 + 1 \bigr) \,dt' \times e^{C_{r}(1+\nu+\frac{1}{\nu})\int_0^t (\mathcal K^2 + 1)(s) \,ds},
    \end{aligned}
\end{equation}
}
where $ \mathcal T \in (0,\infty] $ is given by the following constraints:
\begin{equation}\label{constraint-2}
\tau(s) > 0 \qquad \text{and} \qquad
\dfrac{\nu}{2} - \dfrac{C_{r}(1+\nu)}{|\Omega|}(\mathcal K(s) + 1 ) \geq \dfrac{\nu}{4} > 0 \qquad \text{for} \quad s \in [0,\mathcal T].
\end{equation}
Since $|\Omega|\geq |\Omega_0|$, in particular, there exists a constant $ \mathcal C_{M,r,\tau_0}  \in (1,\infty) $ such that, for $ t \in (0,\mathcal T] $,
{
\begin{equation}\label{est:304}
    F(t) + \dfrac{\nu}{4}\int_0^t H(t') \,dt' \leq \dfrac{1}{|\Omega_0|}(1+\nu+\dfrac{1}{\nu})^2 \exp\biggl[(1+\nu+\frac{1}{\nu})\exp[\exp[\exp[\mathcal C_{M,r,\tau_0}(t+1)]]]\biggr] .
\end{equation}
}

Now we will be able to estimate $ \mathcal T $. To ensure $ \tau > 0 $ in \eqref{constraint-2}, from \eqref{eq:tau-main}, \eqref{growing-1}, \eqref{growing-2}, and \eqref{est:304}, one has
{
\begin{equation}\label{est:tau}
    \begin{aligned}
    & \tau(t) =  - C_r \int_0^t e^{-C_r \int_{t'}^t K^{\frac{1}{2}}(s)\,ds} \bigl(\|V_+\|_{3/2+\delta, 1,0} +\|V_-\|_{3/2+\delta, 1,0} + F^{\frac{1}{2}} + H^{\frac{1}{2}}  \bigr) \,dt'\\
    & \qquad + \tau_0 e^{-C_r \int_0^t K^{\frac{1}{2}}(t')\,dt'} \geq \tau_0 \exp[\exp[\exp[\exp[-\mathcal C'_{M,r}(t+1)]]]] \\
    & \qquad\qquad - C_r \biggl((1+\dfrac{1}{\nu^{1/2}})\|\widetilde V_0 \|_{\frac{3}{2}+\delta, 0,0} + \dfrac{1}{|\Omega_0|^{1/2}} (1 + \nu + \dfrac{1}{\nu^{3/2}}) \biggr)\exp\biggl[(1+\nu+\frac{1}{\nu})\exp[\exp[\exp[C_{M,r,\tau_0}(t+1)]]] \biggr]
    \end{aligned}
\end{equation}
}
for some constant $ \mathcal C'_{M,r}, \mathcal C'_{M,r,\tau_0} \in (1,\infty) $. Notably, the function $\tau(t)$ we obtain is bounded above by \eqref{tau-limit-system}.
Therefore, for $ t > 0 $ satisfying
{
\begin{equation}\label{est:305}
\begin{gathered}
    \exp\biggl[(1+\nu+\frac{1}{\nu})\exp[\exp[\exp[ \mathcal C'_{M,r,\tau_0}(t+1)]]] -\exp[\exp[\exp[-C_{M,r}'(t+1)]]] \biggr]\\
    < \dfrac{\tau_0}{2 C_r \bigl((1+\dfrac{1}{\nu^{1/2}})\|\widetilde V_0 \|_{\frac{3}{2}+\delta, 0,0} + \frac{1}{|\Omega_0|^{1/2}} (1 + \nu + \dfrac{1}{\nu^{3/2}}) \bigr)},
\end{gathered}
\end{equation}
or equivalently, under the assumption of \eqref{constraint-tilde-V}, for some $ C_{M,r,\tau_0}'' \in (1,\infty) $,
\begin{equation}
    \exp\biggl[ (1+\nu+\frac{1}{\nu})\exp[\exp[\exp[ \mathcal C''_{M,r,\tau_0}(t+1)]]] -\exp[\exp[\exp[-C_{M,r,\tau_0}''(t+1)]]] \biggr] \\
    < \dfrac{C_{M,r,\tau_0}''|\Omega_0|^{1/2}}{1+\nu + \frac{1}{\nu^{3/2}}},
\end{equation}
}
it follows that $ \tau(t) > 0 $.
{
In particular, for $ t \in (0,T] $ with
\begin{equation}
T = \dfrac{1}{C''_{M,r,\tau_0}}\log\left[\frac1{e^{C''_{M,r,\tau_0}}}\log\left[\log\left[\dfrac{\log \left(\frac{C_{M,r,\tau_0}''|\Omega_0|^{1/2}}{1+\nu + \frac{1}{\nu^{3/2}}}\right)
}{1+\nu+\frac{1}{\nu}}\right]\right]\right],
\end{equation}
the above inequality is satisfied.
}

Consequently, under condition \eqref{constraint-tilde-V}, \eqref{constraint-2} and \eqref{est:305} imply \eqref{constraint-1}, and \eqref{est:304} implies \eqref{main-regularity} thanks to \eqref{def:perturbation}.
This completes the proof of Theorem \ref{theorem-main}.


\subsection{Proof of Theorem \ref{theorem-small-barotropic}}\label{section-smallbaro}

In this section, we prove Theorem \ref{theorem-small-barotropic}. We only sketch the proof for the first two parts, and will provide detailed proof for the third part.

For the first part of the theorem, thanks to Remark \ref{special-situation-limit}, we know that when $ \sup_{0\leq t< \infty }\|\overline{V}(t)\|_{r+3,0,\tau(t)} \leq C_{M,r}$ the growth of $\|\widetilde{V}(t)\|_{r+2,1,\tau(t)}$ will only be exponentially in time. Thus, the function $\mathcal K(t)$ appears in the proof of Theorem \ref{theorem-main} (e.g., \eqref{growing-1} and \eqref{est:302}) becomes only exponentially in time. This reduces two logarithms in the estimate of existence time and gives
\begin{equation*}
    \mathcal T = \frac{1}{C_{\tau_0, M, r, \nu} } \log(\log (|\Omega_0|) ).
\end{equation*}
This can be seen as in \eqref{constraint-2} -- \eqref{est:305}.

Similarly, for the second part of Theorem \ref{theorem-small-barotropic}, thanks to Remark \ref{special-situation-limit}, when $\sup_{0\leq t < \infty }\|\overline{V}(t)\|_{r+3,0,\tau} \leq \frac{\nu}{4C_{r,\alpha}}$ is small enough $\|\widetilde{V}(t)\|_{r+2,1,\tau(t)}$ does not grow and thus the function $\mathcal K(t)$ is uniformly-in-time bounded. This reduces one more logarithm and gives
\begin{equation*}
    \mathcal T = \frac{1}{C_{\tau_0, M, r, \nu} } \log (|\Omega_0|) ).
\end{equation*}
To show that the smallness condition \eqref{constraint-tilde-V} can be relaxed, recalling $ K $ in \eqref{K}. Under our new assumption on $\overline{\mathcal V}$, thanks to Remark \ref{special-situation-limit}, we have that $K^{\frac{1}{2}} \leq \frac{\nu}{C_{r,\alpha}} + C_M e^{-\frac{\nu}{2}t}$ and $\|V_+(t)\|_{3/2+\delta, 1,0} +\|V_-(t)\|_{3/2+\delta, 1,0} \leq \frac{\tau_0}{C_{r,\nu,M}} e^{-\frac{\nu}{2}t}$. Now recall from \eqref{est:tau} that
\begin{equation*}
    \begin{aligned}
    \tau(t) =  \Big(\tau_0 - C_r \int_0^t e^{C_r \int_0^{t'} K^{\frac{1}{2}}(s)\,ds} \bigl(\|V_+\|_{3/2+\delta, 1,0} +\|V_-\|_{3/2+\delta, 1,0} + F^{\frac{1}{2}} + H^{\frac{1}{2}}  \bigr) \,dt' \Big) e^{-C_r \int_0^t K^{\frac{1}{2}}(t')\,dt'},
    \end{aligned}
\end{equation*}
in which we will ask for
\begin{gather*}
    \tau_0 - C_r \int_0^t e^{C_r \int_0^{t'} K^{\frac{1}{2}}(s)\,ds} \bigl(\|V_+\|_{3/2+\delta, 1,0} +\|V_-\|_{3/2+\delta, 1,0}\bigr) dt'
    \\
    \geq \tau_0 - C_r \int_0^\infty C_{M,\nu} \frac{\tau_0}{C_{r,\nu,M}} e^{\frac{C_r}{C_{r,\alpha}} \nu t' - \frac{\nu}{2} t'}  dt' \geq \frac{\tau_0}{2},
\end{gather*}
provided that $C_{r,\nu,M}$ and $C_{r,\alpha}$ are large enough. From this, one can conclude that the smallness assumption can be relaxed and replaced by $\|\widetilde{\mathcal{V}}_0\|_{\frac{3}{2}+\delta,0,0} \leq  \frac{\tau_0}{C_{r,\nu,M}}$.

Next we give the detailed proof to the third part of Theorem \ref{theorem-small-barotropic}. Consider the initial data satisfying $\|\overline{\mathcal V}_0\|_{r+3,0,\tau_0} \leq \frac{M}{|\Omega|_0}$. We set $\overline{V} = 0$ and replace the initial condition \eqref{difference-ic} of the perturbed system to
\begin{equation*}
    \overline{\phi}_0 = \overline{\mathcal{V}}_0, \quad (\phi_\pm)_0 = 0.
\end{equation*}
With more careful estimates, \eqref{main-estimate-1} becomes
\begin{equation}\label{main-estimate-small-1}
    \begin{aligned}
        &\frac{d}{dt} F  + \nu H
        \leq  \Big[ \dot{\tau} + C_r K^{\frac{1}{2}} \tau  + C_r\big( \| V_+\|_{\frac{3}{2}+\delta,1,0} + \|  V_-\|_{\frac{3}{2}+\delta,1,0}  \big) \\
        & \qquad \qquad  + C_r F^{\frac{1}{2}} + C_r H^{\frac{1}{2}}  \Big]
        \times G
         + C_{r,\nu}L F
        \\
        & \quad + \frac{C_{r,\nu}}{|\Omega|}  K H
         + \frac{C_{r,\nu}}{|\Omega|}    \Big(\|\partial_z V_+\|_{r,1,\tau} + \|\partial_z V_-\|_{r,1,\tau} \Big) K^{\frac{1}{2}} H^{\frac{1}{2}}
        \\
        & \quad + \frac{C_{r,\nu,\tau_0}}{|\Omega|} L  +  \dfrac{C_{r, \nu}}{|\Omega|} \partial_t N,
    \end{aligned}
\end{equation}
where $ \delta \in (0,\frac{1}{2}) $ and $F,G,H$ are defined as in \eqref{F}--\eqref{H},
\begin{equation*}
    K:=  \|V_+\|_{r+2,0,\tau}^2 + \|V_-\|_{r+2,0,\tau}^2 +  \|V_+\|_{r+1,1,\tau}^2+ \|V_-\|_{r+1,1,\tau}^2, \quad L:= K^{\frac{1}{2}} + K + K^2,
\end{equation*}
and
\begin{equation}\label{constraint-tau-small}
    \dot{\tau} + C_r K^{\frac{1}{2}} \tau  + C_r\big( \| V_+\|_{\frac{3}{2}+\delta,1,0} + \|  V_-\|_{\frac{3}{2}+\delta,1,0}  \big)  + C_r F^{\frac{1}{2}} + C_r H^{\frac{1}{2}}=0.
\end{equation}
On the other hand, thanks to Remark \ref{special-situation-limit}, \eqref{def:limit-V-tilde}, and \eqref{apri:107}, there exist $ C_{M,\nu}, C_r, C > 1 $ such that
\begin{gather}
	 \label{growing-small-1}
		L
		\leq C_{M} e^{-\frac{\nu}{C}t} =: \mathcal K(t) , \\
		\nu\int_0^t \Big(\|\partial_z V_+ (s) \|_{r,1,\tau}^2 + \|\partial_z V_-(s) \|_{r,1,\tau}^2 \Big) e^{\nu s}\,ds \leq C_M
		 \qquad  \text{and}  \\
		 \label{growing-small-2}
		  \nu\int_0^t \bigl( \| V_+(s) \|_{\frac{3}{2}+\delta,1,0}^2 + \|  V_-(s) \|_{\frac{3}{2}+\delta,1,0}^2 \bigr) e^{\nu s} \,ds \leq  C\| \widetilde{V}_0 \|_{\frac{3}{2}+\delta, 0, 0}^2 .
\end{gather}
With these conditions, from \eqref{main-estimate-small-1}, one can derive that
\begin{equation*}
     \begin{aligned}
    	&\frac{d}{dt} F  + \nu H
    	\leq  C_{r,\nu} L F + \frac{C_{r,\nu}}{|\Omega|}  (\mathcal K + 1 ) H
    	+ \dfrac{C_{r,\nu}}{|\Omega|^2} \Big(\|\partial_z V_+\|_{r,1,\tau}^2 + \|\partial_z V_-\|_{r,1,\tau}^2 \Big) \mathcal K
    	+ \frac{C_{r,\nu,\tau_0}}{|\Omega|}L +  \dfrac{C_{r, \nu}}{|\Omega|} \partial_t N,
    \end{aligned}
\end{equation*}
and thus
\begin{equation}\label{main-estimate-small-2}
     \begin{aligned}
    	&\frac{d}{dt} F  + \frac{\nu}{2} H
    	\leq  C_{r,\nu} L F
    	+ \dfrac{C_{r,\nu}}{|\Omega|^2} \Big(\|\partial_z V_+\|_{r,1,\tau}^2 + \|\partial_z V_-\|_{r,1,\tau}^2 \Big) \mathcal K
    	+ \frac{C_{r,\nu,\tau_0}}{|\Omega|}L +  \dfrac{C_{r, \nu}}{|\Omega|} \partial_t N,
    \end{aligned}
\end{equation}
provided that $|\Omega| > C_{M,r,\nu}$ for some positive constant $C_{M,r,\nu} >0$.
Multiplying \eqref{main-estimate-small-2} with $ e^{- C_{r,\nu}\int_0^t L(s) \,ds} $ leads to
\begin{align*}
	& \dfrac{d}{dt} \bigl( F e^{- C_{r,\nu}\int_0^t L(s) \,ds} \bigr) + \dfrac{\nu}{2}  H e^{- C_{r,\nu}\int_0^t L(s) \,ds}  \leq \dfrac{C_{r,\nu}}{|\Omega|^2} \Big(\|\partial_z V_+\|_{r,1,\tau}^2 + \|\partial_z V_-\|_{r,1,\tau}^2 \Big) \mathcal K e^{- C_{r,\nu}\int_0^t L(s) \,ds}
    	\\
    & \qquad \qquad \qquad + \frac{C_{r,\nu,\tau_0}}{|\Omega|}L e^{- C_{r,\nu}\int_0^t L(s) \,ds} +  \dfrac{C_{r, \nu}}{|\Omega|} \partial_t N e^{- C_{r,\nu}\int_0^t L(s) \,ds}.
\end{align*}
After integrating the above equation in time and recalling that $F(t=0)\leq \frac{M}{|\Omega_0|}$, since $|\Omega|>|\Omega_0|>1$, one obtains
\begin{equation}\label{est:301-1}
\begin{aligned}
	& \bigl( F(t) e^{- C_{r,\nu}\int_0^t L(s) \,ds} \bigr) + \int_0^t  \dfrac{\nu}{2}  H (t') e^{- C_{r,\nu}\int_0^{t'} L(s) \,ds} \,dt' \\
	& \leq \frac{C_M}{|\Omega_0|} + \int_0^t \dfrac{C_{r,\nu}}{|\Omega_0|} \Big(\|\partial_z V_+(t') \|_{r,1,\tau}^2 + \|\partial_z V_-(t') \|_{r,1,\tau}^2 \Big) \mathcal K e^{- C_{r,\nu}\int_0^{t'} L(s) \,ds} \,dt' \\
	& \quad + \int_0^t \frac{C_{r,\nu,\tau_0}}{|\Omega_0|}L(t') e^{- C_{r,\nu}\int_0^{t'} L(s) \,ds} \,dt'  +  \int_0^t  \dfrac{C_{r, \nu}}{|\Omega_0|} \partial_t N(t') e^{- C_{r,\nu}\int_0^{t'} L(s) \,ds} \,dt' \\
	& \leq \dfrac{C_{M,r,\nu,\tau_0}}{|\Omega_0| } + \dfrac{C_{r, \nu}}{|\Omega_0|} \int_0^t  \partial_t N(t') e^{- C_{r,\nu}\int_0^{t'} L(s) \,ds} \,dt'.
\end{aligned}
\end{equation}
According to \eqref{N}, since now $N(0)\neq 0$ due to $\overline{\phi}_0 \neq 0$, the estimate becomes
\begin{equation*}
    \begin{split}
        & \int_0^t  \partial_t N(t') e^{- C_{r,\nu}\int_0^{t'} L(s) \,ds} \,dt' = N(t) e^{- C_{r,\nu}\int_0^{t} (\mathcal K^2 + 1)(s) \,ds}  -N(0)\\
    & \qquad + C_{r,\nu} \int_0^t N(t') L(t') e^{- C_{r,\nu}\int_0^{t'} L(s) \,ds} \,dt'\\
    & \leq C_{M,r,\nu} \Big(F^{\frac{1}{2}}(t)+1 \Big) + C_{r,\nu} \int_0^t  \mathcal K(t') F^{\frac{1}{2}}(t')  \,dt' .
    \end{split}
\end{equation*}
Hence, \eqref{est:301-1} implies that, for $ t \in [0,\mathcal T] $, after applying the young inequality, one has
\begin{equation}\label{est:302-1}
    \begin{aligned}
    & F(t) + \int_0^t H (t')  \,dt' \leq \dfrac{C_{M,r,\nu,\tau_0}}{|\Omega_0|} ,
    \end{aligned}
\end{equation}
where $ \mathcal T \in (0,\infty] $ is given by the  constraint
\begin{equation*}
\tau(s) > 0 \qquad \text{for} \quad s \in [0,\mathcal T].
\end{equation*}
Now we will be able to estimate $ \mathcal T $. To ensure $ \tau > 0 $, from \eqref{constraint-tau-small}, \eqref{growing-small-1}, \eqref{growing-small-2}, and \eqref{est:302-1}, one has
\begin{equation}\label{est:tau-2}
    \begin{aligned}
     \tau(t) &=  - C_r \int_0^t e^{-C_r \int_{t'}^t K^{\frac{1}{2}}(s)\,ds} \bigl(\|V_+\|_{3/2+\delta, 1,0} +\|V_-\|_{3/2+\delta, 1,0} + F^{\frac{1}{2}} + H^{\frac{1}{2}}  \bigr) \,dt'\\
    & \qquad \qquad \qquad+ \tau_0 e^{-C_r \int_0^t K^{\frac{1}{2}}(t')\,dt'}
    \\
    &\geq \tau_0  C'_{M,r,\nu} - C'_{M,r,\nu,\tau_0} \dfrac{1}{|\Omega_0|^{\frac{1}{2}}} (t+1) - C_{r,\nu} \|\widetilde V_0 \|_{\frac{3}{2}+\delta, 0,0}
    \end{aligned}
\end{equation}
for some constant $ \mathcal C'_{M,r,\nu} \in(0,1), C_{r,\nu}, \mathcal C'_{M,r,\nu,
    \tau_0} \in (1,\infty) $. Therefore, for $ t > 0 $ satisfying
\begin{equation}\label{est:305-2}
    t+1 < \dfrac{C'_{M,r,\nu} \tau_0 |\Omega_0|^{\frac{1}{2}}}{2\mathcal C'_{M,r,\nu,
    \tau_0} }
\end{equation}
and $\|\widetilde V_0 \|_{\frac{3}{2}+\delta, 0,0}$ satisfying
\begin{equation*}
    \|\widetilde V_0 \|_{\frac{3}{2}+\delta, 0,0} < \frac{\tau_0  C'_{M,r,\nu}}{2C_{r,\nu}},
\end{equation*}
it follows that $ \tau(t) > 0 $. Consequently, \eqref{est:305-2} implies $\mathcal T = \frac{|\Omega_0|^{\frac{1}{2}}}{C_{\tau_0, M, r, \nu} }$. This completes the proof of Theorem \ref{theorem-small-barotropic}.


\section{Global Existence in $2D$ with $\Omega=0$}\label{section-global}

In this section, we show that the weak solution obtained in section \ref{section-local} exists globally in time in the case of $2D$ and $\Omega=0$, provided that the initial data is small. This result is similar to the one in \cite{PZZ18}, where system \eqref{PE-system} with Dirichlet boundary condition is considered.

To be more precious,
let us consider $\mathcal{V}=(u,v)^\top(x,z,t)$ with \eqref{mean-zero-2}, i.e., the solution to system \eqref{PE-system} independent of the $ y $-variable. It is easy to verify that
\begin{subequations}\label{system-2d-with-rotation}
\begin{align}
    &\overline{u}=0,
    \\
    &\partial_t \overline{v} + \partial_x  \mathfrak P_0 (\widetilde{u} \widetilde{v})=0,
    \\
    &\partial_t \widetilde{u} + \widetilde{u}\partial_x \widetilde{u} - \partial_x \mathfrak P_0 (\widetilde{u}^2) - \Big(\int_0^z \partial_x\widetilde{u}(x,s) ds \Big) \partial_z \widetilde{u} - \Omega \widetilde{v} -\nu \partial_{zz} \widetilde{u} = 0,
    \\
    &\partial_t \widetilde{v} + \widetilde{u}\partial_x \widetilde{v} + \widetilde{u}\partial_x \overline{v} - \partial_x \mathfrak P_0 (\widetilde{u} \widetilde{v})  - \Big(\int_0^z \partial_x\widetilde{u}(x,s) ds \Big) \partial_z \widetilde{v} + \Omega \widetilde{u} -\nu \partial_{zz} \widetilde{v}=0.\label{system-2d-with-rotation-4}
\end{align}
\end{subequations}
 We remind readers that $ \mathfrak P_0 $ is the barotropic projection operator defined in \eqref{P0}.
In addition, let $\Omega = 0$. Then one can observe that $\overline{v}\equiv 0$ and $\widetilde{v}\equiv 0$ are invariant in time, a property that is not true in the case of $\Omega \neq 0$.
Consequently, with $\Omega = 0$ and $\overline{v}_0 = \widetilde{v}_0 = 0$, system \eqref{system-2d-with-rotation} reduces to
\begin{equation}\label{system-2d-without-rotation}
   \partial_t \widetilde{u} + \widetilde{u}\partial_x \widetilde{u} - \partial_x \mathfrak P_0 (\widetilde{u}^2) - \Big(\int_0^z \partial_x\widetilde{u}(x,s) ds \Big) \partial_z \widetilde{u}  -\nu \partial_{zz} \widetilde{u} = 0 \qquad \text{with} \qquad \partial_z \widetilde{u}|_{z=0,1}.
\end{equation}

We have the following theorem concerning the global existence of the weak solutions to  \eqref{system-2d-without-rotation} with $\Omega = 0$:
\begin{theorem}\label{theorem-global}
For $r>2$ and $\tau_0 > 0$, suppose that the initial data
$\widetilde u|_{t=0} = \widetilde{u}_0 \in \mathcal{S}_{r,0,\tau_0} $
with $ \int_0^1 \widetilde{u}_0(x,z) \,dz = 0 $ satisfies the smallness condition
\begin{equation}\label{small-conditn}
   \| \widetilde{u}_0 \|_{r,0,\tau_0} < \frac{\nu \tau_0}{\mathcal C_r} ,
\end{equation}
where $\mathcal C_r > 0$ is a constant as in \eqref{Cr-global}, below. Then the unique weak solution to system \eqref{system-2d-without-rotation} exists globally in time.
\end{theorem}

\begin{proof}[Sketch of proof]
Similarly to \eqref{apri-001}, we have
\begin{equation*}
    \begin{aligned}
        & \frac{1}{2} \frac{d}{dt} \| \widetilde{u} \|_{r,0,\tau}^2 + \nu \|\partial_z \widetilde{u} \|_{r,0,\tau}^2= \dot{\tau} \|A^{r+\frac{1}{2}} e^{\tau A} \widetilde{u}\|^2 - \Big\langle A^r e^{\tau A}\widetilde{u}\partial_x \widetilde{u} , A^r e^{\tau A} \widetilde{u} \Big\rangle
        \\
        & \qquad\qquad - \Big\langle A^r e^{\tau A}\Big(\int_0^z \partial_x\widetilde{u}(x,s) ds \Big) \partial_z \widetilde{u} , A^r e^{\tau A} \widetilde{u} \Big\rangle\\
        & \qquad \leq \Big(\dot{\tau} + C_r(\| \widetilde{u} \|_{r,0,\tau} + \|\partial_z \widetilde{u} \|_{r,0,\tau}) \Big) \|A^{r+\frac{1}{2}} e^{\tau A} \widetilde{u}\|^2,
    \end{aligned}
\end{equation*}
thanks to Lemma \ref{lemma-type1} and Lemma \ref{lemma-type2}.

It is easy to see that $\int_0^1 \widetilde{u}(x,z) dz = 0$. One can apply the Poincar\'e inequality to get
$
    \| \widetilde{u} \|_{r,0,\tau} \leq  \|\partial_z \widetilde{u} \|_{r,0,\tau},
$
and consequently,
\begin{equation*}
   \frac{1}{2} \frac{d}{dt} \| \widetilde{u} \|_{r,0,\tau}^2 + \frac{\nu}{2} \|\partial_z \widetilde{u} \|_{r,0,\tau}^2 \leq \Big(\dot{\tau} + C_r \|\partial_z \widetilde{u} \|_{r,0,\tau} \Big) \|A^{r+\frac{1}{2}} e^{\tau A} \widetilde{u}\|^2 - \frac{\nu}{2}\| \widetilde{u} \|_{r,0,\tau}^2.
\end{equation*}
Assuming that
\begin{equation}\label{global-tau}
    \dot{\tau} + C_r \|\partial_z \widetilde{u} \|_{r,0,\tau} = 0,
\end{equation}
one has
\begin{equation*}
    \frac{d}{dt} \| \widetilde{u} \|_{r,0,\tau}^2 + \nu \|\partial_z \widetilde{u} \|_{r,0,\tau}^2 \leq -\nu \| \widetilde{u} \|_{r,0,\tau}^2.
\end{equation*}
After applying the Gr\"onwall inequality, one obtains
\begin{equation*}
    \| \widetilde{u}(t) \|_{r,0,\tau(t)}^2 e^{\nu t} + \nu \int_0^t \|\partial_z \widetilde{u}(s) \|_{r,0,\tau(s)}^2 e^{\nu s} ds \leq \| \widetilde{u}_0 \|_{r,0,\tau_0}^2 .
\end{equation*}
Therefore, integrating \eqref{global-tau} from $0$ to $t\in (0,\infty)$ and applying the H\"older inequality in the resultant lead to
\begin{equation}\label{Cr-global}
\begin{split}
    \tau(t) &= \tau_0 - C_r \int_0^t  \|\partial_z \widetilde{u}(s) \|_{r,0,\tau(s)} ds
    \\
    &\geq \tau_0 -C_r \Big( \int_0^t  \|\partial_z \widetilde{u}(s) \|_{r,0,\tau(s)}^2 e^{\nu s} ds\Big)^{\frac{1}{2}} \Big( \int_0^t   e^{-\nu s} ds\Big)^{\frac{1}{2}}
    \\
    &\geq \tau_0 - \frac{\mathcal C_r}{\nu} \| \widetilde{u}_0 \|_{r,0,\tau_0},
\end{split}
\end{equation}
for some positive constant $ \mathcal C_r \in (0,\infty) $.

In summary, for the initial data satisfying \eqref{small-conditn},
we have that $\tau(t) > 0$ for all $t>0$, and thus the solution exists for all time.
\end{proof}

\appendix

\section{Estimates of nonlinear terms}
In this appendix, we list the estimates of nonlinear terms in the analytic-Sobolev spaces $\mathcal{S}_{r,s,\tau}$. Lemma \ref{lemma-type1}--\ref{lemma-type2} will be used to prove the local well-posedness, and similar estimates can be found in \cite{GILT20,HL22}.
\begin{lemma}\label{lemma-type1}
For $f, g, h\in \mathcal{S}_{r+\frac{1}{2},s,\tau}$, where $r>1$, $s\geq 0$, and $\tau \geq 0$, one has
\begin{equation}\label{lemma-type1-inequality}
    \begin{split}
        &\Big|\Big\langle A^r e^{\tau A} (f\cdot \nabla  g), A^r e^{\tau A} h  \Big\rangle\Big| \\
        \leq&  \int_0^1 C_r\Big[ (\|A^r e^{\tau A} f(z)\|_{L^2(\mathbb{T}^2)} + |\hat{f}_0(z)| ) \|A^{r+\frac{1}{2}} e^{\tau A} g(z)\|_{L^2(\mathbb{T}^2)} \|A^{r+\frac{1}{2}} e^{\tau A} h(z)\|_{L^2(\mathbb{T}^2)} \\
       &+ \|A^{r+\frac{1}{2}} e^{\tau A} f(z)\|_{L^2(\mathbb{T}^2)}  \|A^{r+\frac{1}{2}} e^{\tau A} g(z)\|_{L^2(\mathbb{T}^2)} \|A^{r} e^{\tau A} h(z)\|_{L^2(\mathbb{T}^2)}\Big] dz.
    \end{split}
\end{equation}
\end{lemma}
\begin{proof}
    First, notice that $\Big|\Big\langle A^r e^{\tau A} (f\cdot \nabla  g), A^r e^{\tau A} h \Big\rangle\Big| = \Big|\Big\langle f\cdot \nabla  g, A^r e^{\tau A} H  \Big\rangle\Big| $, where $H =  A^r e^{\tau A} h$.
    Using the Fourier representation, we have,
    \begin{subequations}\label{id:frr-rpn}
    \begin{align}
     f(\boldsymbol{x},z) & = \sum\limits_{\boldsymbol{j}\in  2\pi \mathbb{Z}^2} \hat{f}_{\boldsymbol{j}}(z) e^{ i\boldsymbol{j}\cdot \boldsymbol{x}}, \\
    g(\boldsymbol{x},z) & = \sum\limits_{\boldsymbol{k}\in 2\pi \mathbb{Z}^2} \hat{g}_{\boldsymbol{k}}(z) e^{ i\boldsymbol{k}\cdot \boldsymbol{x}}, \\
    h(\boldsymbol{x},z) & = \sum\limits_{\boldsymbol{l}\in 2\pi \mathbb{Z}^2} \hat{h}_{\boldsymbol{l}}(z) e^{ i\boldsymbol{l}\cdot \boldsymbol{x}}, \qquad \text{and by definition},  \\
    A^r e^{\tau A} H(\boldsymbol{x},z) & = \sum\limits_{\boldsymbol{l}\in 2 \pi \mathbb{Z}^2} |\boldsymbol{l}|^r e^{\tau |\boldsymbol{l}|}\hat{H}_{\boldsymbol{l}}(z) e^{i\boldsymbol{l}\cdot \boldsymbol{x}}, \qquad \text{with} \quad \hat H_{\boldsymbol{l}}(z) = | \boldsymbol{l} |^r e^{\tau |\boldsymbol{l}|} \hat h_{\boldsymbol{l}}(z).
    \end{align}
    \end{subequations}
    Therefore,
    \begin{eqnarray*}
    \Big|\Big\langle f\cdot \nabla  g, A^r e^{\tau A} H  \Big\rangle\Big| \leq  \int_0^1 \sum\limits_{\boldsymbol{j}+\boldsymbol{k}+\boldsymbol{l}=0} |\hat{f}_{\boldsymbol{j}}(z)||\boldsymbol{k}||\hat{g}_{\boldsymbol{k}}(z)||\boldsymbol{l}|^r e^{\tau |\boldsymbol{l}|} |\hat{H}_{\boldsymbol{l}}(z)| dz.
    \end{eqnarray*}

Since $|\boldsymbol{l}| = |\boldsymbol{j}+\boldsymbol{k}| \leq |\boldsymbol{j}|+|\boldsymbol{k}|$, we have the following inequalities:
\begin{equation*}
    |\boldsymbol{l}|^r \leq (|\boldsymbol{j}|+|\boldsymbol{k}|)^r \leq C_r(|\boldsymbol{j}|^r + |\boldsymbol{k}|^r), \;\;\;  e^{\tau |\boldsymbol{l}|} \leq e^{\tau |\boldsymbol{j}|} e^{\tau |\boldsymbol{k}|}.
\end{equation*}
Applying these inequalities, we have
\begin{eqnarray*}
\Big|\Big\langle f\cdot \nabla  g, A^r e^{\tau A} H  \Big\rangle\Big| \leq  \int_0^1 \sum\limits_{\boldsymbol{j}+\boldsymbol{k}+\boldsymbol{l}=0} C_r|\hat{f}_{\boldsymbol{j}}(z)||\boldsymbol{k}||\hat{g}_{\boldsymbol{k}}(z)|(|\boldsymbol{j}|^r+|\boldsymbol{k}|^r)e^{\tau |\boldsymbol{j}|}e^{\tau |\boldsymbol{k}|}|\boldsymbol{l}|^r e^{\tau |\boldsymbol{l}|}|\hat{h}_{\boldsymbol{l}}(z)| dz.
\end{eqnarray*}
Since $|\boldsymbol{k}|, |\boldsymbol{j}|, |\boldsymbol{l}| \geq 0 $, we have $|\boldsymbol{k}|^{\frac{1}{2}} \leq (|\boldsymbol{j}|+|\boldsymbol{l}|)^{\frac{1}{2}} \leq |\boldsymbol{j}|^{\frac{1}{2}} + |\boldsymbol{l}|^{\frac{1}{2}}$, therefore,
\begin{equation*}
\begin{aligned}
& 
\Big|\Big\langle f\cdot \nabla  g, A^r e^{\tau A} H  \Big\rangle\Big| \\
\leq & \int_0^1 \sum\limits_{\boldsymbol{j}+\boldsymbol{k}+\boldsymbol{l}=0} C_r|\hat{f}_{\boldsymbol{j}}(z)||\boldsymbol{k}|^{\frac{1}{2}}(|\boldsymbol{j}|^{\frac{1}{2}} + |\boldsymbol{l}|^{\frac{1}{2}})|\hat{g}_{\boldsymbol{k}}(z)|(|\boldsymbol{j}|^r+|\boldsymbol{k}|^r)e^{\tau |\boldsymbol{j}|}e^{\tau |\boldsymbol{k}|}|\boldsymbol{l}|^r e^{\tau |\boldsymbol{l}|}|\hat{h}_{\boldsymbol{l}}(z)| dz \\
\leq & \int_0^1 \sum\limits_{\boldsymbol{j}+\boldsymbol{k}+\boldsymbol{l}=0} C_r \Big(|\boldsymbol{k}|^{\frac{1}{2}}|\boldsymbol{j}|^{r+\frac{1}{2}} |\boldsymbol{l}|^{r} + |\boldsymbol{k}|^{r+\frac{1}{2}}|\boldsymbol{j}|^{\frac{1}{2}} |\boldsymbol{l}|^{r} + |\boldsymbol{k}|^{\frac{1}{2}}|\boldsymbol{j}|^{r} |\boldsymbol{l}|^{r+\frac{1}{2}} + |\boldsymbol{k}|^{r+\frac{1}{2}}|\boldsymbol{l}|^{r+\frac{1}{2}} \Big) \\
& \qquad 
\times e^{\tau |\boldsymbol{j}|}e^{\tau |\boldsymbol{k}|} e^{\tau |\boldsymbol{l}|}  |\hat{f}_{\boldsymbol{j}}(z)| |\hat{g}_{\boldsymbol{k}}(z)| |\hat{h}_{\boldsymbol{l}}(z)| dz =: \int_0^1 (A_1 + A_2 + A_3 + A_4)(z) dz.
\end{aligned}
\end{equation*}

Thanks to Cauchy–Schwarz inequality, since $r>1$, we have
\begin{equation*}
    \begin{aligned}
A_1 & = \sum\limits_{\boldsymbol{j}+\boldsymbol{k}+\boldsymbol{l}=0} C_r |\boldsymbol{k}|^{\frac{1}{2}}|\boldsymbol{j}|^{r+\frac{1}{2}} |\boldsymbol{l}|^{r} e^{\tau |\boldsymbol{j}|}e^{\tau |\boldsymbol{k}|} e^{\tau |\boldsymbol{l}|}  |\hat{f}_{\boldsymbol{j}}(z)| |\hat{g}_{\boldsymbol{k}}(z)| |\hat{h}_{\boldsymbol{l}}(z)| \\
&
= C_r \sum\limits_{\substack{\boldsymbol{k}\in 2\pi \mathbb{Z}^2 \\ \boldsymbol{k}\neq 0} } \biggl\lbrack |\boldsymbol{k}|^{\frac{1}{2}} |\hat{g}_{\boldsymbol{k}}(z)| e^{\tau |\boldsymbol{k}|} \sum\limits_{\substack{\boldsymbol{j}\in 2\pi \mathbb{Z}^2 \\ \boldsymbol{j}\neq 0, -\boldsymbol{k}} } |\boldsymbol{j}|^{r+\frac{1}{2}} e^{\tau |\boldsymbol{j}|}|\hat{f}_{\boldsymbol{j}}(z)|  |\boldsymbol{j}+\boldsymbol{k}|^{r}e^{\tau |\boldsymbol{j}+\boldsymbol{k}|}|\hat{h}_{-\boldsymbol{j}-\boldsymbol{k}}(z)| \biggr\rbrack \\
&
\leq C_r \Big( \sum\limits_{\substack{\boldsymbol{k}\in 2\pi \mathbb{Z}^2 \\ \boldsymbol{k}\neq 0} } |\boldsymbol{k}|^{-2r}\Big)^{\frac{1}{2}} \Big( \sum\limits_{\substack{\boldsymbol{k}\in 2\pi \mathbb{Z}^2 \\ \boldsymbol{k}\neq 0} } |\boldsymbol{k}|^{2r+1} e^{2\tau |\boldsymbol{k}|} |\hat{g}_{\boldsymbol{k}}(z)|^2\Big)^{\frac{1}{2}}   \\
& \qquad
\times \sup\limits_{\boldsymbol{k}\in 2\pi \mathbb{Z}^2}\biggl\lbrack \Big( \sum\limits_{\substack{\boldsymbol{j}\in 2\pi \mathbb{Z}^2 \\ \boldsymbol{j}\neq 0, -\boldsymbol{k}} } |\boldsymbol{j}|^{2r+1}e^{2\tau |\boldsymbol{j}|} |\hat{f}_{\boldsymbol{j}}(z)|^2\Big)^{\frac{1}{2}} \Big( \sum\limits_{\substack{\boldsymbol{j}\in 2\pi \mathbb{Z}^2 \\ \boldsymbol{j}\neq 0, -\boldsymbol{k}} } |\boldsymbol{j}+\boldsymbol{k}|^{2r}e^{2\tau |\boldsymbol{j}+\boldsymbol{k}|} |\hat{h}_{-\boldsymbol{j}-\boldsymbol{k}}(z)|^2\Big)^{\frac{1}{2}} \biggr\rbrack \\
& \leq C_r \|A^{r+\frac{1}{2}} e^{\tau A} f(z)\|_{L^2(\mathbb{T}^2)} \|A^{r+\frac{1}{2}} e^{\tau A} g(z)\|_{L^2(\mathbb{T}^2)} \|A^{r} e^{\tau A} h(z)\|_{L^2(\mathbb{T}^2)}.
\end{aligned}
\end{equation*}
Similarly, we have
\begin{equation*}
\begin{aligned}
A_2 & = \sum\limits_{\boldsymbol{j}+\boldsymbol{k}+\boldsymbol{l}=0} C_r |\boldsymbol{k}|^{r+\frac{1}{2}}|\boldsymbol{j}|^{\frac{1}{2}} |\boldsymbol{l}|^{r} e^{\tau |\boldsymbol{j}|}e^{\tau |\boldsymbol{k}|} e^{\tau |\boldsymbol{l}|}  |\hat{f}_{\boldsymbol{j}}(z)| |\hat{g}_{\boldsymbol{k}}(z)| |\hat{h}_{\boldsymbol{l}}(z)|  \\
& \leq C_r \|A^{r+\frac{1}{2}} e^{\tau A} f(z)\|_{L^2(\mathbb{T}^2)} \|A^{r+\frac{1}{2}} e^{\tau A} g(z)\|_{L^2(\mathbb{T}^2)} \|A^{r} e^{\tau A} h(z)\|_{L^2(\mathbb{T}^2)},
\end{aligned}
\end{equation*}
and
\begin{equation*}
\begin{aligned}
A_3 & = \sum\limits_{\boldsymbol{j}+\boldsymbol{k}+\boldsymbol{l}=0} C_r |\boldsymbol{k}|^{\frac{1}{2}}|\boldsymbol{j}|^{r} |\boldsymbol{l}|^{r+\frac{1}{2}} e^{\tau |\boldsymbol{j}|}e^{\tau |\boldsymbol{k}|} e^{\tau |\boldsymbol{l}|}  |\hat{f}_{\boldsymbol{j}}(z)| |\hat{g}_{\boldsymbol{k}}(z)| |\hat{h}_{\boldsymbol{l}}(z)|  \\
& \leq C_r \|A^{r} e^{\tau A} f(z)\|_{L^2(\mathbb{T}^2)} \|A^{r+\frac{1}{2}} e^{\tau A} g(z)\|_{L^2(\mathbb{T}^2)} \|A^{r+\frac{1}{2}} e^{\tau A} h(z)\|_{L^2(\mathbb{T}^2)}.
\end{aligned}
\end{equation*}

For $A_4$, thanks to Cauchy–Schwarz inequality, since $r>1$, we have
\begin{equation*}
\begin{aligned}
A_4 & = \sum\limits_{\boldsymbol{j}+\boldsymbol{k}+\boldsymbol{l}=0} C_r |\boldsymbol{k}|^{r+\frac{1}{2}}|\boldsymbol{l}|^{r+\frac{1}{2}} e^{\tau |\boldsymbol{j}|}e^{\tau |\boldsymbol{k}|} e^{\tau |\boldsymbol{l}|}  |\hat{f}_{\boldsymbol{j}}(z)| |\hat{g}_{\boldsymbol{k}}(z)| |\hat{h}_{\boldsymbol{l}}(z)| \\
&
= C_r \sum\limits_{\boldsymbol{j}\in 2\pi \mathbb{Z}^2 } \biggl\lbrack e^{\tau |\boldsymbol{j}|}|\hat{f}_{\boldsymbol{j}}(z)|  \sum\limits_{\substack{\boldsymbol{k}\in  2\pi \mathbb{Z}^2 \\ \boldsymbol{k}\neq 0, -\boldsymbol{j}} } |\boldsymbol{k}|^{r+\frac{1}{2}} |\hat{g}_{\boldsymbol{k}}(z)| e^{\tau |\boldsymbol{k}|}  |\boldsymbol{j}+\boldsymbol{k}|^{r+\frac{1}{2}}e^{\tau |\boldsymbol{j}+\boldsymbol{k}|}|\hat{h}_{-\boldsymbol{j}-\boldsymbol{k}}| \biggr\rbrack \\
&
\leq C_r \Big\{ |\hat{f}_0(z)| + \Big( \sum\limits_{\substack{\boldsymbol{j}\in 2\pi \mathbb{Z}^2 \\ \boldsymbol{j}\neq 0} } |\boldsymbol{j}|^{-2r}\Big)^{\frac{1}{2}} \Big( \sum\limits_{\substack{\boldsymbol{j}\in 2\pi \mathbb{Z}^2 \\ \boldsymbol{j}\neq 0} } |\boldsymbol{j}|^{2r} e^{2\tau |\boldsymbol{j}|} |\hat{f}_{\boldsymbol{j}}(z)|^2\Big)^{\frac{1}{2}} \Big\}   \\
& \qquad
\times \sup\limits_{\boldsymbol{j}\in 2\pi \mathbb{Z}^2} \biggl\lbrack \Big( \sum\limits_{\substack{\boldsymbol{k}\in 2\pi \mathbb{Z}^2 \\ \boldsymbol{k}\neq 0, -\boldsymbol{j}} } |\boldsymbol{k}|^{2r+1}e^{2\tau |\boldsymbol{k}|} |\hat{g}_{\boldsymbol{k}}(z)|^2\Big)^{\frac{1}{2}} \Big( \sum\limits_{\substack{\boldsymbol{k}\in 2\pi \mathbb{Z}^2 \\ \boldsymbol{k}\neq 0, -\boldsymbol{j}} } |\boldsymbol{j}+\boldsymbol{k}|^{2r+1}e^{2\tau |\boldsymbol{j}+\boldsymbol{k}|} |\hat{h}_{-\boldsymbol{j}-\boldsymbol{k}}|^2\Big)^{\frac{1}{2}} \biggr\rbrack \\
& \leq C_r (\|A^{r} e^{\tau A} f(z)\|_{L^2(\mathbb{T}^2)} + |\hat{f}_0(z)|) \|A^{r+\frac{1}{2}} e^{\tau A} g(z)\|_{L^2(\mathbb{T}^2)} \|A^{r+\frac{1}{2}} e^{\tau A} h(z)\|_{L^2(\mathbb{T}^2)}.
\end{aligned}
\end{equation*}

Combining the estimates for $A_1$ to $A_4$, we achieve the desired inequality.
\end{proof}

\begin{lemma}\label{lemma-type2}
For $f,  h\in \mathcal{S}_{r+\frac{1}{2},s,\tau}$ and $g, \partial_z g \in \mathcal{S}_{r,s,\tau}$, where $r > \frac{3}{2} $, $s\geq 0$, and $\tau\geq 0$, one has
\begin{equation*}
\begin{aligned}
    & \Big|\Big\langle A^r e^{\tau A} \big((\int_0^z \nabla \cdot f(\boldsymbol{x},s)ds)\partial_z g\big), A^r e^{\tau A} h  \Big\rangle\Big|\\
    \leq & C_r\|A^{r+\frac{1}{2}} e^{\tau A} f\|  \| \partial_z  g\|_{r,0,\tau} \|A^{r+\frac{1}{2}} e^{\tau A} h\|.
\end{aligned}
\end{equation*}
\end{lemma}
\begin{proof}
First, $\Big|\Big\langle A^r e^{\tau A} \big((\int_0^z \nabla \cdot f(\boldsymbol{x},s)ds)\partial_z g\big), A^r e^{\tau A} h  \Big\rangle\Big| = \Big|\Big\langle (\int_0^z \nabla \cdot f(\boldsymbol{x},s)ds)\partial_z g, A^r e^{\tau A} H  \Big\rangle\Big| $.
Owing to the Fourier representation in \eqref{id:frr-rpn} ,
we have
\begin{equation*}
\begin{aligned}
& \Big|\Big\langle (\int_0^z \nabla \cdot f(\boldsymbol{x},s)ds)\partial_z g, A^r e^{\tau A} H  \Big\rangle\Big| = \Big|\Big\langle \int_0^z \sum\limits_{\boldsymbol{j}\in 2\pi \mathbb{Z}^2 } \boldsymbol{j} \cdot \hat{f}_{\boldsymbol{j}}(s) e^{i\boldsymbol{j} \cdot \boldsymbol{x}} ds)\partial_z g, A^r e^{\tau A} H  \Big\rangle\Big| \\
\leq & \int_0^1 \sum\limits_{\boldsymbol{j}+\boldsymbol{k}+\boldsymbol{l}=0} C_r|\boldsymbol{j}|\Big(\int_0^z|\hat{f}_{\boldsymbol{j}}(s)|ds\Big)|\partial_z \hat{g}_{\boldsymbol{k}}(z)|(|\boldsymbol{j}|^r+|\boldsymbol{k}|^r)e^{\tau |\boldsymbol{j}|}e^{\tau |\boldsymbol{k}|}|\boldsymbol{l}|^r e^{\tau |\boldsymbol{l}|}|\hat{h}_{\boldsymbol{l}}(z)| dz \\
\leq & \int_0^1 \sum\limits_{\boldsymbol{j}+\boldsymbol{k}+\boldsymbol{l}=0} C_r \Big(|\boldsymbol{k}|^{\frac{1}{2}}|\boldsymbol{j}|^{r+\frac{1}{2}} |\boldsymbol{l}|^{r} + |\boldsymbol{j}|^{r+\frac{1}{2}} |\boldsymbol{l}|^{r+\frac{1}{2}} + |\boldsymbol{j}||\boldsymbol{k}|^{r} |\boldsymbol{l}|^{r}  \Big) \\
& \qquad \qquad
\times e^{\tau |\boldsymbol{j}|}e^{\tau |\boldsymbol{k}|} e^{\tau |\boldsymbol{l}|}  \Big(\int_0^z|\hat{f}_{\boldsymbol{j}}(s)|ds\Big)|\partial_z \hat{g}_{\boldsymbol{k}}(z)| |\hat{h}_{\boldsymbol{l}}(z)| dz
=: B_1 + B_2 + B_3 .
\end{aligned}
\end{equation*}
where we have substituted the following inequalities: for $  \boldsymbol{j}+\boldsymbol{k}+ \boldsymbol{l}=0$,
\begin{equation*}
    |\boldsymbol{j}|^{\frac{1}{2}} \leq (|\boldsymbol{k}|^{\frac{1}{2}} + |\boldsymbol{l}|^{\frac{1}{2}}),  \;\;\;
    |\boldsymbol{l}|^r  \leq C_r (|\boldsymbol{j}|^r+|\boldsymbol{k}|^r).
\end{equation*}

Thanks to the Cauchy–Schwarz inequality, since $r > \frac{3}{2} $, we have
\begin{equation*}
\begin{aligned}
B_1 & = \int_0^1 \sum_{\boldsymbol{j}+\boldsymbol{k}+\boldsymbol{l}=0}  C_r  |\boldsymbol{k}|^{\frac{1}{2}}|\boldsymbol{j}|^{r+\frac{1}{2}} |\boldsymbol{l}|^{r} e^{\tau |\boldsymbol{j}|}e^{\tau |\boldsymbol{k}|} e^{\tau |\boldsymbol{l}|}  \Big(\int_0^z|\hat{f}_{\boldsymbol{j}}(s)|ds\Big)|\partial_z \hat{g}_{\boldsymbol{k}}(z)| |\hat{h}_{\boldsymbol{l}}(z)| dz  \\
&
= C_r \int_0^1  \sum_{\substack{\boldsymbol{k}\in 2\pi \mathbb{Z}^2 \\ \boldsymbol{k} \neq 0}} \biggl\lbrack |\boldsymbol{k}|^{\frac{1}{2}} |\partial_z \hat{g}_{\boldsymbol{k}}(z)| e^{\tau |\boldsymbol{k}|}  \sum_{\boldsymbol{j}\in 2\pi \mathbb{Z}^2 }   |\boldsymbol{j}|^{r+\frac{1}{2}} e^{\tau |\boldsymbol{j}|}\Big(\int_0^z|\hat{f}_{\boldsymbol{j}}(s)|ds\Big)  |\boldsymbol{j}+\boldsymbol{k}|^{r}e^{\tau |\boldsymbol{j}+\boldsymbol{k}|}|\hat{h}_{-\boldsymbol{j}-\boldsymbol{k}}(z)|  \biggr\rbrack dz \\
&
\leq  C_r \int_0^1  \Big(\sum_{\boldsymbol{k}\neq 0} |\boldsymbol{k}|^{1-2r}\Big)^{\frac{1}{2}}  \Big(\sum_{\boldsymbol{k}\neq 0} |\boldsymbol{k}|^{2r} |\partial_z \hat{g}_{\boldsymbol{k}}(z)|^2 e^{2\tau |\boldsymbol{k}|}\Big)^{\frac{1}{2}} \sup_{\boldsymbol{k}\neq 0}\biggl\lbrack \Big( \sum_{\boldsymbol{j}\in 2\pi \mathbb{Z}^2}  |\boldsymbol{j}|^{2r+1} e^{2\tau |\boldsymbol{j}|} \|\hat{f}_{\boldsymbol{j}}\|_{L^2_z}^2\Big)^{\frac{1}{2}}   \\
& \qquad\qquad
\times \Big( \sum_{\boldsymbol{j}\in 2\pi \mathbb{Z}^2}   |\boldsymbol{j}+\boldsymbol{k}|^{2r}e^{2\tau |\boldsymbol{j}+\boldsymbol{k}|}|\hat{h}_{-\boldsymbol{j}-\boldsymbol{k}}(z)|^2\Big)^{\frac{1}{2}} \biggr\rbrack dz  \\
&  \leq C_r \|A^{r+\frac{1}{2}} e^{\tau A} f\| \int_0^1 \|A^{r} e^{\tau A} \partial_z g(z)\|_{L^2(\mathbb{T}^2)} \|A^{r} e^{\tau A} h(z)\|_{L^2(\mathbb{T}^2)} dz  \\
&
\leq C_r \|A^{r+\frac{1}{2}} e^{\tau A} f\|  \|A^{r} e^{\tau A} \partial_z  g\| \|A^{r} e^{\tau A} h\|,
\end{aligned}
\end{equation*}

For $B_2$, we have
\begin{equation*}
\begin{aligned}
B_2 & = \int_0^1 \sum\limits_{\boldsymbol{j}+\boldsymbol{k}+\boldsymbol{l}=0}  C_r  |\boldsymbol{j}|^{r+\frac{1}{2}} |\boldsymbol{l}|^{r+\frac{1}{2}} e^{\tau |\boldsymbol{j}|}e^{\tau |\boldsymbol{k}|} e^{\tau |\boldsymbol{l}|}  \Big(\int_0^z|\hat{f}_{\boldsymbol{j}}(s)|ds\Big)|\partial_z \hat{g}_{\boldsymbol{k}}(z)| |\hat{h}_{\boldsymbol{l}}(z)| dz  \\
& = C_r \int_0^1 \sum\limits_{\boldsymbol{k}\in 2\pi \mathbb{Z}^2} \biggl\lbrack |\partial_z \hat{g}_{\boldsymbol{k}}(z)| e^{\tau |\boldsymbol{k}|} \sum\limits_{\boldsymbol{j}\in 2\pi \mathbb{Z}^2 }   |\boldsymbol{j}|^{r+\frac{1}{2}} e^{\tau |\boldsymbol{j}|}\Big(\int_0^z|\hat{f}_{\boldsymbol{j}}(s)|ds\Big)  |\boldsymbol{j}+\boldsymbol{k}|^{r+\frac{1}{2}}e^{\tau |\boldsymbol{j}+\boldsymbol{k}|}|\hat{h}_{-\boldsymbol{j}-\boldsymbol{k}}(z)| \biggr\rbrack dz \\
&
\leq \int_0^1 C_r \Big\{|\partial_z \hat{g}_0(z)|+\Big(\sum\limits_{\boldsymbol{k}\neq 0} |\boldsymbol{k}|^{-2r}\Big)^{\frac{1}{2}}  \Big(\sum\limits_{\boldsymbol{k}\neq 0} |\boldsymbol{k}|^{2r} |\partial_z \hat{g}_{\boldsymbol{k}}(z)|^2 e^{2\tau |\boldsymbol{k}|}\Big)^{\frac{1}{2}} \Big\}\\
& \qquad\qquad
\times \sup\limits_{\boldsymbol{k}\in 2\pi \mathbb{Z}^2}\biggl\lbrack  \Big( \sum\limits_{\boldsymbol{j}\in 2\pi \mathbb{Z}^2}  |\boldsymbol{j}|^{2r+1} e^{2\tau |\boldsymbol{j}|} \|\hat{f}_{\boldsymbol{j}}\|_{L^2_z}^2\Big)^{\frac{1}{2}}
 \Big( \sum\limits_{\boldsymbol{j}\in 2\pi \mathbb{Z}^2}   |\boldsymbol{j}+\boldsymbol{k}|^{2r+1}e^{2\tau |\boldsymbol{j}+\boldsymbol{k}|}|\hat{h}_{-\boldsymbol{j}-\boldsymbol{k}}(z)|^2\Big)^{\frac{1}{2}} \biggr\rbrack dz  \\
&
\leq C_r \|A^{r+\frac{1}{2}} e^{\tau A} f\| \int_0^1 \Big(|\partial_z \hat{g}_0(z)| + \|A^{r} e^{\tau A} \partial_z g(z)\|_{L^2(\mathbb{T}^2)}\Big) \|A^{r+\frac{1}{2}} e^{\tau A} h(z)\|_{L^2(\mathbb{T}^2)} dz  \\
&
\leq C_r \|A^{r+\frac{1}{2}} e^{\tau A} f\|  \| \partial_z  g\|_{r,0,\tau} \|A^{r+\frac{1}{2}} e^{\tau A} h\|.
\end{aligned}
\end{equation*}

The estimate of $B_3$ is similar to that of $B_1$, and one can obtain that
\begin{equation*}
    B_3 \leq C_r \|A^{r+\frac{1}{2}} e^{\tau A} f\|  \|A^{r} e^{\tau A} \partial_z  g\| \|A^{r} e^{\tau A} h\|.
\end{equation*}

Combine the estimates of $B_1$, $B_2$, and $B_3$, we obtain the desired result.
\end{proof}

\begin{lemma}\label{lemma-type3}
For $f, g, h\in \mathcal{S}_{r+\frac{1}{2},s,\tau}$, where $r>1$, $s\geq 0$, and $\tau \geq 0$, one has
\begin{equation*}
    \begin{split}
        &\Big|\Big\langle A^r e^{\tau A} \big((\nabla\cdot f)g\big), A^r e^{\tau A} h  \Big\rangle\Big|
        \\
        \leq&  \int_0^1 C_r\Big[ (\|A^r e^{\tau A} g(z)\|_{L^2(\mathbb{T}^2)} + |\hat{g}_0(z)| ) \|A^{r+\frac{1}{2}} e^{\tau A} f(z)\|_{L^2(\mathbb{T}^2)} \|A^{r+\frac{1}{2}} e^{\tau A} h(z)\|_{L^2(\mathbb{T}^2)} \\
       &+ \|A^{r+\frac{1}{2}} e^{\tau A} f(z)\|_{L^2(\mathbb{T}^2)}  \|A^{r+\frac{1}{2}} e^{\tau A} g(z)\|_{L^2(\mathbb{T}^2)} \|A^{r} e^{\tau A} h(z)\|_{L^2(\mathbb{T}^2)}\Big] dz.
    \end{split}
\end{equation*}
\end{lemma}
The proof of Lemma \ref{lemma-type3} is almost the same as Lemma \ref{lemma-type1}, so we omit it.

We will show lemmas which are essential in the study of effect of rotation. Lemma \ref{lemma-difference-type1} to Lemma \ref{lemma-difference-type4} are concerning the commutator estimates.
\begin{lemma} \label{lemma-difference-type1}
For $f, g, h\in \mathcal{S}_{r+\frac{1}{2},s,\tau}$, where $r > 2 $, $s\geq 0$, and $\tau\geq 0$, one has
\begin{eqnarray*}
&&\hskip-.8in
\Big|\Big\langle A^r e^{\tau A} (f\cdot \nabla  g), A^r e^{\tau A} h  \Big\rangle - \Big\langle  f\cdot \nabla  A^r e^{\tau A} g, A^r e^{\tau A} h  \Big\rangle\Big| \nonumber\\
&&\hskip-.9in
\leq C_r \int_0^1 \|A^r f(z)\|_{L^2(\mathbb{T}^2)} \|A^r g(z)\|_{L^2(\mathbb{T}^2)} \|A^r h(z)\|_{L^2(\mathbb{T}^2)} dz \nonumber\\
&&\hskip-.8in
+ C_r \tau \int_0^1 \|A^{r+\frac{1}{2}} e^{\tau A} f(z)\|_{L^2(\mathbb{T}^2)} \|A^{r+\frac{1}{2}} e^{\tau A} g(z)\|_{L^2(\mathbb{T}^2)} \|A^{r+\frac{1}{2}} e^{\tau A} h(z)\|_{L^2(\mathbb{T}^2)} dz.
\end{eqnarray*}
\end{lemma}
Next, we have
\begin{lemma} \label{lemma-difference-type2}
For $f, g, h\in \mathcal{S}_{r+\frac{1}{2},s,\tau}$, where $r > 2 $, $s\geq 0$, and $\tau\geq 0$, one has
\begin{eqnarray*}
&&\hskip-.8in
\Big|\Big\langle A^r e^{\tau A} \big( (\nabla\cdot f)g\big), A^r e^{\tau A} h  \Big\rangle - \Big\langle  (\nabla\cdot A^r e^{\tau A} f)g, A^r e^{\tau A} h  \Big\rangle\Big| \nonumber \\
&&\hskip-.9in
\leq C_r \int_0^1 \|A^r f(z)\|_{L^2(\mathbb{T}^2)} \|A^r g(z)\|_{L^2(\mathbb{T}^2)} \|A^r h(z)\|_{L^2(\mathbb{T}^2)} dz \nonumber\\
&&\hskip-.8in
+ C_r \tau \int_0^1 \|A^{r+\frac{1}{2}} e^{\tau A} f(z)\|_{L^2(\mathbb{T}^2)} \|A^{r+\frac{1}{2}} e^{\tau A} g(z)\|_{L^2(\mathbb{T}^2)} \|A^{r+\frac{1}{2}} e^{\tau A} h(z)\|_{L^2(\mathbb{T}^2)} dz.
\end{eqnarray*}

\end{lemma}
We start with the proof of Theorem \ref{lemma-difference-type1}. The proof of Theorem \ref{lemma-difference-type2} will be similarly.

\begin{proof}[Proof of Lemma \ref{lemma-difference-type1}]
First, notice that $\Big|\Big\langle A^r e^{\tau A} (f\cdot \nabla  g), A^r e^{\tau A} h \Big\rangle\Big| = \Big|\Big\langle f\cdot \nabla  g, A^r e^{\tau A} H  \Big\rangle\Big| $, where $H =  A^r e^{\tau A} h$.
    We use Fourier representation of $f, g$ and $H$, in which we can write
    \begin{eqnarray*}
    &&\hskip-.8in
     f(\boldsymbol{x},z) = \sum\limits_{\boldsymbol{j}\in 2\pi\mathbb{Z}^2} \hat{f}_{\boldsymbol{j}}(z) e^{i\boldsymbol{j}\cdot \boldsymbol{x}}, \\
     &&\hskip-.8in
    g(\boldsymbol{x},z) = \sum\limits_{\boldsymbol{k}\in 2\pi\mathbb{Z}^2} \hat{g}_{\boldsymbol{k}}(z) e^{i\boldsymbol{k}\cdot \boldsymbol{x}}, \\
    &&\hskip-.8in
    A^r e^{\tau A} H(\boldsymbol{x},z) = \sum\limits_{\boldsymbol{l}\in 2\pi\mathbb{Z}^2} |\boldsymbol{l}|^r e^{\tau |\boldsymbol{l}|}\hat{H}_{\boldsymbol{l}}(z) e^{i\boldsymbol{l}\cdot \boldsymbol{x}}.
    \end{eqnarray*}
Therefore,
\begin{eqnarray*}
&&\hskip-1in
I :=\Big|\Big\langle A^r e^{\tau A} (f\cdot \nabla  g), A^r e^{\tau A} h  \Big\rangle - \Big\langle  f\cdot \nabla  A^r e^{\tau A} g, A^r e^{\tau A} h  \Big\rangle\Big| \nonumber\\
&&\hskip-.9in
= \Big|\Big\langle  (f\cdot \nabla  g), A^r e^{\tau A} H  \Big\rangle - \Big\langle  f\cdot \nabla  A^r e^{\tau A} g, H  \Big\rangle\Big|\nonumber\\
&&\hskip-.9in
\leq \sum\limits_{\boldsymbol{j}+\boldsymbol{k}+\boldsymbol{l}=0} \int_0^1  |\hat{f}_{\boldsymbol{j}}(z)||\boldsymbol{k}||\hat{g}_{\boldsymbol{k}}(z)||\hat{H}_{\boldsymbol{l}}(z)| \Big| |\boldsymbol{l}|^r e^{\tau|\boldsymbol{l}|} - |\boldsymbol{k}|^r e^{\tau|\boldsymbol{k}|} \Big| dz.
\end{eqnarray*}
By virtue of the following observation \cite{LO97}:

For $r\geq 1$ and $\tau \geq 0$, and for all positive $\xi,\eta \in \mathbb{R}$, we have
\begin{eqnarray}\label{lemma-inequality}
&&\hskip-.8in
|\xi^r e^{\tau \xi} - \eta^r e^{\tau \eta}| \leq C_r|\xi - \eta|\Big( |\xi - \eta|^{r-1} + \eta^{r-1} + \tau(|\xi - \eta|^{r} + \eta^r)e^{\tau|\xi-\eta|}e^{\tau \eta} \Big);
\end{eqnarray}
with $\xi = |\boldsymbol{l}|$, $\eta = |\boldsymbol{k}|$, and $|\xi - \eta| \leq |\boldsymbol{j}|$, inequality \eqref{lemma-inequality} implies
\begin{equation*}
    I \leq C_r \sum\limits_{\boldsymbol{j}+\boldsymbol{k}+\boldsymbol{l}=0} \int_0^1  |\hat{f}_{\boldsymbol{j}}(z)||\boldsymbol{k}||\hat{g}_{\boldsymbol{k}}(z)||\hat{H}_{\boldsymbol{l}}(z)| |\boldsymbol{j}|\Big( |\boldsymbol{j}|^{r-1} + |\boldsymbol{k}|^{r-1} + \tau(|\boldsymbol{j}|^{r} + |\boldsymbol{k}|^r)e^{\tau|\boldsymbol{j}|}e^{\tau |\boldsymbol{k}|} \Big) dz.
\end{equation*}
By the definition of $H$, and since $e^x \leq 1+xe^x$ for any $x\geq 0$, we have
\begin{eqnarray*}
|\hat{H}_{\boldsymbol{l}}(z)| = |\boldsymbol{l}|^r e^{\tau |\boldsymbol{l}|} |\hat{h}_{\boldsymbol{l}}(z)| \leq |\boldsymbol{l}|^r(1+\tau |\boldsymbol{l}|e^{\tau |\boldsymbol{l}|}) |\hat{h}_{\boldsymbol{l}}(z)| \leq |\boldsymbol{l}|^r |\hat{h}_{\boldsymbol{l}}(z)| + \tau(|\boldsymbol{j}|+|\boldsymbol{k}|) |\hat{H}_{\boldsymbol{l}}(z)|.
\end{eqnarray*}
Therefore, one obtains that
\begin{eqnarray*}
&&\hskip-.68in
|\hat{H}_{\boldsymbol{l}}(z)| \Big( |\boldsymbol{j}|^{r-1} + |\boldsymbol{k}|^{r-1} + \tau(|\boldsymbol{j}|^{r} + |\boldsymbol{k}|^{r})e^{\tau|\boldsymbol{k}|}e^{\tau|\boldsymbol{j}|} \Big)\nonumber\\
&&\hskip-.8in
\leq \Big(|\boldsymbol{l}|^r |\hat{h}_{\boldsymbol{l}}(z)| + \tau(|\boldsymbol{j}|+|\boldsymbol{k}|) |\hat{H}_{\boldsymbol{l}}(z)|\Big) \Big( |\boldsymbol{k}|^{r-1} + |\boldsymbol{j}|^{r-1} \Big) + |\hat{H}_{\boldsymbol{l}}(z)| \Big(\tau(|\boldsymbol{k}|^{r} + |\boldsymbol{j}|^{r})e^{\tau|\boldsymbol{k}|}e^{\tau|\boldsymbol{j}|} \Big)\nonumber\\
&&\hskip-.8in
\leq |\hat{h}_{\boldsymbol{l}}(z)| |\boldsymbol{l}|^r(|\boldsymbol{k}|^{r-1} + |\boldsymbol{j}|^{r-1}) + \tau C_r |\hat{H}_{\boldsymbol{l}}(z)|(|\boldsymbol{k}|^{r} + |\boldsymbol{j}|^{r})e^{\tau|\boldsymbol{k}|}e^{\tau|\boldsymbol{j}|}.\label{Hl}
\end{eqnarray*}
Based on this, one has
\begin{eqnarray*}
&&\hskip-.8in
I \leq  C_r \sum\limits_{\boldsymbol{j}+\boldsymbol{k}+\boldsymbol{l}=0} \int_0^1  |\hat{f}_{\boldsymbol{j}}(z)||\boldsymbol{k}||\hat{g}_{\boldsymbol{k}}(z)| |\boldsymbol{j}||\hat{h}_{\boldsymbol{l}}(z)| |\boldsymbol{l}|^r(|\boldsymbol{k}|^{r-1} + |\boldsymbol{j}|^{r-1}) dz \nonumber\\
&&\hskip-.6in
+ \tau C_r \sum\limits_{\boldsymbol{j}+\boldsymbol{k}+\boldsymbol{l}=0} \int_0^1  |\hat{f}_{\boldsymbol{j}}(z)||\boldsymbol{k}||\hat{g}_{\boldsymbol{k}}(z)| |\boldsymbol{j}| |\hat{H}_{\boldsymbol{l}}(z)|(|\boldsymbol{k}|^{r} + |\boldsymbol{j}|^{r})e^{\tau|\boldsymbol{k}|}e^{\tau|\boldsymbol{j}|} dz:= I_{1} + I_{2}.
\end{eqnarray*}
Here
\begin{equation*}
    I_1 = C_r \sum\limits_{\boldsymbol{j}+\boldsymbol{k}+\boldsymbol{l}=0} \int_0^1 \Big( |\hat{f}_{\boldsymbol{j}}(z)||\boldsymbol{k}|^r|\hat{g}_{\boldsymbol{k}}(z)| |\boldsymbol{j}||\hat{h}_{\boldsymbol{l}}(z)| |\boldsymbol{l}|^r + |\hat{f}_{\boldsymbol{j}}(z)||\boldsymbol{k}||\hat{g}_{\boldsymbol{k}}(z)| |\boldsymbol{j}|^r|\hat{h}_{\boldsymbol{l}}(z)| |\boldsymbol{l}|^r \Big) dz := \int_0^1 I_{11} + I_{12} dz.
\end{equation*}

Thanks to Cauchy–Schwarz inequality, since $r> 2$, we have
\begin{eqnarray*}
&&\hskip-.8in
I_{11} = C_r \sum\limits_{\boldsymbol{j}+\boldsymbol{k}+\boldsymbol{l}=0}  |\boldsymbol{j}||\hat{f}_{\boldsymbol{j}}(z)||\boldsymbol{k}|^r|\hat{g}_{\boldsymbol{k}}(z)| |\boldsymbol{l}|^r|\hat{h}_{\boldsymbol{l}}(z)|  \nonumber \\
&&\hskip-.58in
= C_r \sum\limits_{\substack{\boldsymbol{j}\in 2\pi\mathbb{Z}^2 \\ \boldsymbol{j}\neq 0} } |\boldsymbol{j}| |\hat{f}_{\boldsymbol{j}}(z)| \sum\limits_{\substack{\boldsymbol{k}\in 2\pi\mathbb{Z}^2 \\ \boldsymbol{k}\neq 0, -\boldsymbol{j}} } |\boldsymbol{k}|^{r} |\hat{g}_{\boldsymbol{k}}(z)|  |\boldsymbol{j}+\boldsymbol{k}|^{r}|\hat{h}_{-\boldsymbol{j}-\boldsymbol{k}}(z)| \nonumber\\
&&\hskip-.58in
\leq C_r \Big( \sum\limits_{\substack{\boldsymbol{j}\in 2\pi\mathbb{Z}^2 \\ \boldsymbol{j}\neq 0} } |\boldsymbol{j}|^{2-2r}\Big)^{\frac{1}{2}} \Big( \sum\limits_{\substack{\boldsymbol{j}\in 2\pi\mathbb{Z}^2 \\ \boldsymbol{j}\neq 0} } |\boldsymbol{j}|^{2r}  |\hat{f}_{\boldsymbol{j}}(z)|^2\Big)^{\frac{1}{2}}  \nonumber \\
&&\hskip-.38in
\times \sup\limits_{\boldsymbol{j}\in 2\pi\mathbb{Z}^2}\Big( \sum\limits_{\substack{\boldsymbol{k}\in 2\pi\mathbb{Z}^2 \\ \boldsymbol{k}\neq 0, -\boldsymbol{j}} } |\boldsymbol{k}|^{2r} |\hat{g}_{\boldsymbol{k}}(z)|^2\Big)^{\frac{1}{2}} \Big( \sum\limits_{\substack{\boldsymbol{k}\in 2\pi\mathbb{Z}^2 \\ \boldsymbol{k}\neq 0, -\boldsymbol{j}} } |\boldsymbol{j}+\boldsymbol{k}|^{2r} |\hat{h}_{-\boldsymbol{j}-\boldsymbol{k}}(z)|^2\Big)^{\frac{1}{2}} \nonumber\\
&&\hskip-.58in \leq C_r \|A^r f(z)\|_{L^2(\mathbb{T}^2)} \|A^r g(z)\|_{L^2(\mathbb{T}^2)} \|A^r h(z)\|_{L^2(\mathbb{T}^2)}.
\end{eqnarray*}
Similarly, one gets
\begin{equation*}
    I_{12} \leq C_r \|A^r f(z)\|_{L^2(\mathbb{T}^2)} \|A^r g(z)\|_{L^2(\mathbb{T}^2)} \|A^r h(z)\|_{L^2(\mathbb{T}^2)}.
\end{equation*}
Therefore,
\begin{equation*}
    I_1 \leq C_r \int_0^1 \|A^r f(z)\|_{L^2(\mathbb{T}^2)} \|A^r g(z)\|_{L^2(\mathbb{T}^2)} \|A^r h(z)\|_{L^2(\mathbb{T}^2)} dz.
\end{equation*}

Next, we estimate
\begin{equation*}
\begin{split}
   I_2 = \tau C_r  \sum\limits_{\boldsymbol{j}+\boldsymbol{k}+\boldsymbol{l}=0} \int_0^1 &\Big( |\boldsymbol{j}|^{r+1}e^{\tau|\boldsymbol{j}|}|\hat{f}_{\boldsymbol{j}}(z)||\boldsymbol{k}|e^{\tau|\boldsymbol{k}|}|\hat{g}_{\boldsymbol{k}}(z)|  |\hat{H}_{\boldsymbol{l}}(z)| \\
   &+ |\boldsymbol{j}|e^{\tau|\boldsymbol{j}|}|\hat{f}_{\boldsymbol{j}}(z)||\boldsymbol{k}|^{r+1}e^{\tau|\boldsymbol{k}|}|\hat{g}_{\boldsymbol{k}}(z)|  |\hat{H}_{\boldsymbol{l}}(z)| \Big) dz := \int_0^1 I_{21} + I_{22} dz.
\end{split}
\end{equation*}
Thanks to Cauchy–Schwarz inequality, since $r> 2$, by using $|\boldsymbol{j}|^{\frac{1}{2}} \leq |\boldsymbol{k}|^{\frac{1}{2}} + |\boldsymbol{l}|^{\frac{1}{2}}$, and $|\boldsymbol{k}|^{\frac{1}{2}} + |\boldsymbol{l}|^{\frac{1}{2}} \leq 2 |\boldsymbol{k}|^{\frac{1}{2}} |\boldsymbol{l}|^{\frac{1}{2}}$ when $|\boldsymbol{k}|\geq 1$ and $|\boldsymbol{l}|\geq 1$,
we have
\begin{eqnarray*}
&&\hskip-.8in
I_{21} = \tau C_r \sum\limits_{\boldsymbol{j}+\boldsymbol{k}+\boldsymbol{l}=0} |\boldsymbol{j}|^{r+1}e^{\tau|\boldsymbol{j}|}|\hat{f}_{\boldsymbol{j}}(z)||\boldsymbol{k}|e^{\tau|\boldsymbol{k}|}|\hat{g}_{\boldsymbol{k}}(z)|  |\hat{H}_{\boldsymbol{l}}(z)| \nonumber \\
&&\hskip-.58in
\leq \tau C_r \sum\limits_{\substack{\boldsymbol{j}+\boldsymbol{k}+\boldsymbol{l}=0\\ \boldsymbol{j}, \boldsymbol{k}, \boldsymbol{l} \neq 0  }} |\boldsymbol{j}|^{r+\frac{1}{2}}e^{\tau|\boldsymbol{j}|}|\hat{f}_{\boldsymbol{j}}(z)||\boldsymbol{k}|^{\frac{3}{2}}e^{\tau|\boldsymbol{k}|}|\hat{g}_{\boldsymbol{k}}(z)|  |\boldsymbol{l}|^{r+\frac{1}{2}}e^{\tau|\boldsymbol{l}|}|\hat{h}_{\boldsymbol{l}}(z)| \nonumber\\
&&\hskip-.58in
\leq C_r \tau \sum\limits_{\substack{\boldsymbol{k}\in 2\pi\mathbb{Z}^2 \\ \boldsymbol{k}\neq 0} } |\boldsymbol{k}|^{\frac{3}{2}} |\hat{g}_{\boldsymbol{k}}(z)| e^{\tau |\boldsymbol{k}|} \sum\limits_{\substack{\boldsymbol{j}\in 2\pi\mathbb{Z}^2 \\ \boldsymbol{j}\neq 0, -\boldsymbol{k}} } |\boldsymbol{j}|^{r+\frac{1}{2}} e^{\tau |\boldsymbol{j}|}|\hat{f}_{\boldsymbol{j}}(z)|  |\boldsymbol{j}+\boldsymbol{k}|^{r+\frac{1}{2}}e^{\tau |\boldsymbol{j}+\boldsymbol{k}|}|\hat{h}_{-\boldsymbol{j}-\boldsymbol{k}}(z)| \nonumber\\
&&\hskip-.58in
\leq C_r \tau \Big( \sum\limits_{\substack{\boldsymbol{k}\in 2\pi\mathbb{Z}^2 \\ \boldsymbol{k}\neq 0} } |\boldsymbol{k}|^{2-2r}\Big)^{\frac{1}{2}} \Big( \sum\limits_{\substack{\boldsymbol{k}\in 2\pi\mathbb{Z}^2 \\ \boldsymbol{k}\neq 0} } |\boldsymbol{k}|^{2r+1} e^{2\tau |\boldsymbol{k}|} |\hat{g}_{\boldsymbol{k}}(z)|^2\Big)^{\frac{1}{2}}  \nonumber \\
&&\hskip-.38in
\times \sup\limits_{\boldsymbol{k}\in 2\pi\mathbb{Z}^2}\Big( \sum\limits_{\substack{\boldsymbol{j}\in 2\pi\mathbb{Z}^2 \\ \boldsymbol{j}\neq 0, -\boldsymbol{k}} } |\boldsymbol{j}|^{2r+1}e^{2\tau |\boldsymbol{j}|} |\hat{f}_{\boldsymbol{j}}(z)|^2\Big)^{\frac{1}{2}} \Big( \sum\limits_{\substack{\boldsymbol{j}\in 2\pi\mathbb{Z}^2 \\ \boldsymbol{j}\neq 0, -\boldsymbol{k}} } |\boldsymbol{j}+\boldsymbol{k}|^{2r+1}e^{2\tau |\boldsymbol{j}+\boldsymbol{k}|} |\hat{h}_{-\boldsymbol{j}-\boldsymbol{k}}(z)|^2\Big)^{\frac{1}{2}} \nonumber\\
&&\hskip-.58in \leq C_r \tau\|A^{r+\frac{1}{2}} e^{\tau A} f(z)\|_{L^2(\mathbb{T}^2)} \|A^{r+\frac{1}{2}} e^{\tau A} g(z)\|_{L^2(\mathbb{T}^2)} \|A^{r+\frac{1}{2}} e^{\tau A} h(z)\|_{L^2(\mathbb{T}^2)}.
\end{eqnarray*}
Similarly, one gets
\begin{equation*}
    I_{22} \leq C_r \tau\|A^{r+\frac{1}{2}} e^{\tau A} f(z)\|_{L^2(\mathbb{T}^2)} \|A^{r+\frac{1}{2}} e^{\tau A} g(z)\|_{L^2(\mathbb{T}^2)} \|A^{r+\frac{1}{2}} e^{\tau A} h(z)\|_{L^2(\mathbb{T}^2)}.
\end{equation*}
Therefore,
\begin{equation*}
    I_2 \leq  C_r \tau \int_0^1 \|A^{r+\frac{1}{2}} e^{\tau A} f(z)\|_{L^2(\mathbb{T}^2)} \|A^{r+\frac{1}{2}} e^{\tau A} g(z)\|_{L^2(\mathbb{T}^2)} \|A^{r+\frac{1}{2}} e^{\tau A} h(z)\|_{L^2(\mathbb{T}^2)} dz.
\end{equation*}

\end{proof}

\begin{lemma} \label{lemma-difference-type4}
For  $f, g, \partial_z g, h\in \mathcal{S}_{r+\frac{1}{2},s,\tau}$, where $r > 2 $, $s\geq 0$, and $\tau\geq 0$, one has
\begin{eqnarray*}
&&\hskip-.8in
\Big|\Big\langle A^r e^{\tau A} \Big( (\int_0^z \nabla\cdot f(\boldsymbol{x},s)ds) \partial_z g  \Big), A^r e^{\tau A} h  \Big\rangle -  \Big\langle \partial_z g A^r e^{\tau A} (\int_0^z \nabla\cdot f(\boldsymbol{x},s)ds)  , A^r e^{\tau A} h  \Big\rangle\Big| \\
&&\hskip-.9in
\leq C_r \|A^{r} \partial_z g\| \|A^{r}  f\| \|A^{r}  h\| + C_r \tau \|A^{r+\frac{1}{2}} e^{\tau A} \partial_z g\| \|A^{r+\frac{1}{2}} e^{\tau A} f\| \|A^{r+\frac{1}{2}} e^{\tau A} h\|.
\end{eqnarray*}
\end{lemma}
\begin{proof}
Observe that Lemma \ref{lemma-difference-type4} follows directly from Lemma \ref{lemma-difference-type2}. Indeed, if one replaces $f$ by $\int_0^z f(\boldsymbol{x}, s) ds$ and $g$ by $\partial_z g$ in Lemma \ref{lemma-difference-type2}, by the H\"older inequality, one obtains that
\begin{equation*}
\begin{split}
   &\Big|\Big\langle A^r e^{\tau A} \Big( (\int_0^z \nabla\cdot f(\boldsymbol{x},s)ds) \partial_z g  \Big), A^r e^{\tau A} h  \Big\rangle -  \Big\langle \partial_z g A^r e^{\tau A} (\int_0^z \nabla\cdot f(\boldsymbol{x},s)ds)  , A^r e^{\tau A} h  \Big\rangle\Big|
   \\
   \leq &C_r \int_0^1 \|A^r \int_0^z f(\boldsymbol{x}, s) ds\|_{L^2(\mathbb{T}^2)} \|A^r \partial_z g(z)\|_{L^2(\mathbb{T}^2)} \|A^r h(z)\|_{L^2(\mathbb{T}^2)} dz
   \\
   &+ C_r \tau \int_0^1 \|A^{r+\frac{1}{2}} e^{\tau A} \int_0^z f(\boldsymbol{x}, s) ds\|_{L^2(\mathbb{T}^2)} \|A^{r+\frac{1}{2}} e^{\tau A} \partial_z g(z)\|_{L^2(\mathbb{T}^2)} \|A^{r+\frac{1}{2}} e^{\tau A} h(z)\|_{L^2(\mathbb{T}^2)} dz
   \\
   \leq & C_r \|A^{r} \partial_z g\| \|A^{r}  f\| \|A^{r}  h\| + C_r \tau \|A^{r+\frac{1}{2}} e^{\tau A} \partial_z g\| \|A^{r+\frac{1}{2}} e^{\tau A} f\| \|A^{r+\frac{1}{2}} e^{\tau A} h\|.
\end{split}
\end{equation*}
\end{proof}

\noindent
\section*{Acknowledgments}
X.L. acknowledges the partial funding by the Deutsche Forschungsgemeinschaft (DFG) through project AA2--9 {\it ``Variational Methods for Viscoelastic Flows and Gelation"} within MATH+. X.L. and E.S.T. would like to thank the Isaac Newton Institute for Mathematical Sciences, Cambridge, for support and hospitality during the programme {\it``Mathematical aspects of turbulence: where do we stand?"} where part of the work on this paper was undertaken. This work was supported in part by EPSRC grant no EP/R014604/1. X.L.’s work was partially supported by a grant from the Simons Foundation, during his visit to the Isaac Newton Institute for Mathematical Sciences.

\section*{Conflict of interest}
On behalf of all authors, the corresponding author states that there is no conflict of interest.

\section*{Data availability statement}
Our manuscript has no associated data.

\end{document}